\documentclass[11pt,a4paper]{amsart}
\usepackage[margin=1in]{geometry}
\usepackage[T1]{fontenc}
\usepackage{lmodern}
\usepackage{amssymb,mathtools,booktabs,array,microtype}
\usepackage[numbers,sort]{natbib}
\let\cite\citep
\usepackage{hyperref}
\usepackage{needspace}
\hypersetup{colorlinks=true,citecolor=blue,linkcolor=black,urlcolor=black,
  pdftitle={Sharp Projection Bounds and Direction Dependence for Intermediate Dimensions of Measures},
  pdfauthor={Shengyang Zang; Yingyu Zhao; Yunhao Zhao; Chen Zhou},
  pdfsubject={Intermediate dimensions and projections of measures}}
\allowdisplaybreaks[2]

\makeatletter
\let\originalsetaddresses\@setaddresses
\renewcommand{\@setaddresses}{
  \par\begingroup
  \leftskip=\normalparindent
  \parindent=0pt
  \originalsetaddresses
  \endgroup
}
\makeatother

\newtheorem{theorem}{Theorem}[section]
\newtheorem{proposition}[theorem]{Proposition}
\newtheorem{lemma}[theorem]{Lemma}
\newtheorem{corollary}[theorem]{Corollary}
\theoremstyle{definition}
\newtheorem{definition}[theorem]{Definition}
\theoremstyle{remark}
\newtheorem*{remark}{Remark}
\DeclareMathOperator{\spt}{supp}

\newcommand{\N}{\mathbb N}
\newcommand{\R}{\mathbb R}
\newcommand{\G}{G(n,m)}
\newcommand{\qA}{\dim_{\mathrm{qA}}}
\newcommand{\dA}{\dim_{\mathrm A}}

\numberwithin{equation}{section}
\title[Projection bounds for intermediate dimensions]
{Sharp Projection Bounds\\
for Intermediate Dimensions of Measures}
\author{Shengyang Zang}
\address{Shengyang Zang, School of Mathematical Sciences,
Nankai University, Tianjin, China}
\email{1120230032@mail.nankai.edu.cn}
\author{Yingyu Zhao}
\address{Yingyu Zhao, Qiuzhen College, Tsinghua University, Beijing, China}
\email{zhao-yy25@mails.tsinghua.edu.cn}
\author{Yunhao Zhao}
\address{Yunhao Zhao, School of Advanced Interdisciplinary Sciences,
University of Chinese Academy of Sciences, Beijing, China}
\email{zhaoyunhao23@mails.ucas.ac.cn}
\author{Chen Zhou}
\address{Chen Zhou, School of Mathematical Sciences,
Zhejiang University of Technology, Hangzhou, China}
\email{zcygo01@163.com}
\date{}

\begin{document}
\raggedbottom
\begin{abstract}
We prove sharp bounds for the upper and lower intermediate dimensions
of orthogonal and radial projections of compactly supported measures.
Both dimension functions are preserved almost everywhere when the
quasi-Assouad dimension does not exceed the target dimension.
We obtain exceptional-set estimates using the entire source dimension
function, with sharpness for explicit families.
We also construct a fixed measure whose projected dimension functions
vary on open direction sets, even when all source dimensions lie below
the target dimension.
\end{abstract}
\maketitle

\section{Introduction}\label{sec:introduction}

Intermediate dimensions of measures, introduced by Angelini, Molter
and Tejada~\cite{AMT}, quantify uniform lower mass bounds over a range
of scales. At a small scale $r$, every point of the support must admit a ball
with radius between $r$ and $r^\theta$ whose mass is bounded below
by a fixed power of its radius. At $\theta=1$,
the upper and lower versions recover the Minkowski dimensions of the
measure. These dimensions are sensitive to points carrying little mass,
even on a geometrically simple support. For example, the probability
measure with density $\beta x^{\beta-1}$ on $[0,1]$, with $\beta>1$,
has both intermediate dimensions equal to $\beta$ for every positive
parameter. In particular, these intermediate dimensions can exceed the
ambient dimension.

Projection combines mass along fibres and can eliminate the small-mass
behaviour that determines a source dimension. Understanding the resulting
change therefore requires information about the mass distribution.
Classical Hausdorff projection theorems determine almost-everywhere
projected dimensions and bound exceptional directions
\cite{Kaufman1968,Mattila2015}. For intermediate dimensions of compact
sets, Burrell, Falconer and Fraser~\cite{BFF} obtained
almost-everywhere formulas in terms of capacities; the resulting values
are independent of the direction outside a null set. For measures,
this leads to three related questions: under what conditions are the
entire dimension functions preserved, how large can the sets of
dimension-dropping directions be, and must the projected dimension
functions of one fixed measure be constant almost everywhere?

Let $\mu$ be a nonzero finite Borel measure with compact support
$K\subset\R^n$, and let $1\le m<n$. For $V\in G(n,m)$, write
$\mu_V=(P_V)_\#\mu$, where $P_V$ is orthogonal projection onto $V$,
and let $\gamma_{n,m}$ be Haar probability measure on $G(n,m)$.
Set $q=\dim_{\mathrm{qA}}\mu$. We use the conventions
$(x)_+=\max\{x,0\}$, $1/\infty=0$ and $1/0=\infty$.
Both intermediate dimension functions are nondecreasing in $\theta$,
so their limits as $\theta\to0^+$ exist. The definitions and basic
comparisons are recalled in Section~\ref{sec:definitions}. We state
the results below for upper intermediate dimensions and indicate their
lower-dimensional counterparts.

\subsection{Projection bounds and preservation}

The quasi-Assouad dimension measures relative mass growth between
sufficiently separated scales. Our first theorem shows how this growth
controls the loss under projection through an additive correction to
the reciprocal dimension. The estimate holds on one set of directions
for both dimension functions and every positive parameter.

\begin{theorem}\label{thm:orthogonal-main}
There is a set $\mathcal V\subset G(n,m)$ of full
$\gamma_{n,m}$-measure such that
\begin{equation}\label{eq:intro-quantitative}
 \left[\frac1{\overline{\dim}_\theta\mu}
       +\left(\frac1m-\frac1q\right)_+\right]^{-1}
 \le\overline{\dim}_\theta\mu_V\le\overline{\dim}_\theta\mu
\end{equation}
for every $V\in\mathcal V$ and $0<\theta\le1$.
The same bounds hold for lower intermediate dimensions on the same set.
\end{theorem}

When $q\le m$, the correction term vanishes and the two bounds
coincide, giving simultaneous preservation of both dimension functions.

\begin{corollary}\label{cor:intro-preservation}
If $\dim_{\mathrm{qA}}\mu\le m$, there is a set
$\mathcal V\subset G(n,m)$ of full $\gamma_{n,m}$-measure such that
\[
 \overline{\dim}_\theta\mu_V=\overline{\dim}_\theta\mu,
 \quad V\in\mathcal V,\ 0<\theta\le1.
\]
The same equality holds for lower intermediate dimensions on the same set.
\end{corollary}

Corollary~\ref{cor:intro-preservation} replaces the assumption
$\dim_{\mathrm A}\mu<m$ in
\cite[Theorem~5.13 and Remark~5.14]{AMT} by
$\dim_{\mathrm{qA}}\mu\le m$. It includes equality at the growth
threshold and answers the question in \cite[Remark~5.16]{AMT}.
The criterion also applies to measures with infinite Assouad dimension.
For related projection results for sets involving the quasi-Assouad
dimension, see \cite{FFS,AM}.

The growth threshold cannot be relaxed, even when every source
intermediate dimension is below $m$. Indeed, for every $m<q\le\infty$
and $0<a\le\infty$, Theorem~\ref{thm:sharp-clouds} constructs a
measure with countable compact support in $\R^{m+1}$,
quasi-Assouad dimension $q$ and infinite Assouad dimension, such that
\begin{equation}\label{eq:intro-sharpness}
 \overline{\dim}_\theta\mu
 =\left(\frac1q+\frac1{a\theta}\right)^{-1},\qquad
 \overline{\dim}_\theta\mu_V
 =\left(\frac1m+\frac1{a\theta}\right)^{-1},
 \quad0<\theta\le1,
\end{equation}
on one set of full $\gamma_{m+1,m}$-measure, with the same formulas
for lower dimensions. For sufficiently small $a$, all source dimensions
are below $m$, but strict loss still occurs at every positive parameter.

A complementary preservation criterion comes from the source
dimension function as $\theta\to0^+$. To use this information,
Theorem~\ref{thm:transversal} treats the source parameter $\kappa$
and the image parameter $\theta$ independently. Letting
$\kappa\to0^+$ gives the following orthogonal form of
Corollary~\ref{cor:limits}, without any restriction on $q$.

\begin{corollary}\label{cor:intro-limits}
Without any growth assumption on $\mu$, there is a set
$\mathcal V\subset G(n,m)$ of full $\gamma_{n,m}$-measure such that
\begin{equation}\label{eq:intro-limits}
 \min\left\{m,\lim_{\theta\to0^+}\overline{\dim}_\theta\mu\right\}
 \le\lim_{\theta\to0^+}\overline{\dim}_\theta\mu_V
 \le\lim_{\theta\to0^+}\overline{\dim}_\theta\mu
\end{equation}
for every $V\in\mathcal V$. In particular, if
$\overline{\dim}_\theta\mu=h\le m$ for all $0<\theta\le1$, then
$\overline{\dim}_\theta\mu_V=h$ for every $V\in\mathcal V$ and
$0<\theta\le1$. The same assertions hold for lower intermediate
dimensions on the same set.
\end{corollary}

This removes the finite Assouad dimension hypothesis in
\cite[Theorem~6.2 and Remark~6.3]{AMT}. In particular, a constant
source dimension function with value at most $m$ is preserved even
when the quasi-Assouad dimension is infinite. The entire source
function can therefore force preservation when the growth criterion
alone gives only a lower bound.

\subsection{Exceptional directions}

We next quantify the size of the set of directions in which dimension
decreases. The two-parameter estimate makes the entire source function
relevant: a lower bound for the image dimension at $\theta$ may follow
from source information at a smaller parameter. Optimizing this choice
leads to the exponent below. Fix $0<\theta\le1$ and finite
$0<d\le s=\overline{\dim}_\theta\mu$, and put
\[
 \kappa_d=\inf\{\kappa\in(0,1]:\overline{\dim}_\kappa\mu\ge d\},
 \qquad
 b_*=\left[\frac1d-\frac{\kappa_d}{\theta}
                    \left(\frac1d-\frac1q\right)\right]^{-1}.
\]
Thus $\kappa_d$ records how early the source function reaches the
threshold $d$; it is zero if the threshold is reached arbitrarily near
zero. Theorem~\ref{thm:exception} gives the following Hausdorff-null
estimate at the critical exponent.

\begin{theorem}\label{thm:intro-exception}
If $b_*\le m$, then
\begin{equation}\label{eq:intro-exception}
 \mathcal H^{m(n-m-1)+b_*}
 \{V\in G(n,m):\overline{\dim}_\theta\mu_V<d\}=0.
\end{equation}
The same holds for lower intermediate dimensions, with $s$ and
$\kappa_d$ defined using the lower source dimension function.
\end{theorem}

Using only the source value at $\theta$ gives the exponent
\[
 b=\left(\frac1d-\frac1s+\frac1q\right)^{-1},
 \qquad d\le b_*\le b\le q.
\]
The inequality $b_*\le b$ can be strict even when the improved
exponent is sharp for a fixed entire source function.
Corollary~\ref{cor:nonharmonic-sharp} constructs such a function
with $q=2$: at $\theta=1$, the values $s=2/3$ and $d=1/2$
give $b=1$ and $b_*=2/3$.

For the functions in \eqref{eq:intro-sharpness}, the two exponents
coincide. When $b\le m$, Theorem~\ref{thm:finite-graded-sharp}
proves sharpness with $q$ and both entire source functions fixed:
for every $\varepsilon>0$, the dimension of a strict exceptional set
exceeds $m(n-m-1)+b-\varepsilon$. If $d<s$, a non-strict sublevel
at $d$ has dimension exactly $m(n-m-1)+b$, and the set where both
image dimensions equal $d$ has positive Hausdorff measure at this
exponent. The strict inequality in \eqref{eq:intro-exception} is
therefore essential for the Hausdorff-null conclusion.
Corollary~\ref{cor:capped-constant-sharp} gives corresponding sharpness
examples with constant source functions below a prescribed growth
dimension $q$.

\subsection{Radial projections}

For radial projections, the centre replaces the direction as the
parameter. Given $K\subset\R^{m+1}$ and $z\notin K$, write
$\pi_z(x)=(x-z)/|x-z|$ and $\mu_z=(\pi_z)_\#\mu$.
A geometric transversality estimate allows the same potential argument
to apply to this family; see \cite{Orp} for radial projection theory
and \cite{BT} for transversal nonlinear families.
Theorem~\ref{thm:radial} gives the radial analogues of
\eqref{eq:intro-quantitative} and \eqref{eq:intro-limits} for
Lebesgue-almost every exterior centre, on one set common to both
dimension functions and all positive parameters. The corresponding
exceptional-set estimate is
\[
 \mathcal H^{1+b_*}
 \{z\notin K:\overline{\dim}_\theta\mu_z<d\}=0
 \qquad\text{if }b_*\le m,
\]
with the same conclusion for lower dimensions.
For prescribed source functions of the form
\eqref{eq:intro-sharpness}, Theorem~\ref{thm:radial-stable-sharp}
proves sharpness in specified parameter ranges.
Proposition~\ref{prop:radial-constant-reference} extends the available
ranges for constant source functions using measures of finite Assouad
dimension. The precise conditions, including their boundary cases,
are given in Section~\ref{sec:third-round}.

\subsection{Dependence on the direction}

The bounds above leave a further question: how much can the projected
dimension functions vary with the direction when the measure itself is
fixed? The almost-everywhere constancy known for compact sets suggests
one possible answer, but the following result shows a different behaviour
for measures. Proposition~\ref{prop:nonconstant-directions}(i) realizes
the two extremal functions allowed by our bounds on different nonempty
open sets of directions.

\begin{proposition}\label{prop:intro-direction}
For every $m<q\le\infty$ and $0<a\le\infty$, there is a probability
measure $\rho$ with countable compact support in $\R^n$ such that
$\dim_{\mathrm{qA}}\rho=q$, $\dim_{\mathrm A}\rho=\infty$, and
\[
 \overline{\dim}_\theta\rho
 =\left(\frac1q+\frac1{a\theta}\right)^{-1},
 \qquad0<\theta\le1.
\]
There are disjoint nonempty open sets
$\mathcal U_-,\mathcal U_+\subset G(n,m)$ such that, for every
$0<\theta\le1$,
\[
 \overline{\dim}_\theta\rho_V=
 \begin{cases}
 \displaystyle\left(\frac1m+\frac1{a\theta}\right)^{-1},
       &V\in\mathcal U_-,\\[5pt]
 \displaystyle\left(\frac1q+\frac1{a\theta}\right)^{-1},
       &V\in\mathcal U_+.
 \end{cases}
\]
The same formulas hold for lower intermediate dimensions.
\end{proposition}

Every nonempty open set of directions has positive Haar measure.
Thus, for this one measure, both bounds of
Theorem~\ref{thm:orthogonal-main} occur on sets of positive measure,
and the projected dimension functions are not almost-everywhere
constant. Taking $a$ sufficiently small places all source dimensions
below $m$. Consequently, a direction-independent capacity cannot
give a single almost-everywhere projected value for arbitrary fixed
measures, even under this restriction.

\subsection{Methods and organization}

The analytic argument begins with a point at which the mass is small
throughout a source scale interval. We estimate its potential to control
the mass that projection can bring in from larger distances. The
quasi-Assouad bound controls this contribution when $q<\infty$;
the total mass suffices when $q=\infty$. Keeping the source and image
parameters independent allows the entire source dimension function to
enter the estimate. Averaging over the projection family gives the
almost-everywhere bounds. Averaging against Frostman measures then
gives the exceptional-set bounds at the critical Hausdorff exponent.

The sharpness constructions use relative sampling \cite{LLS,HPS}
to replace reference measures by finite samples placed in separated
layers. A common sample controls source balls and ball preimages under
all orthogonal and exterior radial projections. We can therefore
prescribe the source dimension functions and growth while keeping
control of every projected layer. Exceptional directions are obtained
by making points with small mass project into a regular component;
weighted polygon boundaries allow the projected dimensions to vary
across open sets of directions.

The constructions also give exact exceptional dimensions and prescribed
direction distributions. Proposition~\ref{prop:exact-strict-lower} attains the strict exceptional
dimension with one measure whose lower source dimension function is
constant, simultaneously at every positive parameter and in the
zero-parameter limit. In the plane,
Proposition~\ref{prop:nonconstant-directions}(ii) assigns finitely many
functions $\theta\mapsto(1/p_i+1/(a\theta))^{-1}$, with
$1\le p_i\le q$, to open direction sets with prescribed positive Haar
measures summing to one, and gives corresponding radial formulas.

Section~\ref{sec:definitions} recalls the dimensions of measures.
Sections~\ref{sec:potentials}--\ref{sec:exception} establish the
potential estimate, the projection bounds and the exceptional-set
bounds. Sections~\ref{sec:layered} and~\ref{sec:sharp} develop the
sampling construction and prove sharpness of the projection bounds
simultaneously at all positive parameters.
Section~\ref{sec:third-round} treats sharp exceptional exponents, and
Section~\ref{sec:examples} constructs measures with nonconstant
projected dimension functions. Appendix~\ref{sec:capacity} gives
sharp point-capacity comparisons and computes the associated critical
exponents for the examples.

\section{Dimensions of measures}\label{sec:definitions}

We recall the mass conditions defining the dimensions used in the
projection estimates. Throughout, $\mu$ is a nonzero finite Borel measure
with compact support $K\subset\R^n$, and $B(x,r)$ is a closed ball.
For a Borel map $f:K\to\R^N$, its pushforward is the measure
\[
 (f_\#\mu)(A)=\mu(f^{-1}(A)),\quad A\subset\R^N\text{ Borel}.
\]
Constants in $\lesssim$ may depend on fixed parameters but not on
the scales or centres under consideration. Multiplying a measure
by a positive constant leaves all the dimensions below unchanged.

\begin{definition}\label{def:intermediate}
Let $0<\theta\le1$. The $\theta$-intermediate dimensions of $\mu$ are defined by
\begin{align*}
\overline{\dim}_{\theta}\mu := \inf\big\{s\ge0:\ &\exists c,r_0>0,\ \forall 0<r<r_0,\ \forall x\in K,\
\exists u\in[r,r^\theta]\text{ with }\mu(B(x,u))\ge cu^s\big\},\\
\underline{\dim}_{\theta}\mu := \inf\big\{s\ge0:\ &\exists c>0,\ r_n\to0^+,\ \forall n\in\N,\ \forall x\in K,\
\exists u\in[r_n,r_n^\theta]\text{ with }\mu(B(x,u))\ge cu^s\big\}.
\end{align*}
In both definitions, $c$ is independent of the scale and the centre.
At $\theta=1$ these are the upper and lower Minkowski dimensions of $\mu$.
\end{definition}

\begin{definition}\label{def:qA}
Given $\delta\ge0$, set
\begin{align*}
H(\delta):=\inf\Big\{s\ge0:\ &\exists C,R_*>0,\ \forall 0<u\le R^{1+\delta}\le R\le R_*,\
\sup_{x\in K}\frac{\mu(B(x,R))}{\mu(B(x,u))}\le C\left(\frac Ru\right)^s\Big\}.
\end{align*}
The constants $C,R_*$ are independent of $x,u,R$.
The Assouad and quasi-Assouad dimensions of $\mu$ are, respectively,
\[
 \dim_{\mathrm A}\mu=H(0),\qquad
 \dim_{\mathrm{qA}}\mu=\lim_{\delta\downarrow0}H(\delta).
\]
\end{definition}

The following comparisons describe the dependence on the scale
parameter and relate the intermediate dimensions to mass growth.
For upper dimensions, \eqref{eq:window} is
\cite[Proposition~3.7 and Corollary~3.10]{AMT}; the lower-dimensional
case follows by applying the same estimates along a sequence of scales.
The last inequality in \eqref{eq:dimchain} follows from
Definition~\ref{def:qA} with a fixed outer radius.

\begin{lemma}\label{lem:window}
Let $D_\theta$ denote either $\underline{\dim}_\theta\mu$ or
$\overline{\dim}_\theta\mu$. Then
\begin{equation}\label{eq:window}
 D_\theta\le D_\varphi\le\frac\varphi\theta D_\theta,
 \qquad 0<\theta\le\varphi\le1.
\end{equation}
In particular, each dimension function is either finite and continuous
on $(0,1]$, or identically infinite there. Moreover,
\begin{equation}\label{eq:dimchain}
 0\le\underline{\dim}_\theta\mu\le\overline{\dim}_\theta\mu
 \le\overline{\dim}_1\mu\le\dim_{\mathrm{qA}}\mu.
\end{equation}
\end{lemma}

\section{The two-scale potential estimate}
\label{sec:potentials}

Small mass near a source point need not give small mass near its image,
since distant points may project nearby. We control this contribution
using the kernels of Burrell, Falconer and Fraser~\cite{BFF}, allowing
a real decay exponent. The full kernel will bound the parameter average
of the truncated kernel after projection.

\begin{definition}\label{def:kernels}
Let $\theta\in(0,1]$, $m>0$, $s\in[0,m]$ and $r\in(0,1)$.
Define $\phi_{r,\theta}^{s,m},\widetilde\phi_{r,\theta}^s:\R^n\to\R$ by
\begin{equation}\label{eq:originalkernels}
\begin{gathered}
\phi_{r,\theta}^{s,m}(x)=\begin{cases}
1,&|x|\le r,\\
r^s/|x|^s,&r<|x|\le r^\theta,\\
r^{s+\theta(m-s)}/|x|^m,&|x|>r^\theta,
\end{cases}\\[3pt]
\widetilde\phi_{r,\theta}^s(x)=\begin{cases}
1,&|x|\le r,\\
r^s/|x|^s,&r<|x|\le r^\theta,\\
0,&|x|>r^\theta.
\end{cases}
\end{gathered}
\end{equation}
and define the point capacity of $\mu$ by
\begin{equation}\label{eq:originalcapacity}
C_{r,\theta}^{s,m}(\mu)=\left(\inf_{x\in K}\int\phi_{r,\theta}^{s,m}(x-y)\,d\mu(y)\right)^{-1}.
\end{equation}
The point potential is continuous on the compact set $K$, so its infimum is attained.
\end{definition}

By Definition~\ref{def:intermediate}, $t<\overline{\dim}_\kappa\mu$
provides a point with small mass throughout the source scale interval
at arbitrarily small scales; $t<\underline{\dim}_\kappa\mu$ provides
such a point at every sufficiently small scale. The next lemma estimates
the potential at this point while keeping the source and image parameters
independent. This distinction will yield the limiting theorem in
Corollary~\ref{cor:limits} without a growth assumption. The point capacity
itself is used only in Appendix~\ref{sec:capacity}.

\begin{lemma}\label{lem:potential}
Fix $\alpha>0$ and $\theta,\kappa\in(0,1]$. Take either
$Q>\max\{\dim_{\mathrm{qA}}\mu,\alpha\}$ finite, or $Q=\infty$. For $0<t<Q$, let
\begin{equation}\label{eq:scale-choice}
 b=\frac t\alpha+\frac\kappa\theta\left(1-\frac tQ\right),
 \quad 0<s<\min\{t,t/b\}.
\end{equation}
For finite $Q$, choose $0<\varphi<\kappa$; for $Q=\infty$,
take $\varphi=\kappa$. There exist $\beta,C,\delta_0>0$ such that,
whenever $0<\delta<\delta_0$ and $x\in K$ satisfy
\begin{equation}\label{small-mass}
    \mu(B(x,u))\le u^t,\quad u\in[\delta,\delta^\varphi],
\end{equation}
one has
\begin{equation}\label{eq:potential-bound}
 \int\phi_{r,\theta}^{s,\alpha}(x-y)\,d\mu(y)
 \le C r^{s+\beta},\quad r=\delta^b.
\end{equation}
\end{lemma}

\begin{proof}
We introduce the simple tail potential
\[
J_u(x)=\int\min\{1,(u/|x-y|)^\alpha\}\,d\mu(y):=\int
k_u(x-y)\,d\mu(y),
\]
where the integrand is $1$ at $y=x$. 

We first prove
\begin{equation}\label{eq:simple-potential}
J_u(x)\lesssim\delta^t+u^t+
 u^\alpha\delta^{-\kappa\alpha(1-t/Q)},
\quad 0<u<1.
\end{equation}
Suppose first that $Q<\infty$.
Since $Q>\dim_{\mathrm{qA}}\mu$, we may choose
\[
0<\zeta<\frac{\kappa}{\varphi}-1
\]
and constants $C,R_*>0$ such that
\begin{equation}\label{growth-est}
    \mu(B(x,R))\le C\left(\frac R\rho\right)^Q \mu(B(x,\rho))
\end{equation}
whenever $x\in K,
\
0<\rho\le R^{1+\zeta}\le R\le R_*.$

For $\delta\in(0,1)$ sufficiently small, we may assume that $\delta^\varphi<R_*<1$. For $\delta^\varphi\le R\le R_*$, we have
\[
\delta^{\kappa}=(\delta^\varphi)^{\kappa/\varphi}
\le(\delta^\varphi)^{1+\zeta}
\le R^{1+\zeta}.
\]
Applying \eqref{growth-est} and \eqref{small-mass}, we have
\[
\mu(B(x,R))\leq
\begin{cases}
R^t,
& \delta^\kappa\le R\le\delta^\varphi,\\[4pt]
CR^Q\delta^{-\kappa(Q-t)},
& \delta^\varphi\le R\le R_*,\\[4pt]
\mu(\mathbb R^n)R_*^{-Q}R^Q\delta^{-\kappa(Q-t)},
& R\ge R_*.
\end{cases}
\]
Combining these estimates with the total-mass bound yields
\begin{equation}\label{eq:mass-cutoff}
\mu(B(x,R))\le C\min\{1,R^Q\delta^{-\kappa(Q-t)}\},
\quad R\ge \delta^{\kappa}.
\end{equation}

Denote $R_c=\delta^{\kappa(1-t/Q)}$. By Tonelli's theorem, we have
\begin{align}\label{eq:far-field}
\int_{|x-y|>\delta^{\kappa}}|x-y|^{-\alpha}\,d\mu(y)
&\le\alpha\int_{\delta^{\kappa}}^\infty R^{-\alpha-1}\mu(B(x,R))\,dR\notag\\
&\lesssim \int_0^\infty
R^{-\alpha-1}\min\{1,R^Q\delta^{-\kappa(Q-t)}\}\,dR\\\notag
&=\delta^{-\kappa(Q-t)}\int_0^{R_c}R^{Q-\alpha-1}\,dR
+\int_{R_c}^\infty R^{-\alpha-1}\,dR\notag\\
&\lesssim \delta^{-\kappa\alpha(1-t/Q)}\notag.
\end{align}
When $Q=\infty$, the same estimate follows directly from
\[
\int_{|x-y|>\delta^{\kappa}}|x-y|^{-\alpha}\,d\mu(y)
\le \mu(\mathbb R^n)\delta^{-\kappa\alpha}.
\]

We now split $J_u(x)$ over the regions
$|x-y|\le\delta$, $\delta<|x-y|\le \delta^{\kappa}$, and $|x-y|>\delta^{\kappa}$.
The first contribution is at most $\delta^t$.
By \eqref{eq:far-field}, the last contribution is at most
\[
u^\alpha\int_{|x-y|>\delta^{\kappa}}|x-y|^{-\alpha}\,d\mu(y)
\le C u^\alpha\delta^{-\kappa\alpha(1-t/Q)}.
\]
Applying the small-mass bound on the middle region gives
\[
\int_{\delta<|x-y|\le \delta^{\kappa}}k_u(x-y)\,d\mu(y)\lesssim
\begin{cases}
u^t,
& \delta^\kappa\le u<1,\\[4pt]
u^t,
& 0<u<\delta^\kappa,\ 0<t<\alpha,\\[4pt]
u^\alpha(1+|\log\delta|),
& 0<u<\delta^\kappa,\ t=\alpha,\\[4pt]
u^\alpha\delta^{\kappa(t-\alpha)},
& 0<u<\delta^\kappa,\ \alpha<t<Q.
\end{cases}
\]
Since $\kappa\alpha(1-t/Q)>0$, we have $1+|\log\delta|
\le C\delta^{-\kappa\alpha(1-t/Q)}$ for sufficiently small $\delta$.
Combining these estimates proves \eqref{eq:simple-potential}.

We next pass from $J_u$ to the full kernel.
The definition of $b$ gives
\[
\alpha b
=t+\frac{\kappa\alpha}{\theta}\left(1-\frac tQ\right)>t,
\]
so $s<t/b<\alpha$.
For $0<r<1$, we have
\begin{equation}\label{eq:kernel-majorant}
\phi_{r,\theta}^{s,\alpha}(z)
\le r^{s(1-\theta)}k_{r^\theta}(z)
+s r^s\int_r^{r^\theta}u^{-s-1}k_u(z)\,du.
\end{equation}
Integrating \eqref{eq:kernel-majorant} and applying Tonelli's theorem,
we obtain
\[
\int\phi_{r,\theta}^{s,\alpha}(x-y)\,d\mu(y)
\le r^{s(1-\theta)}J_{r^\theta}(x)
+s r^s\int_r^{r^\theta}u^{-s-1}J_u(x)\,du.
\]
To substitute \eqref{eq:simple-potential}, note that
\[
r^{s(1-\theta)}
+s r^s\int_r^{r^\theta}u^{-s-1}\,du=1,
\]
and, since $s<t$ and $s<\alpha$,
\begin{align*}
r^{s(1-\theta)}r^{\theta t}
+s r^s\int_r^{r^\theta}u^{t-s-1}\,du
&\le\left(1+\frac{s}{t-s}\right)
r^{s+\theta(t-s)},\\
r^{s(1-\theta)}r^{\theta\alpha}
+s r^s\int_r^{r^\theta}u^{\alpha-s-1}\,du
&\le\left(1+\frac{s}{\alpha-s}\right)
r^{s+\theta(\alpha-s)}.
\end{align*}
It follows that
\begin{align}\label{eq:three-terms}
\int\phi_{r,\theta}^{s,\alpha}(x-y)\,d\mu(y)
\le C\left[
\delta^t+r^{s+\theta(t-s)}
+r^{s+\theta(\alpha-s)}
\delta^{-\kappa\alpha(1-t/Q)}
\right].
\end{align}

Finally, take $r=\delta^b$.
After division by $r^s$, the three terms in
\eqref{eq:three-terms} have exponents
\[
t/b-s,\qquad
\theta(t-s),\qquad
\theta(\alpha-s)
-\frac{\kappa\alpha}{b}\left(1-\frac tQ\right).
\]
By the choice of $b$, we have
\[
b\theta\alpha-\kappa\alpha\left(1-\frac{t}{Q}\right)
=\theta t,
\]
and hence
\[
\theta(\alpha-s)
-\frac{\kappa\alpha}{b}\left(1-\frac{t}{Q}\right)
=\theta\left(\frac{t}{b}-s\right).
\]
Thus the first and third exponents are positive
precisely when $s<t/b$, while the second is positive
when $s<t$. Since $s<\min\{t,t/b\}$, all three
exponents are positive. Hence, we have 
\[
\int\phi_{r,\theta}^{s,\alpha}(x-y)\,d\mu(y)
\le C r^s\left[
r^{t/b-s}
+r^{\theta(t-s)}
+r^{\theta(t/b-s)}
\right].
\]
Let $\beta=\theta\bigl(\min\{t,t/b\}-s\bigr)>0.$ Since $0<\theta\le1$, each exponent in brackets is at least $\beta$.
As $0<r<1$, this proves
\[
\int\phi_{r,\theta}^{s,\alpha}(x-y)\,d\mu(y)
\le C r^{s+\beta},
\qquad r=\delta^b.
\]
All constants can be chosen independently of $x$ and $\delta$.
\end{proof}

\section{Quantitative projection bounds}
\label{sec:transversal}

We now average the potential estimate over a family of maps.
The required geometric input is a bound on how often two distinct
source points have nearby images. Let $\Lambda$ be a metric space
with a nonzero finite Borel measure $\nu$, and suppose
$F:\Lambda\times K\to\R^N$ is jointly continuous. Write
$f_\lambda(x)=F(\lambda,x)$ and
$\mu_\lambda=(f_\lambda)_\#\mu$, and assume that each
$f_\lambda$ is Lipschitz on $K$. We require the following
transversality condition for some $\alpha>0$:
\begin{equation}\label{eq:transversality}
\nu\{\lambda:|f_\lambda(x)-f_\lambda(y)|\le\rho|x-y|\}
\le C\rho^\alpha
\end{equation}
for all distinct $x,y\in K$ and $0<\rho<1$, uniformly in
$x,y,\rho$. For geometric background on such families, see \cite{BT}.

The upper bounds for the image dimensions follow from the Lipschitz
comparison \cite[Lemma~5.9]{AMT}: if $f:K\to\R^N$ is Lipschitz, then
\begin{equation}\label{eq:lipschitz}
 \overline{\dim}_\theta(f_\#\mu)\le\overline{\dim}_\theta\mu,
 \qquad
 \underline{\dim}_\theta(f_\#\mu)\le\underline{\dim}_\theta\mu.
\end{equation}

For the lower bounds, transversality allows us to compare the averaged
image kernel with the source kernel estimated in the preceding section.

\begin{lemma}\label{lem:averaging}
If \eqref{eq:transversality} holds and $0<s<\alpha$, then
\begin{equation}\label{eq:averaging}
\int_\Lambda\widetilde\phi_{r,\theta}^{s}
(f_\lambda(x)-f_\lambda(y))\,d\nu(\lambda)
\le C_s\phi_{r,\theta}^{s,\alpha}(x-y).
\end{equation}
\end{lemma}
\begin{proof}
The case $x=y$ is immediate. For $x\ne y$, let $L=|x-y|$ and
\begin{equation*}
    T(\lambda):=\frac{|f_\lambda(x)-f_\lambda(y)|}{|x-y|}.
\end{equation*}
By \eqref{eq:transversality}, $\nu\{T=0\}=0$ and, for $0<\rho\le1$,
\begin{align*}
 \int_{\{T\le \rho\}}T^{-s}\,d\nu
 &=\rho^{-s}\nu\{T\le \rho\}
   +s\int_0^\rho v^{-s-1}\nu\{T\le v\}\,dv\\
 &\lesssim \rho^{\alpha-s}.
\end{align*}
Splitting the integral at the kernel thresholds $T=r/L$ and
$T=r^\theta/L$ gives the following bounds in the three ranges of $L$:
\[
 \int_\Lambda\widetilde\phi_{r,\theta}^{s}
 (f_\lambda(x)-f_\lambda(y))\,d\nu(\lambda)
 \lesssim
 \begin{cases}
 1,&L\le r,\\
 (r/L)^\alpha+(r/L)^s,&r<L\le r^\theta,\\
 (r/L)^\alpha+(r/L)^s(r^\theta/L)^{\alpha-s},&L>r^\theta.
 \end{cases}
\]
These bounds are at most $C_s\phi_{r,\theta}^{s,\alpha}(x-y)$.
\end{proof}

Combining Lemmas~\ref{lem:potential} and~\ref{lem:averaging}, we
obtain a lower bound at image parameter $\theta$ from the source
dimension at any parameter $\kappa$. Taking $\kappa=\theta$ then
gives the quantitative loss bound.

\begin{theorem}\label{thm:transversal}
Assume \eqref{eq:transversality} and let $q=\dim_{\mathrm{qA}}\mu$.
There is a set $\Lambda_0\subset\Lambda$ of full $\nu$-measure such
that, for every $\lambda\in\Lambda_0$, the following hold.
\begin{enumerate}
\item[(i)] For all $\theta,\kappa\in(0,1]$,
\begin{equation}\label{eq:two-window-bound}
\begin{aligned}
 \overline{\dim}_\theta\mu_\lambda
 &\ge\min\biggl\{\overline{\dim}_\kappa\mu,
 \left[\frac1\alpha+\frac\kappa\theta
 \left(\frac1{\overline{\dim}_\kappa\mu}
       -\frac1{\max\{q,\alpha\}}\right)\right]^{-1}
 \biggr\}.
\end{aligned}
\end{equation}
\item[(ii)] For all $\theta\in(0,1]$,
\begin{equation}\label{eq:quantitative-bound}
 \left[\frac1{\overline{\dim}_\theta\mu}
       +\left(\frac1\alpha-\frac1q\right)_+\right]^{-1}
 \le\overline{\dim}_\theta\mu_\lambda
 \le\overline{\dim}_\theta\mu.
\end{equation}
\end{enumerate}
Both assertions hold on the same set $\Lambda_0$ with all upper
dimensions replaced by lower dimensions.
\end{theorem}

\begin{proof}
Fix $\theta,\kappa\in(0,1]$ and let $D$ denote either
intermediate dimension. The case $D_\kappa\mu=0$
is trivial. For $0<t<D_\kappa\mu$, choose
\[
\begin{cases}
\max\{q,\alpha\}<Q<\infty,\quad
0<\varphi<\kappa,\quad t<D_\varphi\mu,
& q<\infty,\\
Q=\infty,\quad \varphi=\kappa,
& q=\infty.
\end{cases}
\]
Let
\[
b=\frac{t}{\alpha}
+\frac{\kappa}{\theta}
\left(1-\frac{t}{Q}\right),
\qquad
0<s<\min\{t,t/b\}.
\]
By Definition~\ref{def:intermediate}, at arbitrarily small scales
$\delta$, there exists $x_\delta\in K$ such that
\[
\mu(B(x_\delta,u))\le u^t,
\qquad u\in[\delta,\delta^\varphi].
\]
Applying Lemmas~\ref{lem:potential} and Lemma~\ref{lem:averaging} with $r=\delta^b$
gives
\begin{equation}\label{3.2+4.1}
\int_\Lambda\int
\widetilde{\phi}^{s}_{r,\theta}
(f_\lambda(x_\delta)-z)
\,d\mu_\lambda(z)\,d\nu(\lambda)
\le Cr^{s+\beta}
\end{equation}
for some $\beta>0$.

Choose a sequence of admissible scales
$\delta_j\downarrow0$, let $r_j=\delta_j^b\leq 2^{-j}$, and write
\[
E_j(\lambda)=
\int\widetilde{\phi}^{s}_{r_j,\theta}
(f_\lambda(x_{\delta_j})-z)
\,d\mu_\lambda(z).
\]
By Tonelli's theorem and \eqref{3.2+4.1},
\[
\begin{aligned}
\int_\Lambda
\sum_{j=1}^\infty
r_j^{-s-\beta/2}E_j(\lambda)
\,d\nu(\lambda)
&=
\sum_{j=1}^\infty
r_j^{-s-\beta/2}
\int_\Lambda E_j(\lambda)\,d\nu(\lambda)\\
&\lesssim
\sum_{j=1}^\infty r_j^{\beta/2}\le
\sum_{j=1}^\infty 2^{-j\beta/2}
<\infty.
\end{aligned}
\]
Consequently, for $\nu$-almost every
$\lambda\in\Lambda$,
\[
E_j(\lambda)\le r_j^{s+\beta/2}
\]
for all sufficiently large $j$. Let $w_j=f_\lambda(x_{\delta_j})$.
Since
\[
\widetilde{\phi}^{s}_{r_j,\theta}(z)
\ge\left(\frac{r_j}{u}\right)^s,
\ \text{for all }
|z|\le u,\ r_j\le u\le r_j^\theta,
\]
we obtain
\begin{equation}\label{eq:local-mass-bound}
\sup_{r_j\le u\le r_j^\theta}
\frac{\mu_\lambda(B(w_j,u))}{u^s}
\le r_j^{-s}E_j(\lambda)
\le r_j^{\beta/2}\longrightarrow0.
\end{equation}

For the upper intermediate dimension, the
admissible scales are arbitrarily small, so
\eqref{eq:local-mass-bound} gives 
\begin{equation*}
    \overline{\dim}_\theta\mu_\lambda\ge s.
\end{equation*}
For the lower intermediate dimension, we may
take $\delta_j=e^{-j/b}$, so that $r_j=e^{-j}$.
Given $r_{j+1}<r\le r_j$ and
$u\in[r,r^\theta]$, we have
\[
\mu_\lambda(B(w_j,u))
\le
\begin{cases}
E_j(\lambda),&u<r_j,\\[2mm]
(u/r_j)^sE_j(\lambda),&u\ge r_j.
\end{cases}
\]
Hence
\[
\sup_{r\le u\le r^\theta}
\frac{\mu_\lambda(B(w_j,u))}{u^s}
\le e^s r_j^{\beta/2}\longrightarrow0,
\]
which gives
\[
\underline{\dim}_\theta\mu_\lambda\ge s.
\]

Thus, for the chosen intermediate dimension $D$,
\[
D_\theta\mu_\lambda\ge \min\{t,t/b\}
\]
for $\nu$-almost every $\lambda$.
Since
\[
\frac{t}{b}
=
\left[
\frac1\alpha+
\frac{\kappa}{\theta}
\left(\frac1t-\frac1Q\right)
\right]^{-1},
\]
letting $t\uparrow D_\kappa\mu$ and, when
$q<\infty$, $Q\downarrow\max\{q,\alpha\}$,
yields \eqref{eq:two-window-bound} for fixed $\theta,\kappa$.

To obtain a common full-measure set, first
take $\theta,\kappa$ rational and perform the
preceding argument for both intermediate
dimensions. Countable intersection gives a
set $\Lambda_0$ of full $\nu$-measure on which
\eqref{eq:two-window-bound} holds for all rational pairs.

By Lemma~\ref{lem:window}, each dimension function is
either finite and continuous on $(0,1]$ or
identically infinite. Thus continuity extends
\eqref{eq:two-window-bound} to all $\theta,\kappa\in(0,1]$ whenever
the relevant dimension functions are finite.
If the image dimension function is identically
infinite, the assertion is immediate. If the
source dimension function is identically
infinite, monotonicity and the rational-parameter
bounds give
\[
\overline{\dim}_\theta\mu_\lambda
\ge
\sup_{\substack{\eta\in\mathbb Q\\0<\eta<\theta}}
\overline{\dim}_\eta\mu_\lambda
\ge\alpha,
\]
as required.

Finally, taking $\kappa=\theta$ in \eqref{eq:two-window-bound},
using the Lipschitz bound \eqref{eq:lipschitz}, and observing
that
\[
\frac1\alpha-\frac1{\max\{q,\alpha\}}
=
\left(\frac1\alpha-\frac1q\right)_+,
\]
we obtain \eqref{eq:quantitative-bound} for both intermediate
dimensions on the same set $\Lambda_0$.
\end{proof}

Before specializing to orthogonal projections, we apply the
two-parameter estimate to the identity map. This sharpens the
comparison between source parameters in Lemma~\ref{lem:window}.
For $q\in(0,\infty)$,
\begin{equation}\label{eq:sharp-window}
\overline{\dim}_{\varphi}\mu
\le
\left[
\frac{\theta}{\varphi\overline{\dim}_{\theta}\mu}
+\frac{1-\theta/\varphi}{q}
\right]^{-1},
\qquad 0<\theta\le\varphi\le1.
\end{equation}
Indeed, the identity map on a one-point parameter space satisfies
\eqref{eq:transversality} with every exponent $\alpha>0$.
Applying Theorem~\ref{thm:transversal}(i) with $\alpha=q$, source
parameter $\varphi$ and image parameter $\theta$ yields
\eqref{eq:sharp-window}. The zero-dimensional case follows from
Lemma~\ref{lem:window}.

Rearranging \eqref{eq:sharp-window}, we obtain
\[
\varphi\left(
\frac{1}{\overline{\dim}_{\varphi}\mu}-\frac1q
\right)
\ge
\theta\left(
\frac{1}{\overline{\dim}_{\theta}\mu}-\frac1q
\right).
\]
Consequently,
\begin{equation}\label{eq:reciprocal-window-monotone}
\theta\longmapsto
\theta\left(
\frac{1}{\overline{\dim}_{\theta}\mu}-\frac1q
\right)
\end{equation}
is nondecreasing whenever $\overline{\dim}_{\theta}\mu>0$. The dimension functions in \eqref{eq:general-source-values} attain equality in \eqref{eq:sharp-window} for every pair of parameters. For $q=\infty$, the same conclusion follows from \eqref{eq:window},
with $1/q=0$. Both conclusions also hold for the lower
intermediate dimensions.

For orthogonal projections, the required transversality follows from
Mattila's estimate~\cite[inequality (5.12)]{Mattila2015}. We state it
for Frostman measures on the Grassmannian, since this form will also
be used to bound exceptional directions.

\begin{lemma}\label{lem:grassmann-trans}
    Denote $g=\dim G(n,m)$. Let $g-m<k\le g$, and let $\nu$ be a finite Borel measure on $\G$ such that
\begin{equation}\label{eq:frost}
\nu(B_{\G}(V,\rho))\le C\rho^k,\ \forall V\in\G,\ 0<\rho<1.
\end{equation}
Let $\alpha=k-(g-m)$. Then, uniformly in $z\in\R^n\setminus\{0\}$ and $0<\rho<1$,
\begin{equation}\label{eq:trans}
\nu\{V\in\G:|P_Vz|\le\rho|z|\}
\lesssim\rho^\alpha.
\end{equation}
\end{lemma}

\begin{proof}[Proof of Theorem~\ref{thm:orthogonal-main}]
Take $\Lambda=G(n,m)$, $\nu=\gamma_{n,m}$,
and $f_V=P_V$. Lemma~\ref{lem:grassmann-trans} gives
\[
\gamma_{n,m}\bigl\{
V:|P_V(x-y)|\le \rho|x-y|
\bigr\}
\lesssim \rho^m
\]
Thus \eqref{eq:transversality} holds with $\alpha=m$.
Since orthogonal projections are $1$-Lipschitz,
Theorem~\ref{thm:transversal}(ii) applies and gives \eqref{eq:intro-quantitative}
simultaneously for all $\theta\in(0,1]$
and both intermediate dimensions.
\end{proof}

\begin{corollary}\label{cor:preserve}
Let $\mu$ be a finite Borel measure with compact support $K\subset\R^n$. Let $m\in\{1,\ldots,n-1\}$ and suppose that $\dim_{\mathrm{qA}}\mu\le m$. Then, for all $\theta\in(0,1]$,
\[
\overline{\dim}_\theta\mu_V=\overline{\dim}_\theta\mu,\qquad
\underline{\dim}_\theta\mu_V=\underline{\dim}_\theta\mu.
\]
for $\gamma_{n,m}$-almost all $V\in G(n,m)$.
\end{corollary}

Keeping the image parameter fixed and letting the source parameter
tend to zero gives another consequence of the two-parameter estimate.
The resulting bound controls the limit as $\theta\to0^+$ without
any growth assumption on $\mu$.

\begin{corollary}\label{cor:limits}
Assume \eqref{eq:transversality}. There is a set $\Lambda_0\subset\Lambda$ of full $\nu$-measure such that, for every $\lambda\in\Lambda_0$,
\begin{equation}\label{eq:limit-preservation}
\begin{aligned}
 \min\left\{\alpha,\lim_{\theta\to0^+}
                  \overline{\dim}_\theta\mu\right\}
 &\le\lim_{\theta\to0^+}\overline{\dim}_\theta\mu_\lambda\le\lim_{\theta\to0^+}\overline{\dim}_\theta\mu.
\end{aligned}
\end{equation}
The same holds for lower dimensions on the same full-measure set. In particular, if \(\overline{\dim}_\theta\mu=h\le\alpha\) for all \(\theta\in(0,1]\), then
\[
\overline{\dim}_\theta\mu_\lambda=h,
\quad \forall\theta\in(0,1].
\]
\end{corollary}

\begin{proof}
Fix $\lambda\in\Lambda_0$ and
$\theta\in(0,1]$. If
\[
0<\lim_{\kappa\to0^+}
\overline{\dim}_\kappa\mu<\infty,
\]
then
\[
\frac{\kappa}{\theta}
\left(
\frac{1}{\overline{\dim}_\kappa\mu}
-\frac{1}{\max\{q,\alpha\}}
\right)\longrightarrow0,
\ \kappa\to0^+.
\]
Hence \eqref{eq:two-window-bound} gives
\[
\overline{\dim}_\theta\mu_\lambda
\ge
\min\left\{
\alpha,
\lim_{\kappa\to0^+}
\overline{\dim}_\kappa\mu
\right\}.
\]
The same bound is trivial when the limit is zero
and follows directly from \eqref{eq:two-window-bound} when it is
infinite. Letting $\theta\to0^+$ and
using \eqref{eq:lipschitz} yields \eqref{eq:limit-preservation}.

If $\overline{\dim}_\theta\mu=h\le\alpha$
for all $\theta\in(0,1]$, the preceding
bound and \eqref{eq:lipschitz} give
\[
h\le\overline{\dim}_\theta\mu_\lambda
\le\overline{\dim}_\theta\mu=h.
\]
The proof for lower intermediate dimensions
is identical.
\end{proof}

\section{Exceptional directions and centres}\label{sec:exception}

The almost-everywhere bounds leave open the size of the sets on
which projected dimensions fall below a prescribed threshold.
We apply Theorem~\ref{thm:transversal} to Frostman measures on these
sets to obtain vanishing Hausdorff measure at the critical exponent.
We first establish the Borel measurability needed for this argument.

\begin{lemma}\label{lem:borel}
Let $\Lambda$ be compact and $F:\Lambda\times K\to\R^N$
jointly continuous. For fixed $\theta\in(0,1]$, the maps
\[
 \lambda\mapsto\overline{\dim}_\theta\mu_\lambda,
 \qquad
 \lambda\mapsto\underline{\dim}_\theta\mu_\lambda
\]
are extended real-valued Borel functions. Their limits as
$\theta\to0^+$ are also Borel.
\end{lemma}

\begin{proof}
Fix $0<\theta\le 1$ and $s\ge 0$. For
$\lambda\in\Lambda$ and $0<r<1$, define
\[
M_\theta(\lambda,r,s)
=
\inf_{x\in K}
\max_{r\le u\le r^\theta}
\frac{\mu\{y\in K:
|f_\lambda(y)-f_\lambda(x)|\le u\}}{u^s}.
\]

Let
\[
f(\lambda,x,u)
=
u^{-s}\int_K
\mathbf{1}_{\{
|f_\lambda(y)-f_\lambda(x)|\le u
\}}\,d\mu(y),
\quad u>0.
\]
Since $F$ is jointly continuous, the indicator
is upper semicontinuous in $(\lambda,x,u)$
for every fixed $y\in K$. As $\mu$ is finite, Fatou's lemma shows that $f$ is jointly
upper semicontinuous.

Parametrizing the variable interval by
\[
u=(1-v)r+vr^\theta,
\qquad 0\le v\le1,
\]
and maximizing over the fixed compact interval
$[0,1]$ gives an upper semicontinuous function
\[
H_\theta(\lambda,x,r)
=
\max_{r\le u\le r^\theta}
f(\lambda,x,u).
\]

Taking the infimum over the fixed compact set
$K$ preserves upper semicontinuity. Indeed,
for every $c\in\mathbb R$,
\[
\{(\lambda,r):M_\theta(\lambda,r,s)<c\}
=
\bigcup_{x\in K}
\{(\lambda,r):H_\theta(\lambda,x,r)<c\}
\]
is open. Thus
\[
\mathcal C_{s,c}
=
\{(\lambda,r):M_\theta(\lambda,r,s)\ge c\}
\]
is closed in $\Lambda\times(0,1)$.

For fixed $s>0$, we define
\begin{align*}
    \mathcal U_{s,\theta}:&=\left\{
\lambda\in\Lambda:
\exists\,c,r_0>0,\ \forall 0<r<r_0,\ 
M_\theta(\lambda,r,s)\ge c
\right\}.\\
\mathcal L_{s,\theta}
:&=
\left\{
\lambda\in\Lambda:
\exists\,c>0,\ 
\forall\,r_0>0,\ 
\exists\,0<r<r_0,
\ 
M_\theta(\lambda,r,s)\ge c
\right\}.
\end{align*}
By Definition~\ref{def:intermediate}, 
\[
\mathcal U_{s,\theta}
=
\bigcup_{i=1}^\infty
\bigcup_{j=2}^\infty
\bigcap_{0<r<1/j}
\{\lambda\in\Lambda:
M_\theta(\lambda,r,s)\ge 1/i\}:=\bigcup_{i=1}^\infty
\bigcup_{j=2}^\infty
\bigcap_{0<r<1/j} U_{s,\theta}(r,i).
\]
For each fixed $r$, the set $U_{s,\theta}(r,i)$ is closed. Hence
$\mathcal U_{s,\theta}$ is Borel.

Similarly, we have
\[
\mathcal L_{s,\theta}
=
\bigcup_{i=1}^\infty
\bigcap_{j=2}^\infty
\bigcup_{k>j}
\left\{
\lambda\in\Lambda:
\exists r\in[1/k,1/j],\
M_\theta(\lambda,r,s)\ge 1/i
\right\}:=\bigcup_{i=1}^\infty
\bigcap_{j=2}^\infty
\bigcup_{k>j} L_{s,\theta}(i,j,k).
\]
For fixed $i,j,k$, the set $L_{s,\theta}(i,j,k)$ is the
projection onto $\Lambda$ of $\mathcal C_{s,1/i}\cap
\bigl(\Lambda\times[1/k,1/j]\bigr)$, which is compact. Therefore
$\mathcal L_{s,\theta}$ is also Borel.

For every $d>0$, the strict dimension
sublevel sets satisfy
\[
\left\{\lambda:
\overline{\dim}_\theta\mu_\lambda<d
\right\}
=
\bigcup_{\substack{s\in\mathbb Q\\0<s<d}}
\mathcal U_{s,\theta},\qquad \left\{\lambda:
\underline{\dim}_\theta\mu_\lambda<d
\right\}
=
\bigcup_{\substack{s\in\mathbb Q\\0<s<d}}
\mathcal L_{s,\theta}.
\]
Thus both intermediate dimension functions
are extended real-valued Borel functions.

Finally, by Lemma~\ref{lem:window}, both intermediate
dimension functions are nondecreasing in
$\theta$. Hence,
\[
\lim_{\theta\downarrow0}\overline{\dim}_\theta\mu_\lambda
=
\inf_{j\ge1}\overline{\dim}_{1/j}\mu_\lambda,\qquad \lim_{\theta\downarrow0}\underline{\dim}_\theta\mu_\lambda
=
\inf_{j\ge1}\underline{\dim}_{1/j}\mu_\lambda.
\]
Both limits are Borel as countable infima of Borel functions.
\end{proof}

\subsection{Orthogonal directions} 

For orthogonal projections, the critical exponent depends on how early
the source dimension function reaches the prescribed image threshold.
Using this information can improve the bound obtained from its value
at the image parameter alone.

\begin{theorem}\label{thm:exception}
Fix $0<\theta\le1$ and finite
$0<d\le\overline{\dim}_\theta\mu$.
Let $q=\dim_{\mathrm{qA}}\mu$ and define
\begin{equation}\label{eq:spectral-exception-exponent}
\kappa_d=\inf\{\kappa\in(0,1]:
\overline{\dim}_\kappa\mu\ge d\},
\qquad
b_*=\left[
\frac1d-\frac{\kappa_d}{\theta}
\left(\frac1d-\frac1q\right)
\right]^{-1}.
\end{equation}
If $b_*\le m$, then
\begin{equation}\label{eq:spectral-graded-exception}
\mathcal H^{m(n-m-1)+b_*}
\{V\in G(n,m):\overline{\dim}_\theta\mu_V<d\}=0.
\end{equation}
The same holds for lower intermediate dimensions, with
the corresponding definition of $\kappa_d$.
\end{theorem}

\begin{proof}
We first prove the assertion for the upper intermediate dimension.
Suppose that
\[
\mathcal H^{m(n-m-1)+b_*}
\left\{V\in G(n,m):
\overline{\dim}_\theta\mu_V<d\right\}>0.
\]
By Lemma~\ref{lem:borel} and Frostman's lemma in local
charts~\cite[Theorem B.2.5 and Lemma 3.1.1]{BishopPeres2017},
there exists a nonzero finite Borel measure $\nu$,
supported on the exceptional set, such that
\[
\nu(B(V,\rho))\lesssim
\rho^{m(n-m-1)+b_*}.
\]
By Lemma~\ref{lem:grassmann-trans}, the transversality
condition \eqref{eq:transversality} holds with exponent $b_*$.

If $\kappa_d>0$, then
$\overline{\dim}_{\kappa_d}\mu=d$.
Applying Theorem~\ref{thm:transversal}(i) with source
parameter $\kappa_d$ and image parameter $\theta$ gives
\[
\overline{\dim}_\theta\mu_V
\ge
\left[
\frac1{b_*}
+\frac{\kappa_d}{\theta}
\left(\frac1d-\frac1q\right)
\right]^{-1}
=d
\]
for $\nu$-almost every $V$, contradicting the support
of $\nu$.

If $\kappa_d=0$, then
$\lim_{\kappa\to0^+}\overline{\dim}_\kappa\mu\ge d$
and $b_*=d$. The same Frostman argument and
Corollary~\ref{cor:limits} give
\[
\overline{\dim}_\theta\mu_V
\ge
\lim_{\kappa\to0^+}
\overline{\dim}_\kappa\mu_V
\ge d
\]
for $\nu$-almost every $V$, again a contradiction.
The proof for the lower intermediate dimension is identical.
\end{proof}

If only the source dimension at $\theta$ is used, the exponent may be
larger. Indeed, \eqref{eq:reciprocal-window-monotone} gives
\[
d\le b_*\le
b:=\left(
\frac1d-\frac1{\overline{\dim}_\theta\mu}
+\frac1q
\right)^{-1}\le q.
\]
Thus, if $b\le m$, we also have
\begin{equation}\label{eq:graded-exception}
\mathcal H^{m(n-m-1)+b}
\{V\in G(n,m):
\overline{\dim}_\theta\mu_V<d\}=0.
\end{equation}

The same argument also controls exceptional directions for the limiting
dimensions as $\theta\to0^+$.

\begin{corollary}\label{cor:limit-exception}
Let
\[
h=\lim_{\theta\to0^+}
\overline{\dim}_\theta\mu.
\]
For every finite $0<d\le\min\{h,m\}$,
\begin{equation}\label{eq:limit-exception}
\mathcal H^{m(n-m-1)+d}
\left\{
V\in G(n,m):
\lim_{\theta\to0^+}
\overline{\dim}_\theta\mu_V<d
\right\}=0.
\end{equation}
The same holds for lower intermediate dimensions.
\end{corollary}

\begin{proof}
The same Frostman argument as in Theorem~\ref{thm:exception},
together with Lemma~\ref{lem:grassmann-trans},
gives a nonzero finite measure $\nu$ supported on the
exceptional set for which \eqref{eq:transversality}
holds with $\alpha=d$.
Corollary~\ref{cor:limits} then yields
\[
\lim_{\theta\to0^+}\overline{\dim}_\theta\mu_V
\ge \min\{d,h\}=d
\]
for $\nu$-almost every $V$, a contradiction.
The proof for lower intermediate dimensions is identical.
\end{proof}

\begin{remark}
The exponent $b_*$ is the smallest transversality exponent for which the two-parameter estimate yields $\overline{\dim}_\theta\mu_V\ge d$. For the source dimension functions in \eqref{eq:general-source-values}, as well as constant source dimension functions equal to \(q\), one has $b_*=b$, so the corresponding sharpness
constructions also apply to Theorem~\ref{thm:exception}. Moreover, Corollary~\ref{cor:nonharmonic-sharp} shows that the improvement $b_*<b$ can be strict and sharp even when the entire source dimension function is fixed.
\end{remark}

We next collect all directions at which either dimension function
fails to be preserved at some positive parameter. Write
\begin{equation*}
    \mathcal E_\mu
:=
\left\{
V\in G(n,m):
\exists\,\theta\in(0,1]
\text{ such that }
\overline{\dim}_\theta\mu_V
<
\overline{\dim}_\theta\mu
\text{ or }
\underline{\dim}_\theta\mu_V
<
\underline{\dim}_\theta\mu
\right\}.
\end{equation*}

\begin{corollary}\label{cor:all-exceptions}
Denote $q=\dim_{\mathrm{qA}}\mu$. The following assertions hold.
\begin{enumerate}
\item[(i)] If $0<q\le m$, then
\[
\mathcal H^{m(n-m-1)+q}
(\mathcal E_\mu)=0.
\]

\item[(ii)] Suppose that both intermediate dimension
functions of $\mu$ are constant on $(0,1]$, and let
\[
h=\max\left\{
\overline{\dim}_\theta\mu,
\underline{\dim}_\theta\mu
\right\}\le m.
\]
If $h>0$, then
\[
\mathcal H^{m(n-m-1)+h}(\mathcal E_\mu)=0.
\]
If $h=0$, then $\mathcal E_\mu=\varnothing$.
\end{enumerate}
\end{corollary}

\begin{proof}
For (i), continuity of the intermediate dimension functions gives
\[
\mathcal E_\mu
=
\bigcup_{D\in\{\overline{\dim},\underline{\dim}\}}
\bigcup_{\theta\in\mathbb Q\cap(0,1]}
\left\{
V\in G(n,m):
D_\theta\mu_V<D_\theta\mu
\right\}.
\]
For each $\theta$ with $D_\theta\mu>0$,
taking $d=D_\theta\mu$ in Theorem~\ref{thm:exception}
gives $b_*\leq q$. If $D_\theta\mu=0$, the
corresponding exceptional set is empty.
Thus (i) follows by countable subadditivity.

For (ii), monotonicity of the intermediate
dimension functions gives
\[
\mathcal E_\mu
=
\bigcup_{D\in\{\overline{\dim},\underline{\dim}\}}
\left\{
V\in G(n,m):
\lim_{\theta\downarrow0}D_\theta\mu_V
<
D_\theta\mu
\right\},
\]
where $D_\theta\mu$ is constant in $\theta$.
If $h>0$, applying \eqref{eq:limit-exception} to each nonzero
constant source dimension gives
\[
\mathcal H^{m(n-m-1)+h}(\mathcal E_\mu)=0.
\]
If $h=0$, neither dimension can decrease,
so $\mathcal E_\mu=\varnothing$.
\end{proof}

For each fixed $0<\alpha\le m$, the set of directions $V\in G(n,m)$ for which either version of
\eqref{eq:two-window-bound} fails for some $(\theta,\kappa)\in(0,1]^2$
has zero $\mathcal H^{m(n-m-1)+\alpha}$-measure. Indeed, applying the Frostman argument as in Theorem~\ref{thm:exception} to each rational pair $(\theta,\kappa)$ gives the conclusion outside a common Hausdorff null set. The continuity and infinite-dimensional arguments in the proof of Theorem~\ref{thm:transversal} then extend the estimate to all $(\theta,\kappa)\in(0,1]^2$.

\subsection{Radial centres}\label{sec:radial}

For radial projections, the parameter is the centre of projection.
To apply the same potential and Frostman arguments, we need to control
the centres from which two source points are seen in nearby directions.
For $z\notin K$, define
\[
 \pi_z(x)=\frac{x-z}{|x-z|},\qquad \mu_z=(\pi_z)_\#\mu.
\]
The sphere carries the chord metric, and $\pi_z|_K$ is Lipschitz.
The next lemma gives the required transversality estimates on bounded
sets of centres separated from $K$.

\begin{lemma}\label{lem:radial-trans}
Let $U\subset\R^n\setminus K$ be bounded, with
$\operatorname{dist}(U,K)>0$. Then
\begin{equation}\label{eq:radial-lebesgue}
 \mathcal L^n\{z\in U:|\pi_z(x)-\pi_z(y)|\le\rho|x-y|\}
 \lesssim_U\rho^{n-1}.
\end{equation}
If $\sigma$ is a finite $k$-Frostman measure on $U$, with
$1<k\le n$, then
\begin{equation}\label{eq:radial-frostman}
 \sigma\{z:|\pi_z(x)-\pi_z(y)|\le\rho|x-y|\}
 \lesssim\rho^{k-1}.
\end{equation}
Both bounds are uniform in distinct $x,y\in K$ and $0<\rho<1$.
\end{lemma}

\begin{proof}
Fix distinct $x,y\in K$, and let $L_{xy}$ be the line through $x$ and $y$. Since $|\xi-\zeta|\ge |\xi\wedge\zeta|$ for unit vectors, we have
\begin{equation}\label{eq:radial-geometric-bound}
\begin{aligned}
|\pi_z(x)-\pi_z(y)|
&\ge \frac{|(x-z)\wedge(y-z)|}{|z-x||z-y|}=\frac{|x-y|\operatorname{dist}(z,L_{xy})}
{|z-x||z-y|}.
\end{aligned}
\end{equation}
Since $U$ and $K$ are bounded, the denominator is uniformly bounded. Thus
\[
|\pi_z(x)-\pi_z(y)|\le \rho|x-y|
\quad\Longrightarrow\quad
\operatorname{dist}(z,L_{xy})\lesssim_U\rho.
\]
As $U$ is bounded, the intersection of $U$ with the $C\rho$-neighborhood of $L_{xy}$ can be covered by $O(\rho^{-1})$ balls of radius $C\rho$. Consequently,
\[
\mathcal L^n\{z\in U:
|\pi_z(x)-\pi_z(y)|\le \rho|x-y|\}
\lesssim_U \rho^{n-1}.
\]
If $\sigma$ is a $k$-Frostman measure on $U$, the same covering gives
\[
\sigma\{z\in U:
|\pi_z(x)-\pi_z(y)|\le \rho|x-y|\}
\lesssim \rho^{-1}\rho^k
=\rho^{k-1}.
\]
Both estimates are uniform in $x,y$ and $\rho$, as required.
\end{proof}

Lemma~\ref{lem:radial-trans} gives transversality exponent $n-1$
for Lebesgue measure on the centres, and exponent $k-1$ for a
$k$-Frostman measure. Substituting these exponents into the preceding
arguments gives the radial projection and exceptional-set bounds.

\begin{theorem}\label{thm:radial}
There exists a set $U_0\subset\mathbb R^n\setminus K$
of full $\mathcal L^n$-measure such that, for every
$z\in U_0$ and $\theta\in(0,1]$,
\begin{equation}\label{eq:radial-dimension-bound}
\left[
\frac{1}{\overline{\dim}_\theta\mu}
+
\left(
\frac{1}{n-1}
-\frac{1}{\dim_{\mathrm{qA}}\mu}
\right)_+
\right]^{-1}
\le \overline{\dim}_\theta\mu_z
\le \overline{\dim}_\theta\mu.
\end{equation}
The corresponding estimates for lower intermediate dimensions,
the two-parameter estimate \eqref{eq:two-window-bound} for all
$\theta,\kappa\in(0,1]$, and the conclusions of
Corollary~\ref{cor:limits} also hold on the same set,
with $\alpha=n-1$.

For fixed $\theta\in(0,1]$ and finite
$0<d\le\overline{\dim}_\theta\mu$, let $b_*$ be defined by \eqref{eq:spectral-exception-exponent}.
If $b_*\le n-1$, then
\begin{equation}\label{eq:radial-exceptional-bound}
\mathcal H^{1+b_*}
\left\{
z\in\mathbb R^n\setminus K:
\overline{\dim}_\theta\mu_z<d
\right\}=0.
\end{equation}
Moreover, writing $h=\lim_{\theta\downarrow0}
\overline{\dim}_\theta\mu$, for every finite $0<d\le\min\{h,n-1\}$,
\begin{equation}\label{eq:radial-limit-exceptional}
\mathcal H^{1+d}
\left\{
z\in\mathbb R^n\setminus K:
\lim_{\theta\downarrow0}
\overline{\dim}_\theta\mu_z<d
\right\}=0.
\end{equation}
Both exceptional-set estimates also hold for lower intermediate dimensions.
\end{theorem}

\begin{proof}
Consider the compact exhaustion
\[
U_j=\left\{
z\in\mathbb R^n:
|z|\le j,\ \operatorname{dist}(z,K)\ge j^{-1}
\right\},
\qquad
\mathbb R^n\setminus K=\bigcup_{j=1}^{\infty}U_j.
\]
On each $U_j\times K$, the radial projection family is jointly continuous. By Lemma~\ref{lem:radial-trans}, the transversality condition \eqref{eq:transversality} holds with exponent $n-1$ for $\mathcal L^n|_{U_j}$. Thus Theorem~\ref{thm:transversal} and Corollary~\ref{cor:limits} give the almost-everywhere conclusions.

For the exceptional-set estimates, Lemma~\ref{lem:radial-trans}
gives transversality exponent $b_*$ for any $(1+b_*)$-Frostman
measure on $U_j$. The proof of Theorem~\ref{thm:exception} then
yields \eqref{eq:radial-exceptional-bound}. Similarly, a $(1+d)$-Frostman measure and Corollary~\ref{cor:limits} give
\eqref{eq:radial-limit-exceptional}. The same arguments apply to lower intermediate dimensions. Finally, the countable exhaustion extends all conclusions to $\mathbb R^n\setminus K$.
\end{proof}

In particular, if $0<q=\dim_{\mathrm{qA}}\mu\le n-1$,
both intermediate dimension functions are preserved
for almost every centre, and
\[
\mathcal H^{1+q}(\mathcal E_\mu^{\mathrm{rad}})=0,
\]
where $\mathcal E_\mu^{\mathrm{rad}}$ denotes the set of
centres at which either dimension decreases for some
$\theta\in(0,1]$.
If both source dimension functions are constant with
maximum $0<h\le n-1$, the same conclusion holds
with $h$ in place of $q$, without any restriction on \(\dim_{\mathrm{qA}}\mu\).

\section{Sampling and layered measures}\label{sec:layered}

To prove sharpness, we replace reference measures by finite samples
that retain the relevant ball-mass bounds for their orthogonal and
exterior radial images. Placing the samples in separated layers gives
explicit dimension functions, with a prescribed quasi-Assouad dimension
and infinite Assouad dimension. We use this construction in
Sections~\ref{sec:sharp} and~\ref{sec:third-round}.

\subsection{A common sample for all projections}

To use the same samples for every projection, we must approximate
the reference measure simultaneously on all relevant ball preimages.
We therefore work with a single family containing ordinary balls,
preimages under orthogonal projections, and spherical-cap preimages
under radial projections with centres outside the support.

We define $\mathcal A=
\left\{
\left\{x\in\mathbb R^n:
P_1(x)\le0,\ P_2(x)\le0\right\}:
P_1,P_2\in\mathbb R[x_1,\ldots,x_n],\
\deg P_i\le2
\right\}.$ This family has
finite VC dimension, depending only on $n$.
It contains ordinary balls and the preimages of
balls under orthogonal projections, since both are
defined by a single quadratic inequality.

For radial projections, fix a centre $b$, a unit
vector $e$, and $0<u\le1$. Writing
$\pi_b(x)=(x-b)/|x-b|$, we have, for $x\ne b$,
\[
|\pi_b(x)-e|\le u
\iff
e\cdot(x-b)\ge
(1-u^2/2)|x-b|.
\]
Since $1-u^2/2>0$, this is equivalent to
\[
\begin{cases}
e\cdot(x-b)\ge0,\\
|e\cdot(x-b)|^2
\ge\left(1-\dfrac{u^2}{2}\right)^2|x-b|^2.
\end{cases}
\]
Thus the preimage of a spherical cap also belongs
to $\mathcal A$. Consequently, a relative sample
that approximates the reference measure uniformly
on $\mathcal A$ controls source balls and the
corresponding balls under all orthogonal and
exterior radial projections.

\begin{lemma}\label{lem:relative-clouds}
Let $\sigma$ be a compactly supported probability measure in $\R^n$,
and fix $x_*\in\R^n$. For every $N\ge2$
one can choose $x_1,\ldots,x_N\in\spt\sigma$ so that
\[
 \eta_N=\frac1{N+1}\left(\delta_{x_*}+\sum_{i=1}^N\delta_{x_i}\right)
\]
satisfies the following bounds for every $A\in\mathcal A$:
\begin{equation}\label{eq:relative-unconditional}
\begin{aligned}
 \eta_N(A)&\le C\left(\sigma(A)+\frac{\log(2N)}N\right),\\
 \eta_N(A)&\ge c\sigma(A),\text{ if }\sigma(A)\ge C\frac{\log(2N)}N.
\end{aligned}
\end{equation}
If $A\cap\spt\eta_N\ne\varnothing$, then
\begin{equation}\label{eq:relative-cloud-balls}
 (C \log(2N))^{-1}\max\{N^{-1},\sigma(A)\}
 \le\eta_N(A)\le C \log(2N)\max\{N^{-1},\sigma(A)\}.
\end{equation}
If $x_*\in\spt\sigma$ and $\sigma$ has the ball-growth bound
$C_0(R/u)^q$, then
\begin{equation}\label{eq:relative-cloud-growth}
 \frac{\eta_N(B(x,R))}{\eta_N(B(x,u))}
 \le C^2 C_0[\log(2N)]^2(R/u)^q,
 \ x\in\spt\eta_N,\  0<u\le R\le1.
\end{equation}
\end{lemma}

\begin{remark}
    Repeated observations are allowed, so every point of
$\operatorname{supp}\eta_N$ has mass at least $1/(N+1)$.
If $\sigma(\{x_*\})=0$, the sample can be chosen so that
$\eta_N(\{x_*\})=1/(N+1)$.
If $\sigma$ is nonatomic, the sample points may be chosen
distinct and different from $x_*$.
For sufficiently large $N$, any fixed finite subset of
$\operatorname{supp}\sigma$ may also be prescribed among
$x_1,\ldots,x_N$, with constants depending on its cardinality.
\end{remark}

\begin{proof}
The degree-two monomial lift maps $\mathcal A$ to
intersections of two halfspaces in a finite-dimensional
space. Hence $\mathcal A$ has finite VC dimension depending
only on $n$. Rational outward approximations on the
compact support of $\sigma$ make the class countably
separable.

For sufficiently large $N$, the relative approximation
theorem~\cite[Theorem~5]{LLS}, in the form
of~\cite[Theorem~2.11(ii)]{HPS}, with
error $1/2$, failure probability $1/4$, and
$p_N=L\log N/N$ for sufficiently large $L$, gives
an empirical measure
\[
\widehat{\sigma}_N
=\frac1N\sum_{i=1}^N\delta_{x_i}
\]
such that, simultaneously for all $A\in\mathcal A$,
\[
|\widehat{\sigma}_N(A)-\sigma(A)|
\le\frac12\max\{\sigma(A),p_N\}.
\]
Since
\[
\eta_N=\frac{N}{N+1}\widehat{\sigma}_N
       +\frac1{N+1}\delta_{x_*},
\]
we have
\[
|\eta_N(A)-\widehat{\sigma}_N(A)|
\le\frac1{N+1}.
\]
These estimates imply \eqref{eq:relative-unconditional}.

If $A\cap\operatorname{supp}\eta_N\ne\varnothing$,
then $\eta_N(A)\ge(N+1)^{-1}$. Moreover, \eqref{eq:relative-unconditional} gives
$\eta_N(A)\ge c\sigma(A)$ whenever $\sigma(A)\ge p_N$,
after adjusting the constants. Consequently,
\[
\eta_N(A)\gtrsim
\begin{cases}
N^{-1},&\sigma(A)<p_N,\\
\sigma(A),&\sigma(A)\ge p_N,
\end{cases}
\]
and hence
\[
\eta_N(A)\gtrsim
\frac1{\log(2N)}
\max\{N^{-1},\sigma(A)\}.
\]
The reverse inequality follows from \eqref{eq:relative-unconditional}, proving \eqref{eq:relative-cloud-balls}.

If $\sigma(\{x_*\})=0$, the sampled points may be chosen
different from $x_*$. When $\sigma$ is nonatomic, they
may also be chosen pairwise distinct. To prescribe a
fixed finite set of support points, reserve the
corresponding sample positions and apply the relative
approximation theorem to the remaining points. The
additional error is $O(N^{-1})$ and is absorbed in
\eqref{eq:relative-unconditional}--\eqref{eq:relative-cloud-balls}.

Finally, suppose that $x_*\in\operatorname{supp}\sigma$
and that $\sigma$ satisfies the stated $q$-growth bound.
For $x\in\operatorname{supp}\eta_N$ and $0<u\le R\le1$,
we have
\[
\begin{aligned}
\max\{N^{-1},\sigma(B(x,R))\}
&\le C_0\left(\frac Ru\right)^q
   \max\{N^{-1},\sigma(B(x,u))\}.
\end{aligned}
\]
Applying \eqref{eq:relative-cloud-balls} to the two balls gives
\[
\frac{\eta_N(B(x,R))}
     {\eta_N(B(x,u))}
\lesssim C_0[\log(2N)]^2
\left(\frac Ru\right)^q,
\]
which proves \eqref{eq:relative-cloud-growth}.
\end{proof}

\subsection{Separated layers}

We place the samples at geometrically separated scales, so that
small balls can isolate individual layers. After a fixed dilation,
assume $\spt\sigma\cup\{x_*\}\subset\overline B(0,1)$.
For a unit vector $v$, let
\begin{equation}\label{eq:sharp-layers}
 R_j=2^{-j},\quad \ell_j=\frac{R_j}{100j},\quad
 T_j(x)=R_jv+\ell_jx,\ j\ge2.
\end{equation}
For samples from Lemma~\ref{lem:relative-clouds}, define
\begin{equation}\label{eq:sharp-mu}
 \mu=Z^{-1}\sum_{j\ge2}j^{-2}(T_j)_\#\eta_{N_j},
 \qquad Z=\sum_{j\ge2}j^{-2}.
\end{equation}
We use either $N_j=\lceil R_j^{-a}\rceil$, with $a>0$, or
$N_j=\lceil e^{j^2}\rceil$. The first lemma turns lower mass bounds
on the individual layers into upper bounds for the dimension functions
of their mixture.

\begin{lemma}\label{lem:image-layers}
Let $a,p>0$, $h\ge0$ and $C_0\ge1$. Let $\nu_j$ be finitely
supported probabilities with $\spt\nu_j\subset B(z_0,C_0R_j)$.
Suppose that for all sufficiently large $j$,
\begin{equation}\label{eq:sharp-common-lower}
 \nu_j(B(z,u))\ge c\left(\log2N_j\right)^{-h}
 \max\left\{N_j^{-1},\min\left\{1,(u/R_j)^p\right\}\right\}
\end{equation}
for $z\in\spt\nu_j$ and $0<u\le1$.
Then $\nu=Z^{-1}\sum_{j\ge2}j^{-2}\nu_j$ satisfies
\begin{equation}\label{eq:general-image-values}
 \overline{\dim}_t\nu\le
 \begin{cases}
 \displaystyle\frac{pat}{p+at},&N_j=\lceil R_j^{-a}\rceil,\\[5pt]
 p,&N_j=\lceil e^{j^2}\rceil,
 \end{cases}
 \qquad 0<t\le1.
\end{equation}
\end{lemma}

\begin{proof}
For $N_j=\lceil R_j^{-a}\rceil$,
\eqref{eq:sharp-common-lower} gives
\[
\nu(B(z,u))
\gtrsim j^{-C_1}
\max\left\{
R_j^a,\min\left\{1,(u/R_j)^p\right\}
\right\},
\ z\in\spt\nu_j.
\]
Fix $0<t\le1$ and let $d=pat/(p+at)$, so that $d/a<t$ and $p(t-d/a)=td$.
For sufficiently small $r>0$ and
$z\in\operatorname{supp}\nu_j$, we distinguish
three cases.

If $R_j\ge r^{d/a}$, take $u=r$. By (6.6),
\[
\nu(B(z,r))
\gtrsim j^{-C_1}R_j^a
\gtrsim (1+|\log r|)^{-C_1}r^d.
\]

If $r^t/(4C_0)\le R_j<r^{d/a}$, take $u=r^t$.
Since $p(t-d/a)=td$, (6.6) gives
\[
\begin{aligned}
\nu(B(z,r^t))
&\gtrsim j^{-C_1}
\min\left\{1,\left(\frac{r^t}{R_j}\right)^p\right\}\gtrsim (1+|\log r|)^{-C_1}r^{td}.
\end{aligned}
\]

If $R_j<r^t/(4C_0)$, take $u=r^t$.
Since
$\operatorname{supp}\nu_k\subset B(z_0,C_0R_k)$,
we have
\[
\bigcup_{R_k<r^t/(4C_0)}
\operatorname{supp}\nu_k
\subset B(z,r^t).
\]
Consequently,
\[
\nu(B(z,r^t))
\gtrsim
\sum_{k\gtrsim|\log r|}k^{-2}
\gtrsim(1+|\log r|)^{-1}.
\]

Hence, for sufficiently small $r>0$
and $z\in\operatorname{supp}\nu_j$,
\[
\nu(B(z,u))\gtrsim
\begin{cases}
(1+|\log r|)^{-C_1}r^d,
&u=r,\quad R_j\ge r^{d/a},\\[3pt]
(1+|\log r|)^{-C_1}r^{td},
&u=r^t,\quad r^t/(4C_0)\le R_j<r^{d/a},\\[3pt]
(1+|\log r|)^{-1},
&u=r^t,\quad R_j<r^t/(4C_0).
\end{cases}
\]
The same tail estimate applies at $z_0$.
Thus, for every $\varepsilon>0$ and all sufficiently
small $r$, there exists $u\in[r,r^t]$ such that
\[
\nu(B(z,u))
\gtrsim_\varepsilon u^{d+\varepsilon},
\ z\in\spt\nu.
\]
Hence $\overline{\dim}_t\nu\le d$.

For $N_j=\lceil e^{j^2}\rceil$, we discard
the $N_j^{-1}$ term in
\eqref{eq:sharp-common-lower}.
If $R_j\ge r/(4C_0)$, taking $u=r$ gives
\[
\nu(B(z,r))
\gtrsim(1+|\log r|)^{-C_2}r^p.
\]
For the remaining layers, the preceding tail
estimate applies. Thus, for every
$\varepsilon>0$ and all sufficiently small $r$,
\[
\nu(B(z,r))\gtrsim_\varepsilon r^{p+\varepsilon},
\ z\in\spt\nu,
\]
and therefore $\overline{\dim}_t\nu
\le\overline{\dim}_1\nu\le p.$
\end{proof}

For projected samples, a lower $p$-power mass bound for the
projected reference measure gives \eqref{eq:sharp-common-lower}
through Lemma~\ref{lem:relative-clouds}. This bound must also hold
at the image of the distinguished point. A fixed logarithmic loss
can be absorbed into a power of $\log(2N_j)$ when
$(u/R_j)^p\ge N_j^{-1}$; below this threshold, the mass at each
sampled point suffices.

To determine the source dimensions exactly, we also require a
matching upper mass bound at the distinguished point. The next
lemma combines this condition with the reference growth estimate.

\begin{lemma}\label{lem:general-reference-layers}
Let $\mu$ be the measure defined in \eqref{eq:sharp-mu},
and suppose that $p\in(0,\infty)$ and
$x_*\in K_0=\operatorname{supp}\sigma$.
Assume that, for some $h\ge0$,
\begin{align}
\frac{\sigma(B(x,R))}{\sigma(B(x,u))}
&\le C(1+|\log u|)^h
\left(\frac{R}{u}\right)^p,\quad \forall x\in K_0,\ 0<u\le R\le1,
\label{eq:reference-ratio}\\
\sigma(B(x_*,u))
&\lesssim u^p,\quad \forall 0<u<u_0.
\label{eq:marked-point-mass}
\end{align}
Then $\mu$ has countable compact support,
$\dim_{\rm qA}\mu=p$ and $\dim_{\rm A}\mu=\infty$.
Moreover, for every $0<t\le1$, we have
\begin{equation}\label{eq:general-source-values}
\underline{\dim}_t\mu
=\overline{\dim}_t\mu
=
\begin{cases}
\displaystyle
\left(\frac1p+\frac1{at}\right)^{-1},
& N_j=\lceil R_j^{-a}\rceil,\\[6pt]
p,
& N_j=\lceil e^{j^2}\rceil.
\end{cases}
\end{equation}

If $p=\infty$, the corresponding conclusions hold
with \eqref{eq:reference-ratio} and
\eqref{eq:marked-point-mass} replaced by
$\operatorname{dist}(x_*,K_0)>0$.
\end{lemma}

\begin{proof}
\emph{Geometry.}
Let $K_j=T_j(\operatorname{supp}\eta_{N_j})$, $x_j^*=T_jx_*$ and $K=\{0\}\cup\bigcup_{j\ge2}K_j$. Since $K_j\subset B(R_jv,R_j/200)$ and $R_{j+1}=R_j/2$, we have
\begin{equation}\label{eq:layer-sep}
    \operatorname{dist}(K_j,K\setminus K_j)
\ge R_j/3.
\end{equation}
Since each $K_j$ is finite and $R_j\to0$,
$K=\operatorname{supp}\mu$ is countable and compact.
Moreover, the tail estimate
$\sum_{j\ge m}j^{-2}\sim m^{-1}$ yields
\begin{equation}\label{eq:ball-0}
    \mu(B(0,u))\gtrsim (1+|\log u|)^{-1}
\end{equation}
for all sufficiently small $u>0$.
Finally, both choices of $N_j$ satisfy
$\log(2N_j)\lesssim j^2$.

\emph{Quasi-Assouad dimension.}
By \eqref{eq:reference-ratio}, taking $R=1$
and using compactness of $K_0$, we obtain
\begin{equation}\label{sigma-lower-bd}
    \sigma(B(x,u))
\gtrsim \frac{u^p}{(1+|\log u|)^h},
\quad\forall x\in K_0,\ 0<u\le1.
\end{equation}
For $\eta_N$, the atomic mass gives
a lower bound of order $N^{-1}$ at every support
point. Combining this with the relative sampling estimate \eqref{eq:relative-cloud-balls} in Lemma~\ref{lem:relative-clouds}, and using
$1+|\log u|\lesssim\log(2N)$ whenever
$u^p\ge N^{-1}$, yields
\[
\eta_N(B(x,u))
\gtrsim
\frac{\max\{N^{-1},u^p\}}
     {[\log(2N)]^{h+1}},
\quad \forall x\in\operatorname{supp}\eta_N,\
0<u\le1.
\]
Therefore Lemma~\ref{lem:image-layers} gives the required upper
bounds for the intermediate dimensions.

To estimate the quasi-Assouad dimension, we
combine the reference growth bound \eqref{eq:reference-ratio} with
the relative sampling estimate \eqref{eq:relative-cloud-balls} of Lemma~\ref{lem:relative-clouds}.
\begin{equation*}
    \frac{\eta_{N_j}(B(x,R))}{\eta_{N_j}(B(x,u))}\lesssim j^4(1+|\log u|)^h \left(R/u\right)^p.
\end{equation*}

Using the scaling factor $\ell_j$ of $T_j$ and the separation
estimate \eqref{eq:layer-sep}, we obtain, for $x\in K_j$ and $R<R_j/4$,
\[
\frac{\mu(B(x,R))}{\mu(B(x,u))}
=
\frac{\eta_{N_j}(B(T_j^{-1}x,R/\ell_j))}
     {\eta_{N_j}(B(T_j^{-1}x,u/\ell_j))}
\lesssim
j^4(1+|\log u|)^h
\left(R/u\right)^p.
\]

For $R\ge R_j/4$, we distinguish three cases. If $u<R_j/4\le R$, the lower mass bound
for the $j$-th layer gives
\[
\mu(B(x,u))
\gtrsim
j^{-4}(1+|\log u|)^{-h}
\left(u/R_j\right)^p.
\]
Since $\mu(B(x,R))\le1$, we obtain
\[
\frac{\mu(B(x,R))}{\mu(B(x,u))}
\lesssim
j^4(1+|\log u|)^h
\left(R_j/u\right)^p.
\]

If $R_j/4\le u\le4R_j$, then
$\operatorname{diam}K_j<R_j/4\le u$, so
$K_j\subset B(x,u)$. Consequently,
\[
\mu(B(x,u))\ge\mu(K_j)=Z^{-1}j^{-2},
\]
and hence 
\[
\frac{\mu(B(x,R))}{\mu(B(x,u))}
\lesssim j^2.
\]

Finally, if $u>4R_j$, then
$B(0,u/2)\cap K\subset B(x,u)$.
By \eqref{eq:ball-0},
\[
\mu(B(x,u))
\ge\mu(B(0,u/2))
\gtrsim(1+|\log u|)^{-1},
\]
which yields
\[
\frac{\mu(B(x,R))}{\mu(B(x,u))}
\lesssim1+|\log u|.
\]
Hence, for $x\in K_j$ and $0<u\le R<r_0$,
\[
\frac{\mu(B(x,R))}{\mu(B(x,u))}
\lesssim
\begin{cases}
j^4(1+|\log u|)^h(R/u)^p,
& R<R_j/4,\\
j^4(1+|\log u|)^h(R_j/u)^p,
& u<R_j/4\le R,\\
j^2,
& R_j/4\le u\le4R_j,\\
1+|\log u|,
& u>4R_j.
\end{cases}
\]
Since $j\lesssim1+|\log u|$ in the first
three cases, and \eqref{eq:ball-0} also applies at $x=0$,
we conclude that
\[
\frac{\mu(B(x,R))}{\mu(B(x,u))}
\le C(1+|\log u|)^{h+8}
\left(R/u\right)^p,
\quad\forall x\in K,\ 0<u\le R<r_0.
\]
For every $\varepsilon,\delta>0$,
\[
u\le R^{1+\delta}
\quad\Longrightarrow\quad
(1+|\log u|)^{h+8}
\lesssim_{\varepsilon,\delta}(R/u)^\varepsilon.
\]
Hence $\dim_{\rm qA}\mu\le p$ by the Definition~\ref{def:qA}.

For the reverse inequality, consider $N_j\sim R_j^{-a}$. Take $\rho_j=\ell_j/2$, $\beta\in(0,a/p)$ and $s_j=\ell_j R_j^{\beta}$. Then, we have
\begin{align*}
    &\rho_j/s_j\sim R_j^{-\beta}=2^{j\beta}\to\infty,\qquad N_j\left(s_j/\ell_j\right)^p\sim R_{j}^{p\beta-a}\to \infty.
\end{align*}
For all sufficiently large $j$, both balls $B(x_j^*,\rho_j)$ and $B(x_j^*,s_j)$ meet only $K_j$, then
\begin{align*}
    \frac{\mu(B(x_j^*,\rho_j))}{\mu(B(x_j^*,s_j))}&=\frac{\eta_{N_j}(B(x_*,1/2))}{\eta_{N_j}(B(x_*,R_j^{\beta}))}\\
    &\geq\frac{\left(C\log(2N_j)\right)^{-1}\max\left\{N_j^{-1},c_0\right\}}{\left(C\log(2N_j)\right)\max\left\{N_j^{-1},\sigma(B(x_*,R_j^{\beta}))\right\}}\quad (\text{use }\eqref{eq:relative-cloud-balls})\\
    &\gtrsim \left(C\log(2N_j)\right)^{-2}\frac{c_0}{\max\left\{N_j^{-1},R_j^{p\beta}\right\}}\quad (\text{use }\eqref{eq:marked-point-mass})\\
    &\gtrsim j^{-4}R_j^{-p\beta}\sim j^{-4}\left(\frac{\rho_j}{s_j}\right)^p.
\end{align*}

It remains to verify the scale condition in the definition of the quasi-Assouad dimension, namely, that \(s_j\le \rho_j^{1+\delta_0}\) for some \(\delta_0>0\) and all sufficiently large \(j\). Take $\delta_0\in(0,\beta)$. Since 
\begin{equation*}
    j^{\delta_0}2^{-j(\beta-\delta_0)}\to0,\quad j\to\infty.
\end{equation*}
Then for all sufficiently large \(j\), 
\begin{equation*}
    j^{\delta_0}R_j^{\beta-\delta_0}\leq\frac{1}{2\cdot 200^{\delta_0}}.
\end{equation*}
Rearranging it yields \(s_j\le \rho_j^{1+\delta_0}\). Hence, $\dim_{\rm qA}\mu\geq p$.

For $N_j\sim e^{j^2}$, the same argument yields \(\dim_{\rm qA}\mu\ge p\) for any fixed \(\beta>0\) and \(\delta_0\in(0,\beta)\).

\emph{Intermediate dimensions.}
First consider $N_j\sim R_j^{-a}$. Let
\[
d=\left(\frac1p+\frac1{at}\right)^{-1},
\]
so that $d/a<t$, $d<p$ and
$p(t-d/a)=td$. Given sufficiently small
$r>0$, choose $j=j_r$ such that
$R_j\sim r^{d/a}$. Since
$r^t=o(R_j)$, the balls $B(x_j^*,u)$,
$r\le u\le r^t$, intersect only $K_j$. Moreover,
\[
\frac{r^t}{\ell_j}
\sim \log\frac{1}{r}\cdot r^{t-d/a}\to0,\quad r\to 0^+.
\]
By the upper relative sampling estimate \eqref{eq:relative-cloud-balls} and \eqref{eq:marked-point-mass}, for every $0<s<d$ and $r\le u\le r^t$,
\begin{align*}
    \frac{\mu(B(x_j^*,u))}{u^s}&=\frac{Z^{-1}j^{-2}\eta_{N_j}\left(B(x_*,u/\ell_{j})\right)}{u^s}\\
    &\leq \frac{Z^{-1}j^{-2}C\log(2N_j)\max\left\{N_{j}^{-1},\sigma\left(B(x_*,u/\ell_{j})\right)\right\}}{u^s}\quad (\text{use }\eqref{eq:relative-cloud-balls})\\
    &\lesssim \left(N_j^{-1}+(u/\ell_j)^p\right)u^{-s}\lesssim r^{d-s}+j^pu^{p-s}R_j^{-p}\quad (N_j^{-1}\sim R_j^a\sim r^d)\\
    &\lesssim \left(1+\log 1/r\right)^p(r^{d-s}+r^{t(d-s)}).\quad (s<d<p \text{ and } p(t-d/a)=td)
\end{align*}
Then
\begin{equation}\label{eq:layer-witness-obstruction}
    \sup_{u\in[r,r^t]}\frac{\mu(B(x_j^*,u))}{u^s}\to 0,\quad r\to 0^+.
\end{equation}
Thus $\underline{\dim}_t\mu\ge d$.

For exponential sampling,
$N_j=\lceil e^{j^2}\rceil$, fix
$0<s<p$ and choose
\[
0<c<\frac{t(p-s)}p,
\qquad R_j\sim r^c.
\]
Again $r^t=o(R_j)$, so the balls
$B(x_j^*,u)$, $r\le u\le r^t$,
meet only $K_j$. Taking the supremum in the same sampling estimate gives
\[
\sup_{r\le u\le r^t}
\frac{\mu(B(x_j^*,u))}{u^s}
\lesssim (1+|\log r|)^C
 \left(e^{-j^2}r^{-s}+r^{t(p-s)-cp}\right)
\longrightarrow0.
\]
Here $j\sim|\log r|$ and $t(p-s)-cp>0$.
Hence $\underline{\dim}_t\mu\ge p$.

For the upper bound, we verify the hypotheses of
Lemma~\ref{lem:image-layers}. Set $\nu_j=(T_j)_{\#}\eta_{N_j}$.
For $u\leq\ell_j$, we have
\begin{align}\label{nu-lower-bd}
    \nu_j(B(z,u))&=\eta_{N_j}\left(B\left(T_j^{-1}z,u/\ell_j\right)\right)\notag\\
    &\geq \left(C\log2N_j\right)^{-1}\max\left\{N_j^{-1},\sigma\left(B\left(T_j^{-1}z,u/\ell_j\right)\right)\right\}\notag\\
    &\gtrsim \left(\log2N_j\right)^{-1}\max\left\{N_j^{-1},\frac{(u/\ell_j)^p}{(1+\log(\ell_j/u))^h}\right\}\quad (\text{use }\eqref{sigma-lower-bd})\notag\\
    &\gtrsim \left(\log2N_j\right)^{-1-h}\max\left\{N_j^{-1},(u/\ell_j)^p\right\}\\
    &\geq \left(\log2N_j\right)^{-1-h}\max\left\{N_j^{-1},(u/R_j)^p\right\}.\notag
\end{align}
If $u>\ell_j$, by the monotonicity of $\nu_j$ and \eqref{nu-lower-bd}, 
\begin{equation}
    \nu_j(B(z,u))\geq\nu_j(B(z,\ell_j))\gtrsim \left(\log2N_j\right)^{-1-h}.
\end{equation}
Hence, we obtain
\begin{equation*}
    \nu_j(B(z,u))
\gtrsim\left(\log2N_j\right)^{-1-h}
\max\left\{
N_j^{-1},
\min\left\{1,\left(\frac{u}{R_j}\right)^p\right\}
\right\}.
\end{equation*}
Combining the above lower bounds with
Lemma~\ref{lem:image-layers} gives
\[
\underline{\dim}_t\mu
=\overline{\dim}_t\mu
=
\begin{cases}
\displaystyle
\left(\frac1p+\frac1{at}\right)^{-1},
& N_j=\lceil R_j^{-a}\rceil,\\[6pt]
p,
& N_j=\lceil e^{j^2}\rceil.
\end{cases}
\]

\emph{The case $p=\infty$.}
Let $d_*=\operatorname{dist}(x_*,K_0)>0$.
Since $T_j$ has scaling factor $\ell_j$,
\[
\operatorname{dist}\bigl(x_j^*,
K_j\setminus\{x_j^*\}\bigr)
\ge \ell_j d_*.
\]
Thus, for $u<\ell_jd_*$, the ball
$B(x_j^*,u)$ contains only \(x_j^*\) in \(K\).

First suppose that $N_j\sim R_j^{-a}$. We show that
$\overline{\dim}_t\mu\le at$.
Fix $\varepsilon>0$ and sufficiently small $r>0$.
For $x\in K_j$, we have
\begin{equation*}
    \mu(B(x,u))
\gtrsim
\begin{cases}
(1+|\log u|)^{-2}r^{at},
& R_j\ge r^t/4,\ u=r,\\[4pt]
(1+|\log u|)^{-1},
& R_j<r^t/4,\ u=r^t.
\end{cases}
\end{equation*}
Thus, for every $x\in K$, there exists
$u\in[r,r^t]$ such that
$\mu(B(x,u))\gtrsim_\varepsilon u^{at+\varepsilon}$.
Consequently,
$\overline{\dim}_t\mu\le at$.

For the reverse inequality, fix $0<s<at$ and choose
$R_j\sim r^\gamma$ with
\[
 s/a<\gamma<t.
\]
Since
\[
r^t=o(\ell_j)
\quad\text{and}\quad
r^t=o(R_j),
\]
every ball $B(x_j^*,u)$, 
$r\le u\le r^t$ contains only $x_j^*$,
provided $r>0$ is sufficiently small.
Consequently,
\[
\begin{aligned}
\sup_{r\le u\le r^t}
\frac{\mu(B(x_j^*,u))}{u^s}
&=\frac{Z^{-1}j^{-2}}{(N_j+1)r^s}\lesssim
(1+|\log r|)^C r^{a\gamma-s}
\to0,\quad r\to 0^+.
\end{aligned}
\]
Hence $\underline{\dim}_t\mu\ge at$.

For $N_j\sim e^{j^2}$, fix any finite $s>0$ and choose
$R_j\sim r^\gamma$ with $0<\gamma<t$.
Then $r^t=o(\ell_j)$, and the same isolated-point argument gives
\[
\sup_{r\le u\le r^t}
\frac{\mu(B(x_j^*,u))}{u^s}
\lesssim
j^{-2}e^{-j^2}r^{-s}
\to0,\quad r\to 0^+.
\]
Together with the upper bounds, this proves
\[
\underline{\dim}_t\mu
=\overline{\dim}_t\mu
=
\begin{cases}
at, & N_j=\lceil R_j^{-a}\rceil,\\[4pt]
\infty, & N_j=\lceil e^{j^2}\rceil.
\end{cases}
\]

We next consider the quasi-Assouad dimension.
For $N_j\sim R_j^{-a}$ and any $\delta>0$,
choose $\rho_j=B\ell_j$, with $B$ large enough that $B(x_j^*,\rho_j)$ contains $K_j$
but no other layer. For sufficiently large
$j$, the smaller ball
$B(x_j^*,\rho_j^{1+\delta})$
contains only $x_j^*$. Hence
\[
\frac{\mu(B(x_j^*,\rho_j))}
{\mu(B(x_j^*,\rho_j^{1+\delta}))}
=N_j+1\sim R_j^{-a}.
\]
Since $\rho_j\sim R_j/j$, we have
\[
\frac{\rho_j}{\rho_j^{1+\delta}}
=\rho_j^{-\delta}
\sim j^\delta R_j^{-\delta}.
\]
It follows that $H(\delta)\ge a/\delta$,
and therefore
$\dim_{\rm qA}\mu=\infty$.
For $N_j\sim e^{j^2}$, the same conclusion follows from
$\overline{\dim}_1\mu=\infty$.

Finally, for either sampling scheme,
\[
\mu(B(x_j^*,R_j/4))=Z^{-1}j^{-2},
\qquad
\mu(B(x_j^*,2R_j))\gtrsim j^{-1}.
\]
The mass ratio diverges while the radius ratio remains fixed, so
$\dim_{\rm A}\mu=\infty$.
\end{proof}

\section{Sharp projection bounds}\label{sec:sharp}

The layer construction reduces sharpness to a choice of reference
measure. We require prescribed source growth and lower $m$-power
image mass bounds, so that the resulting image upper bounds match
the almost-everywhere lower bounds.

\begin{theorem}\label{thm:sharp-clouds}
Let $m\ge1$ be an integer, $m<p\le\infty$, and $0<a\le\infty$.
There is a probability measure $\mu$ with countable compact support
in $\R^{m+1}$ such that $\dim_{\mathrm{qA}}\mu=p$,
$\dim_{\mathrm A}\mu=\infty$, and
\begin{equation}\label{eq:sharp-source}
 \underline{\dim}_\theta\mu=\overline{\dim}_\theta\mu
 =\left(\frac1p+\frac1{a\theta}\right)^{-1},
 \quad0<\theta\le1.
\end{equation}
For every $V\in G(m+1,m)$, $z\notin\spt\mu$, and $0<\theta\le1$,
\begin{equation}\label{eq:sharp-image-upper}
\begin{aligned}
 \overline{\dim}_\theta\mu_V
 &\le\left(\frac1m+\frac1{a\theta}\right)^{-1},\qquad
 \overline{\dim}_\theta\mu_z
 \le\left(\frac1m+\frac1{a\theta}\right)^{-1}.
\end{aligned}
\end{equation}
There are full-measure sets of directions and exterior centres on
which equality holds for both upper and lower dimensions,
simultaneously for all $\theta\in(0,1]$.
\end{theorem}

The construction attains both \eqref{eq:two-window-bound} and
\eqref{eq:quantitative-bound}. It also shows that small source
dimensions alone do not ensure preservation: if
\(a<mp/(p-m)\) for finite \(p\), or \(a<m\) for \(p=\infty\), then
\[
\sup_{0<\theta\le1}\dim_\theta\mu
=\left(\frac1p+\frac1a\right)^{-1}<m,
\]
while the typical image dimensions are strictly smaller than the corresponding source dimensions for every \(0<\theta\le1\). Thus the threshold in Corollary~\ref{cor:preserve} is sharp even when the source dimensions lie strictly below \(m\). 

\subsection{Reference measures with prescribed growth}

The next two lemmas supply the reference estimates needed in
Section~\ref{sec:layered}: prescribed growth, small mass at a
distinguished point, and lower mass bounds for orthogonal and
radial images.

\begin{lemma}
\label{lem:sharp-reference}
Let $m\ge1$ be an integer and let $m<p<\infty$.
There exist a probability measure $\sigma$ supported on
$\{x\in\R^{m+1}:1/2\le |x|\le1\}$ and a point
$x_*\in\spt\sigma$ such that
\begin{alignat}{2}
\frac{\sigma(B(x,R))}{\sigma(B(x,u))}
&\le C\left(\frac{R}{u}\right)^p,
&\qquad& x\in\operatorname{supp}\sigma,\ 0<u\le R\le1,
\label{eq:sharp-growth}\\[4pt]
cu^p\le \sigma(B(x,u))
&\le Cu^m,
&\qquad& x\in\operatorname{supp}\sigma,\ 0<u\le1,
\label{eq:sharp-mass}\\[4pt]
\sigma(B(x_*,u))
&\sim u^p,
&\qquad& 0<u\le1.
\label{eq:sharp-point}
\end{alignat}
Moreover, for every $V\in G(m+1,m)$,
\begin{equation}\label{eq:sharp-reference-projection}
 (P_V)_\#\sigma(B(z,u))\ge c u^m,
 \quad z\in P_V(\spt\sigma),\ 0<u\le1.
\end{equation}
All constants depend only on $m,p$.
\end{lemma}

\begin{proof}
\emph{Radial measure construction.}
Let $\beta=p-m>0$. Choose $0<\lambda<1/4$
sufficiently small such that
\[
w=\lambda^\beta<\frac12,
\qquad
0<\beta_0=\frac{\log(1-w)}{\log\lambda}
<\min\{\beta,1/2\}.
\]
Consider the following two contracting similarities,
\begin{equation*}
    T_0(t)=\frac12+\lambda\left(t-\frac12\right),
\qquad
T_1(t)=1+\lambda(t-1).
\end{equation*}
By Hutchinson Theorem~\cite[Appendix~B]{BishopPeres2017}, there exists a unique probability measure $\tau$ on $[1/2,1]$ satisfying
\begin{equation*}
    \tau=w(T_0)_\#\tau+(1-w)(T_1)_\#\tau.
\end{equation*}

For $t\in\operatorname{supp}\tau$, let $I_k(t)$
denote the $k$-level cylinder containing $t$,
with $I_k(t)=[1/2,1]$ for $k\le0$.
Given $0<u\le R\le1$, choose $k\le l$ such that
$\lambda^k\sim R$ and $\lambda^l\sim u$. Fix an integer $c_0\ge2$ sufficiently large that
\[
\left(\frac12-\lambda\right)\lambda^{-c_0-1}>1.
\]
Since every $j$-level cylinder has diameter
$\lambda^j/2$, we have
\[
\operatorname{diam} I_{l+c_0}(t)
=\frac12\lambda^{l+c_0}<u,
\]
and hence $I_{l+c_0}(t)\subset B(t,u)$.
On the other hand, any two distinct $j$-level
cylinders are separated by a distance at least
$(1/2-\lambda)\lambda^{j-1}$.
If $k>c_0$, taking $j=k-c_0$ gives
\[
\operatorname{dist}(I_j(t),I')
\ge\left(\frac12-\lambda\right)\lambda^{k-c_0-1}
>R
\]
for every other $j$-level cylinder $I'$.
Thus $B(t,R)\cap\operatorname{supp}\tau
\subset I_{k-c_0}(t)$.
The same inclusion is immediate when $k\le c_0$,
by our convention $I_{k-c_0}(t)=[1/2,1]$.
Consequently,
\[
B(t,R)\cap\operatorname{supp}\tau
\subset I_{k-c_0}(t),
\qquad
I_{l+c_0}(t)\subset B(t,u).
\]
Since both weights are at least $w=\lambda^\beta$,
the multiplicative structure of cylinder masses gives
\[
\frac{\tau(I_{k-c_0}(t))}
{\tau(I_{l+c_0}(t))}
\le w^{-(l-k+2c_0)}
\lesssim w^{-(l-k)}.
\]
Hence
\begin{equation}\label{eq:sharp-radial-growth}
\frac{\tau(B(t,R))}{\tau(B(t,u))}
\lesssim w^{-(l-k)}
\lesssim\left(\frac Ru\right)^\beta.
\end{equation}
Taking $R=1$ gives
\[
\tau(B(t,u))\gtrsim u^\beta,
\quad t\in\operatorname{supp}\tau.
\]

The constant-symbol cylinders at the endpoints,
together with comparison at adjacent cylinder scales, also give
\begin{equation}\label{eq:sharp-radial-extremes}
    \tau\left([1/2,1/2+u]\right)\sim u^\beta,
\quad
\tau([1-u,1])\sim u^{\beta_0},
\quad 0<u\le\frac12.
\end{equation}

\emph{Source estimates.}
We now use $\tau$ to construct a measure $\sigma$
satisfying \eqref{eq:sharp-growth}--\eqref{eq:sharp-point}. Let $\omega_m$ be normalized surface measure
on $S^m$, and define
\[
\sigma=F_\#(\tau\times\omega_m),
\qquad F(t,\omega)=t\omega.
\]
For $1/2\le t,s\le1$ and $\omega,v\in S^m$, we have
\[
|t\omega-sv|^2=(t-s)^2+ts|\omega-v|^2.
\]
It follows that
\begin{align*}
    F^{-1}\left(B(t\omega,u)\right)&\subseteq B(t,u)\times \left\{v\in\mathbb{S}^m:|\omega-v|<2u\right\},\\
    F^{-1}\left(B(t\omega,u)\right)&\supseteq B(t,u/4)\times \left\{v\in\mathbb{S}^m:|\omega-v|<u/4\right\}.
\end{align*}
Using the $m$-Ahlfors regularity of $\omega_m$ and
\eqref{eq:sharp-radial-growth}, we obtain
\begin{align*}
    \sigma(B(t\omega,u))&=\int_{[1/2,1]}\omega_m\left(\left\{v\in\mathbb{S}^m:|t\omega-sv|<u\right\}\right){\rm d}\tau(s)\\
    &\sim u^m\tau(B(t,u)),
\quad
t\in\operatorname{supp}\tau,\ 0<u\le1.
\end{align*}
Consequently, \eqref{eq:sharp-radial-growth} gives
\[
\frac{\sigma(B(x,R))}{\sigma(B(x,u))}
\lesssim
\left(\frac Ru\right)^{m+\beta}
=\left(\frac Ru\right)^p,
\]
for $x\in\operatorname{supp}\sigma$ and
$0<u\le R\le1$.
Moreover,
\[
u^p\lesssim\sigma(B(x,u))\lesssim u^m,
\quad
x\in\operatorname{supp}\sigma,\quad 0<u\le1.
\]
Taking $x_*=\frac12e_1$, \eqref{eq:sharp-radial-extremes} yields
\[
\sigma(B(x_*,u))
\sim
u^m\tau([1/2,1/2+u])
\sim u^p,
\quad 0<u\le1.
\]

\emph{Orthogonal projections.}
By rotational invariance, it suffices to consider
$V=\mathbb R^m\times\{0\}$.
Using the two hemispheres of $S^m$, the projection
of normalized surface measure on the sphere of
radius $t$ has density
\[
f_t(y)
=c_m t^{-m}
\left(1-\frac{|y|^2}{t^2}\right)^{-1/2}
\mathbf 1_{\{|y|<t\}}.
\]
Hence $(P_V)_\#\sigma$ has density
\[
f(y)=\int_{t>|y|}f_t(y)\,d\tau(t).
\]

For $1/2\le |y|<1$, put $v=1-|y|$.
Since
\[
f_t(y)=c_m t^{1-m}(t^2-|y|^2)^{-1/2}
\]
for $t>|y|$, and $t^2-|y|^2\le2v$ whenever
$1-v\le t\le1$, \eqref{eq:sharp-radial-extremes} gives
\[
\begin{aligned}
f(y)
&\gtrsim v^{-1/2}\tau([1-v,1])
\gtrsim v^{\beta_0-1/2}
\gtrsim1.\quad (\text{use }\beta_0<1/2)
\end{aligned}
\]
For $|y|<1/2$, we have
$f_t(y)\gtrsim1$ for $t\in[1/2,1]$,
so integrating against $\tau$ gives the same lower bound.
Thus
\[
f(y)\ge c>0,\quad |y|<1.
\]
It follows that, for every
$z\in P_V(\operatorname{supp}\sigma)$
and $0<u\le1$,
\[
\begin{aligned}
(P_V)_\#\sigma(B(z,u))
&\ge c\,\mathcal L^m(B(z,u)\cap B(0,1))
\gtrsim u^m.
\end{aligned}
\]
This proves \eqref{eq:sharp-reference-projection}.
\end{proof}

We next establish the radial estimate, including the surface measure
used when the prescribed quasi-Assouad dimension is infinite.

\begin{lemma}
\label{lem:sharp-radial-reference}
Let $\sigma$ be the reference measure of
Lemma~\ref{lem:sharp-reference}, or normalized surface measure
on $S^m$. For every $b\in\R^{m+1}$ with $d=|b|\ge2$,
$\pi_b(x)=(x-b)/|x-b|$ satisfies
\begin{equation}\label{eq:sharp-radial-reference}
 (\pi_b)_\#\sigma(B(e,u))
 \ge c\min\{1,(du)^m\},
 \quad e\in\pi_b(\overline B(0,1)),\ 0<u\le1.
\end{equation}
where $c>0$ is independent of $b$, and balls on the target sphere
use the chord metric.
\end{lemma}

\begin{proof}
Let $d=|b|$ and $v=b/d\in\mathbb{S}^m$. The image
$\pi_b(\overline B(0,1))$ is the spherical cap of
angular radius $\arcsin(1/d)$ centered at $-v$.
Identifying $v^\perp$ with $\R^m$, parametrize this cap by
\[
e_b(y)=-\sqrt{1-\frac{|y|^2}{d^2}}\,v+\frac yd,
\qquad |y|\le1.
\]
Since $d\ge2$, we have
\[
|e_b(y)-e_b(y')|\le\frac{C}{d}|y-y'|,
\qquad |y|,|y'|\le1.
\]

Let $\omega_{m,s}$ denote normalized surface measure
on $sS^m$, where $1/2\le s\le1$.
For $|y|<s$, the ray $b+\lambda e_b(y)$
intersects $sS^m$ at the two points
$x_\pm=b+\lambda_\pm e_b(y)$, where
\[
\lambda_\pm
=\sqrt{d^2-|y|^2}\pm\sqrt{s^2-|y|^2}\sim d.
\]
Writing $n(x)=x/s$ for the unit normal to $sS^m$,
the tangential Jacobian of $\pi_b$ at $x_\pm$ is
\[
J_m(\pi_b|_{sS^m})(x_\pm)
=\frac{|n(x_\pm)\cdot e_b(y)|}{\lambda_\pm^m}
=\frac{\sqrt{s^2-|y|^2}}{s\lambda_\pm^m}.
\]
Moreover, the spherical area element in the chart
$e_b$ is
\[
d\mathcal H^m(e_b(y))
=\frac{d^{-m}}{\sqrt{1-|y|^2/d^2}}\,dy.
\]
Thus, by the area formula, the density of
$(\pi_b)_\#\omega_{m,s}$ in the $y$-coordinates is
\[
\begin{aligned}
g_{b,s}(y)
&=\frac{(\lambda_+^m+\lambda_-^m)d^{-m}}
{\mathcal H^m(S^m)s^{m-1}
\sqrt{s^2-|y|^2}\sqrt{1-|y|^2/d^2}}\,
\mathbf1_{\{|y|<s\}}\\
&\sim
s^{1-m}(s^2-|y|^2)^{-1/2}
\mathbf1_{\{|y|<s\}}\sim_m f_s(y),
\end{aligned}
\]
where $f_s(y)$ is the density obtained in Lemma~\ref{lem:sharp-reference}, uniformly in $b$ and $s$. The tangent directions
form an $\mathcal H^m$-null set. 

If $\sigma$ is the reference measure of Lemma~\ref{lem:sharp-reference}, the density of $(\pi_b)_\#\sigma$ in the $y$-coordinates is:
\[
g_b(y)=\int_{s>|y|}g_{b,s}(y)\,d\tau(s)\sim_m\int_{s>|y|}f_s(y)\,d\tau(s)=f(y).
\]
Hence, we have
\[
g_b(y)\ge c>0,\quad |y|<1.
\]

If $\sigma=\omega_m$, the same
conclusion follows directly from
\[
g_b(y)=g_{b,1}(y)
\gtrsim(1-|y|^2)^{-1/2}\gtrsim1,
\quad |y|<1.
\]

Finally, let $e=e_b(y)$ with $|y|\le1$.
By the Lipschitz estimate for $e_b$, there exists
a fixed $C_0>C$ such that
\[
e_b\bigl(B(y,du/C_0)\cap\overline B^m(0,1)\bigr)
\subset B(e,u).
\]
Consequently,
\[
\begin{aligned}
(\pi_b)_\#\sigma(B(e,u))
&\ge (\pi_b)_\#\sigma\left(e_b\bigl(B(y,du/C_0)\cap\overline B^m(0,1)\bigr)\right)=\int_A g_b(y){\rm d}y\\
&\gtrsim \mathcal L^m
\bigl(B(y,du/C_0)\cap B^m(0,1)\bigr)\gtrsim \min\{1,(du)^m\},
\end{aligned}
\]
where the last inequality follows from the uniform
volume lower bound for intersections of Euclidean balls.
This proves \eqref{eq:sharp-radial-reference}.
\end{proof}

\subsection{Proof of simultaneous sharpness}

We insert these references into the construction of
Section~\ref{sec:layered} and compare the resulting image upper
bounds with the general projection estimates.

\begin{proof}[Proof of Theorem~\ref{thm:sharp-clouds}]

\emph{Source dimensions and orthogonal projections.}
Suppose first that $p<\infty$. Let $\sigma$ and
$x_*\in\spt\sigma$ be given by
Lemma~\ref{lem:sharp-reference}. If $p=\infty$,
let $\sigma$ be normalized surface measure on $S^m$
and take $x_*=0$, so that
$\operatorname{dist}(x_*,\spt\sigma)=1$.

Define $\mu$ by \eqref{eq:sharp-mu}, with $v=e_1$ and
\[
N_j=
\begin{cases}
\lceil R_j^{-a}\rceil, & 0<a<\infty,\\[6pt]
\lceil e^{j^2}\rceil, & a=\infty.
\end{cases}
\]
For $p<\infty$, the growth and marked-point estimates
of Lemma~\ref{lem:sharp-reference} verify the
hypotheses of Lemma~\ref{lem:general-reference-layers}.
For $p=\infty$, its separated marked-point hypothesis
holds by the choice of $x_*$. Thus
Lemma~\ref{lem:general-reference-layers} yields $\dim_{\mathrm{qA}}\mu=p$, $\dim_{\mathrm A}\mu=\infty$ and
\eqref{eq:sharp-source} for both choices of $p$ and $a$.

We next establish the upper bound for orthogonal
projections. Fix $V\in G(m+1,m)$ and let
\[
\nu_j:=(P_V)_\#[(T_j)_\#\eta_{N_j}].
\]
For $p<\infty$, the projection estimate
\eqref{eq:sharp-reference-projection} in
Lemma~\ref{lem:sharp-reference} gives
\begin{equation}\label{eq:reference-projection-mass}
    (P_V)_\#\sigma(B(y,r))\ge cr^m,
\quad
y\in P_V(\spt\sigma),\quad 0<r\le1.
\end{equation}
This includes $y=P_Vx_*$, since
$x_*\in\spt\sigma$. For $p=\infty$,
the projection density of normalized surface measure
on $S^m$ is bounded below on $B(0,1)$.
Hence the same estimate holds at every point of
$P_V(\spt\sigma)$ and at $P_Vx_*=0$.

Observe that
\[
T_j^{-1}\bigl(P_V^{-1}(B(y,r))\bigr)
=
\left\{x\in\R^{m+1}:
|P_VT_j(x)-y|^2<r^2\right\}
\]
is a quadratic sublevel set. Thus Lemma~\ref{lem:relative-clouds} together with \eqref{eq:reference-projection-mass} gives, for every $y\in\spt\nu_j$ and $0<r\le1$,
\[
\begin{aligned}
\nu_j(B(y,r))
&\gtrsim
\frac1{\log(2N_j)}
\max\left\{
N_j^{-1},
\min\left\{1,\left(\frac{r}{\ell_j}\right)^m\right\}
\right\}\\
&\gtrsim
\frac1{\log(2N_j)}
\max\left\{
N_j^{-1},
\min\left\{1,\left(\frac{r}{R_j}\right)^m\right\}
\right\}.
\end{aligned}
\]
This verifies \eqref{eq:sharp-common-lower}
with exponent $m$ and logarithmic exponent $h=1$.

It remains to check the support condition in Lemma~\ref{lem:image-layers}. Since
$\spt\sigma\cup\{x_*\}\subset\overline B(0,1)$,
we have
\[
\begin{aligned}
\spt\nu_j
&\subset P_V\circ T_j(\spt\sigma\cup\{x_*\})\subset B(P_V(R_je_1),C\ell_j)
 \subset B(0,CR_j).
\end{aligned}
\]
Applying Lemma~\ref{lem:image-layers} with exponent $m$ and the corresponding choice of $N_j$, we obtain
\[
\overline{\dim}_t\mu_V
\le\left(\frac1m+\frac1{at}\right)^{-1},
\quad
V\in G(m+1,m),\quad 0<t\le1.
\]

\emph{Radial projections.}
Fix $z\notin\spt\mu$. Since $0\in\spt\mu$, we have $z\ne0$.
For the $j$th layer, let
\[
b_j=\frac{z-R_je_1}{\ell_j}.
\]
Then $\pi_z\circ T_j=\pi_{b_j}$. Since $R_j\to0$ and $\ell_j=R_j/(100j)$, for all
sufficiently large $j$,
\[
|b_j|\ge2,
\qquad
|b_j|\ell_j=|z-R_je_1|\sim_z1.
\]
In particular, $|b_j|\gtrsim_z R_j^{-1}$.

Define the radial image measure of the $j$th layer by
\[
\nu_j:=(\pi_z)_\#[(T_j)_\#\eta_{N_j}]
     =(\pi_{b_j})_\#\eta_{N_j}.
\]
By Lemma~\ref{lem:sharp-radial-reference},
\begin{equation}\label{eq:sharp-radial-mass}
\begin{aligned}
    (\pi_{b_j})_\#\sigma(B(e,u))
&\ge c\min\{1,(|b_j|u)^m\}\gtrsim_z\min\{1,(u/R_j)^m\},
\end{aligned}
\end{equation}
for every $e\in\pi_{b_j}(\overline B(0,1))$
and $0<u\le1$.

Recall that the preimages of spherical caps under
radial projections belong to the class $\mathcal A$
defined before Lemma~\ref{lem:relative-clouds}. Applying Lemma~\ref{lem:relative-clouds} to a
closed cap of radius $u/2$ and \eqref{eq:sharp-radial-mass}, we obtain
\begin{equation}\label{eq:sharp-radial-empirical}
\begin{aligned}
\nu_j(B(e,u))
&\gtrsim_z
\frac1{\log(2N_j)}
\max\left\{
N_j^{-1},
\min\left\{1,\left(\frac{u}{R_j}\right)^m\right\}
\right\},
\end{aligned}
\end{equation}
for every $e\in\spt\nu_j$ and $0<u\le1$.
Here the estimate also covers the image of the marked
point: for $p<\infty$, $x_*\in\spt\sigma$, while for
$p=\infty$, $x_*=0\in\overline B(0,1)$.
Thus \eqref{eq:sharp-radial-empirical} verifies hypothesis \eqref{eq:sharp-common-lower} with exponent $m$ and logarithmic exponent $h=1$.

It remains to verify the support condition in
Lemma~\ref{lem:image-layers}. Since $z\ne0$, for all
sufficiently large $j$,
\[
T_j(\overline B(0,1))
\subset B(0,2R_j)
\subset B(0,|z|/2).
\]
The map $\pi_z$ is Lipschitz on $B(0,|z|/2)$, with
constant depending only on $z$. Hence
\[
\begin{aligned}
\spt\nu_j
&\subset\pi_z(T_j(\spt\sigma\cup\{x_*\}))\subset B(\pi_z(0),C_zR_j).
\end{aligned}
\]
Applying Lemma~\ref{lem:image-layers} with exponent $m$ and
the corresponding choice of $N_j$, we conclude that
\[
\overline{\dim}_t\mu_z
\le\left(\frac1m+\frac1{at}\right)^{-1},
\quad
z\notin\spt\mu,\quad 0<t\le1.
\]

\emph{Equality and optimization.}
By \eqref{eq:sharp-source}, for either source dimension,
\[
 \frac1{\dim_t\mu}-\frac1p=\frac1{at},
 \qquad 0<t\le1,
\]
where $1/\infty=0$. Theorems~\ref{thm:transversal}
and~\ref{thm:radial} therefore give the matching lower
bound
\[
 \dim_t\mu_V,\ \dim_t\mu_z\geq \left(\frac1m+\frac{1}{\dim_t\mu}-\frac{1}{p}\right)^{-1}= \left(\frac1m+\frac1{at}\right)^{-1}
\]
for almost every $V\in G(m+1,m)$ and almost every
$z\notin\spt\mu$, simultaneously for all $t$.
Together with the upper bounds, this proves the asserted
equalities for both upper and lower intermediate dimensions.

Finally, fix a target parameter $t$ and let
$0<\kappa\le1$ be an arbitrary source parameter. The reciprocal term in \eqref{eq:two-window-bound} has denominator
\[
 \begin{aligned}
 \frac1m+\frac{\kappa}{t}
 \left(\frac1{\dim_\kappa\mu}-\frac1p\right)
 =\frac1m+\frac{\kappa}{t}\frac1{a\kappa}=\frac1m+\frac1{at}.
 \end{aligned}
\]
Thus this term is independent of $\kappa$. At $\kappa=t$,
the other entry in the minimum is strictly larger, since
\[
 \dim_t\mu
 =\left(\frac1p+\frac1{at}\right)^{-1}
 >
 \left(\frac1m+\frac1{at}\right)^{-1},
\]
Hence the minimum attains the projected dimension
at $\kappa=t$, proving the sharpness of the
optimized estimate.
\end{proof}

\section{Sharp exceptional-set exponents}\label{sec:third-round}

We next determine how large exceptional sets can be while the source
dimension functions and growth dimension remain fixed. For orthogonal
projections, strict sublevel dimensions can approach the bound and
non-strict sublevels can attain it. Prescribing only a constant lower
source function also allows strict attainment by one measure.
We then treat radial projections in the parameter ranges stated below.

\subsection{Orthogonal directions}

We first treat source functions of the form \eqref{eq:intro-sharpness},
using measures whose source data remain fixed as the exceptional
direction set varies.

\begin{theorem}\label{thm:finite-graded-sharp}
Let $1\le m<n$, $0<\theta\le1$, and
$0<d\le s\le q\le\infty$, with $d<\infty$, and suppose that
\begin{equation}\label{eq:graded-sharp-parameters}
 b=\left(\frac1d-\frac1s+\frac1q\right)^{-1}\le m.
\end{equation}
Let $a=[\theta(1/s-1/q)]^{-1}$ if $s<q$, and $a=\infty$ if $s=q$. The following assertions hold.
\begin{enumerate}
\item[(i)] For every $\varepsilon>0$, there is a probability measure
$\mu_\varepsilon$ with countable compact support in $\R^n$ such that
\begin{equation}\label{eq:graded-sharp-source-data}
 \begin{gathered}
 \qA\mu_\varepsilon=q,\qquad \dA\mu_\varepsilon=\infty,\\
 \underline{\dim}_t\mu_\varepsilon
 =\overline{\dim}_t\mu_\varepsilon
 =\left(\frac1q+\frac1{at}\right)^{-1},\quad0<t\le1.
 \end{gathered}
\end{equation}
and
\begin{equation}\label{eq:finite-graded-optimal}
 \dim_{\mathrm H}\{V\in G(n,m):\overline{\dim}_\theta(\mu_\varepsilon)_V<d\}
 >m(n-m-1)+b-\varepsilon.
\end{equation}
The same bound holds for lower intermediate dimensions.
There is a set $\mathcal V_\varepsilon$ of full $\gamma_{n,m}$-measure such that
\[
 \underline{\dim}_t(\mu_\varepsilon)_V
 =\overline{\dim}_t(\mu_\varepsilon)_V
 =\left(\frac1q+\frac1{at}\right)^{-1},
 \quad V\in\mathcal V_\varepsilon,\ 0<t\le1.
\]

\item[(ii)] If $d<s$, there is a probability measure $\mu_0$
with countable compact support in $\R^n$ satisfying
\eqref{eq:graded-sharp-source-data} with $\mu_\varepsilon$ replaced by
$\mu_0$, and
\begin{equation}\label{eq:exact-nonstrict-exception}
 \begin{aligned}
 \dim_{\mathrm H}\{V:\overline{\dim}_\theta(\mu_0)_V\le d\}
 &=m(n-m-1)+b.
 \end{aligned}
\end{equation}
The same equality holds for lower intermediate dimensions. When $b<m$, there is a set $\mathcal V_0$ of full
$\gamma_{n,m}$-measure such that
\[
\underline{\dim}_t(\mu_0)_V
=\overline{\dim}_t(\mu_0)_V
=\left(\frac1q+\frac1{at}\right)^{-1},
\quad V\in\mathcal V_0,\quad 0<t\le1.
\]
\end{enumerate}
\end{theorem}

Both source functions are independent of $\varepsilon$ and equal $s$
at parameter $\theta$. For each $0<p\le b$, we realize the image
function $t\mapsto(1/p+1/(at))^{-1}$ on a compact direction set
of dimension $m(n-m-1)+p$. Taking $p<b$ gives part~(i), and
taking $p=b<q$ gives part~(ii).

The reference measure below makes points with small source mass share
their images with a regular component along the selected directions.
Sampling then transfers this geometry to the layers.

\begin{lemma}\label{lem:arithmetic-reference}
Let $0<p\le m$, $q\in[p,\infty]$, and $B\ge16$.
There are compact sets $S\subset A\subset\R^m$ with
$0\notin A$ and
\[
 \dim_{\mathrm H}S=\dim_{\mathrm H}A=p,
\]
a probability $\nu_A$ on $A$, and a reference probability
$\sigma_q$ in $\R v\oplus\R^m$, where $v\perp\R^m$ is a
unit vector, with the following properties. The set $S$ carries
a $p$-Ahlfors regular probability, and
\begin{equation}\label{eq:logarithmic-regular-reference}
 \frac{c u^p}{(1+|\log u|)^2}
 \le\nu_A(B(x,u))\le C u^p,
 \quad x\in A,\ 0<u\le1.
\end{equation}
For $q<\infty$, $0\in\spt\sigma_q$, and
\begin{equation}\label{eq:arithmetic-source-balls}
 \frac{\sigma_q(B(x,R))}{\sigma_q(B(x,u))}
 \le C(1+|\log u|)^2(R/u)^q,\qquad
 \sigma_q(B(0,u))\le Cu^q
\end{equation}
for $x\in\spt\sigma_q$ and $0<u\le R\le1$.
For $q=\infty$, $\operatorname{dist}(0,\spt\sigma_q)>0$;
$0$ is adjoined only to the samples.

For every linear map $P$ satisfying
$\ker(P|_{\R v\oplus\R^m})=\R(v+w)$ for some $w\in S$
and having bi-Lipschitz restriction to $\R^m$, one has
\begin{equation}\label{eq:concealed-reference-image}
 (P_\#\sigma_q)(B(z,u))
 \ge\frac{c_Pu^p}{(1+|\log u|)^2}
 \quad\bigl(z\in\spt(P_\#\sigma_q)\cup\{P0\},\ 0<u\le1\bigr).
\end{equation}
There are also $z_P\in\spt(P_\#\sigma_q)$ and $u_P>0$
such that
\begin{equation}\label{eq:concealed-reference-witness}
 (P_\#\sigma_q)(B(z_P,u))\le C_Pu^p,
 \quad 0<u<u_P.
\end{equation}
The lower-bound constant is uniform under uniform bi-Lipschitz bounds.
\end{lemma}

\begin{proof}
\emph{Construction and ball estimates.}
Choose a compact $p$-Ahlfors regular set $S\subset\R^m$ with
$\operatorname{dist}(0,S)>0$, and let $\nu_S$ be a regular probability
on $S$. Set
\[
 c_0=1,\qquad c_j=1-B^{-j}\quad(j\ge1),\qquad
 A=\bigcup_{j\ge0}c_jS,
\]
and define
\[
 \nu_A=Z_1^{-1}\sum_{j\ge0}(j+1)^{-2}(D_{c_j})_\#\nu_S,
 \qquad Z_1=\sum_{j\ge0}(j+1)^{-2},\qquad D_cx=cx.
\]
Since $c_j\to1$ and $S$ is compact, $A$ is compact. Moreover,
\begin{equation}\label{eq:digit-concealment}
 (1-B^{-j})S\subset A\quad(j\ge1),\qquad
 \dim_{\mathrm H}A=\dim_{\mathrm H}S=p.
\end{equation}
The regularity constants of $c_jS$ are uniform in $j$, since
$15/16\le c_j\le1$. Consequently,
\[
 \nu_A(B(x,u))
 \lesssim\sum_{j\ge0}(j+1)^{-2}u^p\lesssim u^p.
\]

For the lower bound, write $x=c_js$ with $s\in S$ and put
$M=\sup_{s\in S}|s|$. If $j=0$, the component on $S$ gives
$\nu_A(B(x,u))\gtrsim u^p$. If $j\ge1$ and $u\ge4MB^{-j}$,
then $|x-s|\le u/4$, and hence
\[
 \nu_A(B(x,u))\gtrsim\nu_S(B(s,u/2))\gtrsim u^p.
\]
If $u<4MB^{-j}$, then $j+1\lesssim1+|\log u|$, so
\[
 \nu_A(B(x,u))
 \gtrsim(j+1)^{-2}\nu_S(B(s,u/c_j))
 \gtrsim\frac{u^p}{(1+|\log u|)^2}.
\]
This proves \eqref{eq:logarithmic-regular-reference}.

Let $\sigma_A$ be the translate of $\nu_A$ to $v+A$.
For $q<\infty$, define
\begin{equation}\label{eq:thin-geometric-reference}
 \tau_q=(1-B^{-q})\sum_{j\ge1}B^{-q(j-1)}\delta_{B^{-j}v},
 \qquad \sigma_q=\tfrac12(\sigma_A+\tau_q).
\end{equation}
Summing the geometric tail gives, uniformly for
$x\in\spt\tau_q$ and $0<u\le1$,
\[
 \tau_q(B(x,u))\sim\max\{|x|,u\}^q.
\]
Indeed, for $u<|x|/2$ the ball contains only the atom at $x$;
for larger $u$, the atom and the geometric tail give the stated
comparison. At $x=0$, use the tail directly. Thus
\begin{align*}
 \tau_q(B(0,u))&\sim u^q,\\
 \frac{\tau_q(B(x,R))}{\tau_q(B(x,u))}
 &\lesssim
 \left(\frac{\max\{|x|,R\}}{\max\{|x|,u\}}\right)^q
 \le(R/u)^q.
\end{align*}
For the other component, \eqref{eq:logarithmic-regular-reference}
gives
\[
 \frac{\sigma_A(B(x,R))}{\sigma_A(B(x,u))}
 \lesssim(1+|\log u|)^2(R/u)^p.
\]
The supports of the two components are separated. Balls meeting
only one component satisfy its growth estimate. If a ball meets
both components, its outer radius is bounded below; the uniform
lower bound $cu^q/(1+|\log u|)^2$, using $p\le q$, then gives
the first estimate in \eqref{eq:arithmetic-source-balls}.
Its second estimate follows from the geometric tail at $0$.
For $q=\infty$, take $\sigma_\infty=\sigma_A$, so that
$\operatorname{dist}(0,\spt\sigma_\infty)>0$.

\emph{Projection estimates.}
Fix $P$ as in the statement, with $P(v+w)=0$, $w\in S$.
For $t_j=B^{-j}$, we have
\begin{equation}\label{eq:exact-projected-inclusion}
 \begin{aligned}
 P\bigl(v+(1-t_j)w\bigr)
 &=Pv+(1-t_j)Pw=t_jPv=P(t_jv),\\
 P(v+w)&=0=P0.
 \end{aligned}
\end{equation}
By \eqref{eq:digit-concealment}, every projected thin point and
the projected marked point belong to $P(v+A)$. The bi-Lipschitz
restriction of $P$ to $\R^m$ and
\eqref{eq:logarithmic-regular-reference} therefore give
\eqref{eq:concealed-reference-image}, with the asserted uniformity.

The compact set
\[
 \{0\}\cup\{B^{-j}Pv:j\ge1\}
\]
is countable, whereas $P(v+A)$ has Hausdorff dimension $p>0$.
Choose $z_P$ in the latter set and outside the former. For all
sufficiently small $u$, the ball $B(z_P,u)$ meets no projected
thin point. Applying the upper bound for $\nu_A$ and the inverse
Lipschitz bound of $P|_{\R^m}$ gives
\[
 (P_\#\sigma_q)(B(z_P,u))\lesssim_P u^p.
\]
This proves \eqref{eq:concealed-reference-witness}.
A fixed dilation places the reference supports in the unit ball.
\end{proof}

\begin{proof}[Proof of Theorem~\ref{thm:finite-graded-sharp}]
\emph{Construction and source dimensions.}
Fix $0<p\le b$. Since $b\le\min\{q,m\}$, we may apply
Lemma~\ref{lem:arithmetic-reference}. Identify
\[
 W=\R v\oplus U\subset\R^n,\qquad\dim U=m,\qquad Z_0=W^\perp,
\]
and let $Q$ be a compact cube with nonempty interior in
$\operatorname{Hom}(U,Z_0)$. For $w\in S$ and $A_0\in Q$, put
\[
 V_{w,A_0}=\{u-\langle w,u\rangle v+A_0u:u\in U\},
 \qquad E_p=\{V_{w,A_0}:w\in S,\ A_0\in Q\}.
\]
On this compact parameter set, the Grassmannian chart is
bi-Lipschitz. Directly from its definition,
\[
 \ker(P_V|_W)=\R(v+w),\qquad P_V|_U:U\longrightarrow V
 \text{ is an isomorphism}.
\]
The isomorphism bounds are uniform on $E_p$. Since $S\times Q$
is Ahlfors regular, we obtain
\begin{equation}\label{eq:exact-compact-direction-set}
 \dim_{\mathrm H}E_p=m(n-m-1)+p,\qquad
 0<\mathcal H^{m(n-m-1)+p}(E_p)<\infty.
\end{equation}
When $n=m+1$, $Q$ is the single point in the zero-dimensional
space $\operatorname{Hom}(U,Z_0)$.

After a fixed dilation of the reference, choose its relative samples
with marked point $0$ and define $\mu$ by \eqref{eq:sharp-mu}, using
\[
 N_j=\begin{cases}
 \lceil R_j^{-a}\rceil,&a<\infty,\\[3pt]
 \lceil e^{j^2}\rceil,&a=\infty.
 \end{cases}
\]
For $q<\infty$, \eqref{eq:arithmetic-source-balls} verifies the
growth and marked-point assumptions of
Lemma~\ref{lem:general-reference-layers}. For $q=\infty$, the
marked point is separated from the reference support. That lemma
therefore gives the countable compact support, both growth
dimensions, and
\[
 \underline{\dim}_t\mu=\overline{\dim}_t\mu
 =\left(\frac1q+\frac1{at}\right)^{-1},
 \quad0<t\le1.
\]

\emph{Exceptional directions.}
Fix $V\in E_p$. Since $w\ne0$ and $P_V|_U$ is injective,
$P_Vv=-P_Vw\ne0$. The reference image satisfies
\eqref{eq:concealed-reference-image} and
\eqref{eq:concealed-reference-witness}, including the image of
the marked point. Lemma~\ref{lem:cloud-transfer} consequently yields
\begin{equation}\label{eq:exact-exceptional-functions}
 \underline{\dim}_t\mu_V=\overline{\dim}_t\mu_V
 =\left(\frac1p+\frac1{at}\right)^{-1},
 \quad V\in E_p,\ 0<t\le1.
\end{equation}

For part~(i), choose $\max\{0,b-\varepsilon\}<p<b$ and call
the resulting measure $\mu_\varepsilon$. Since
\[
 \frac1p+\frac1{a\theta}
 >\frac1b+\frac1s-\frac1q=\frac1d,
\]
we have
\[
 E_p\subset\{V:\overline{\dim}_\theta(\mu_\varepsilon)_V<d\}
 \subset\{V:\underline{\dim}_\theta(\mu_\varepsilon)_V<d\}.
\]
Equation~\eqref{eq:exact-compact-direction-set} proves
\eqref{eq:finite-graded-optimal} for both dimensions.

For part~(ii), $d<s$ implies $b<q$. Take $p=b$ and write
$\mu_0$ for the measure just constructed. Equations
\eqref{eq:exact-compact-direction-set} and
\eqref{eq:exact-exceptional-functions} give the required lower
bound for both non-strict sublevels. If $b<m$, choose $e>d$
sufficiently close to $d$ that $e<s$ and
\[
 b_e=\left(\frac1e-\frac1s+\frac1q\right)^{-1}<m.
\]
Each non-strict sublevel at $d$ is contained in the corresponding
strict sublevel at $e$. Theorem~\ref{thm:exception} therefore gives
an upper bound $m(n-m-1)+b_e$. Letting $e\downarrow d$ yields
\eqref{eq:exact-nonstrict-exception}. If $b=m$, use the ambient
dimension $\dim G(n,m)=m(n-m)$ instead.

\emph{Almost-everywhere preservation.}
It remains to consider $p<m$. Outside a Haar null set, $P_V|_U$
is an isomorphism and $P_Vv\ne0$. Define
\[
 w_V=-(P_V|_U)^{-1}P_Vv.
\]
In the Grassmannian chart parametrized by
$U\times\operatorname{Hom}(U,Z_0)$, the condition $w_V\in A$
has Haar measure zero, because $\dim_{\mathrm H}A=p<m$.
Thus there is a set $\mathcal V$ of full Haar measure such that
\[
 P_Vv\ne0,\qquad P_V|_U\text{ is invertible},\qquad
 0\notin P_V(v+A),\quad V\in\mathcal V.
\]
Fix $V\in\mathcal V$. If $q<\infty$, the regular image stays
a positive distance from $0$, and the geometric thin component gives
\[
 (P_V)_\#\sigma_q(B(0,u))\sim_V u^q
\]
for all sufficiently small $u$. The source lower mass bound and
Lipschitz contraction also give
\[
 (P_V)_\#\sigma_q(B(y,u))
 \gtrsim\frac{u^q}{(1+|\log u|)^2},
 \quad y\in\spt((P_V)_\#\sigma_q).
\]
Applying Lemma~\ref{lem:cloud-transfer} with exponent $q$ proves
preservation of both dimension functions.

If $q=\infty$, $0$ is separated from the reference image.
The projected layers are separated for all sufficiently large $j$,
since their centres are $R_jP_Vv$ and their diameters are
$O(\ell_j)=o(R_j)$. For $a<\infty$, fix $0<t\le1$, $0<\beta<at$,
and $\beta/a<\gamma<t$. At every sufficiently small $r$, choose
$R_j\sim r^\gamma$ and put $y_j=R_jP_Vv$. Then
$r^t=o(\ell_j)$, and all balls $B(y_j,u)$, $r\le u\le r^t$,
contain only the projected marked atom. Hence
\[
 \sup_{r\le u\le r^t}\frac{\mu_V(B(y_j,u))}{u^\beta}
 =\frac{Z^{-1}j^{-2}}{(N_j+1)r^\beta}
 \lesssim(1+|\log r|)^C r^{a\gamma-\beta}\longrightarrow0.
\]
Together with Lipschitz monotonicity, this gives both image
dimensions equal to $at$. For $a=\infty$, use any $0<\gamma<t$;
the last bound is replaced by $Cj^{-2}e^{-j^2}r^{-\beta}\to0$
for every finite $\beta>0$, so both image dimensions are infinite.
The set $\mathcal V$ was defined independently of $t$. Thus all
the preservation assertions hold on the same full-measure set.
\end{proof}

In the non-strict construction, the whole image function is also
realized on a set of exactly the predicted dimension:
\[
 \dim_{\mathrm H}\left\{V:
 \underline{\dim}_t(\mu_0)_V=\overline{\dim}_t(\mu_0)_V
 =\left(\frac1b+\frac1{at}\right)^{-1}
 \text{ for every }0<t\le1\right\}=m(n-m-1)+b.
\]
Indeed, the compact set $E_b$ gives the lower bound, and
\eqref{eq:exact-nonstrict-exception} gives the upper bound.
The compact set $E_b$ has positive Hausdorff measure at this exponent.
Thus the strict inequality in Theorem~\ref{thm:exception} cannot be
replaced by a non-strict one while retaining its Hausdorff-null conclusion.
For these examples, which prescribe both entire source functions,
the strict conclusion remains an $\varepsilon$-sharpness statement.
The next construction attains strict equality with one measure while
prescribing its lower source function.

\paragraph{Exact strict exceptional dimensions for one measure.}
\begin{proposition}[Strict exceptional sets for one measure]
\label{prop:exact-strict-lower}
Let $1\le m<n$, write $g=m(n-m)$, and let $0<S\le m$.
There is a compactly supported probability measure $\mu$ on $\mathbb R^n$
such that
\[
 \underline{\dim}_t\mu=S,\quad0<t\le1
\]
and, for every $0<t\le1$,
\begin{equation}\label{eq:exact-strict-lower}
 \dim_{\mathrm H}\{V\in G(n,m):
       \underline{\dim}_t\mu_V<S\}=g-m+S.
\end{equation}
Each of the sets
\[
 \begin{gathered}
 \{V\in G(n,m):\underline{\dim}_t\mu_V<S
                    \text{ for every }0<t\le1\},\\
 \{V\in G(n,m):\lim_{t\to0^+}\underline{\dim}_t\mu_V<S\}
 \end{gathered}
\]
also has Hausdorff dimension $g-m+S$.
Every exceptional set in the statement has zero
$\mathcal H^{g-m+S}$-measure.
\end{proposition}

\begin{proof}
\emph{Construction and source dimensions.}
We first work in $\R^\ell$, where $\ell=m+1$.
Choose integers $N_k$, each a perfect $k!$th power, with
$N_{k-1}\mid N_k$, $N_k>N_{k-1}^{k^2}$ for $k\ge2$, and
\begin{equation}\label{eq:strict-moran-rapid}
 \frac{\sum_{i<k}\log N_i}{\log N_k}\longrightarrow0.
\end{equation}
Choose $N_1$ and each subsequent $N_k$ large enough that the
integers $M_k$ below are at least $2$. These choices also satisfy
the rapid-growth condition in
\cite[proof of Theorem~5, pp.~390--391]{KaufmanMattila1975}.
Put
\[
 r_0=1,\qquad r_k=N_k^{-\ell/S},\qquad
 M_k=\lfloor N_kr_{k-1}/2\rfloor.
\]
Starting from $[-1,1]^\ell$, replace each level-$(k-1)$ cube $Q$
with centre $c_Q$ by
\[
 Q(a)=c_Q+N_k^{-1}a+[-r_k,r_k]^\ell,
 \qquad a\in\mathbb Z^\ell,\quad |a|_\infty\le M_k.
\]
Taking the successive $N_k$ sufficiently large makes the child
cubes disjoint and contained in their parent. Their centres lie
in $N_k^{-1}\mathbb Z^\ell$. Let $K$ be the compact intersection.

The number of level-$k$ cubes is
\[
 B_k=\prod_{i\le k}(2M_i+1)^\ell=N_k^{\ell+o(1)},
\]
so $\dim_{\mathrm H}K\le S$. To prove the reverse inequality,
give the level-$k$ cubes equal mass $B_k^{-1}$, consistently in
$k$, and call the resulting measure $\nu$. Counting the cubes
meeting a ball gives
\[
 \nu(B(x,u))\lesssim
 \begin{cases}
 B_k^{-1},&r_k\le u\le N_k^{-1},\\[3pt]
 B_{k-1}^{-1}(u/r_{k-1})^\ell,
       &N_k^{-1}\le u<r_{k-1}.
 \end{cases}
\]
For every $\varepsilon>0$, \eqref{eq:strict-moran-rapid} gives
$B_i^{-1}\le r_i^{S-\varepsilon}$ for all sufficiently large $i$.
In the first range, this is at most $u^{S-\varepsilon}$.
In the second range,
\[
 B_{k-1}^{-1}(u/r_{k-1})^\ell
 \le u^{S-\varepsilon}(u/r_{k-1})^{\ell-S+\varepsilon}
 \le u^{S-\varepsilon}.
\]
Frostman's lemma therefore gives $\dim_{\mathrm H}K=S$.

We now assign different masses to the same cubes. For $M\ge2$, set
\[
 d_M(a)=M-|a|_\infty,\qquad
 Z_M=\sum_{|a|_\infty\le M}\frac1{1+d_M(a)},\qquad
 w_M(a)=\frac1{Z_M(1+d_M(a))}.
\]
Counting the boundary layers of $[-M,M]^\ell\cap\mathbb Z^\ell$ yields
\begin{equation}\label{eq:strict-grid-weights}
 Z_M\sim M^{\ell-1}\log M,
 \qquad \min_a w_M(a)\sim(M^\ell\log M)^{-1}.
\end{equation}
Give the child $Q(a)$ the conditional mass $w_{M_k}(a)$, and
denote the resulting probability by $\mu$. Its support is $K$.
If $p_k$ is the smallest level-$k$ mass, then
\begin{equation}\label{eq:strict-min-cylinder}
 p_k=\prod_{i\le k}\min_a w_{M_i}(a)=N_k^{-\ell+o(1)},
 \qquad p_{k-1}=N_k^{-o(1)}.
\end{equation}
Every ball centred on $K$ with radius $2\sqrt\ell r_k$ contains
a whole level-$k$ cylinder. For every $\varepsilon>0$,
\eqref{eq:strict-min-cylinder} gives
$p_k\ge(2\sqrt\ell r_k)^{S+\varepsilon}$ for all sufficiently
large $k$. Hence $\underline{\dim}_1\mu\le S$.
The Besicovitch covering lemma
also gives $\dim_{\mathrm H}K\le\underline{\dim}_t\mu$.
Indeed, if $s>\underline{\dim}_t\mu$, the definition supplies
arbitrarily small $r$ and, for every $x\in K$, a radius
$u_x\in[r,r^t]$ with $\mu(B(x,u_x))\ge cu_x^s$.
A covering subfamily with bounded overlap then satisfies
\[
 \sum_x(2u_x)^s\lesssim c^{-1}\sum_x\mu(B(x,u_x))
 \lesssim c^{-1}.
\]
Its largest diameter tends to zero with $r$, so
$\dim_{\mathrm H}K\le s$. Letting $s\downarrow
\underline{\dim}_t\mu$ proves the comparison.
Using monotonicity in $t$, we obtain
\begin{equation}\label{eq:strict-lower-source}
 S\le\underline{\dim}_t\mu\le\underline{\dim}_1\mu\le S,
 \quad0<t\le1.
\end{equation}

\emph{Masses on lattice fibres.}
Let $v\in\mathbb Z^\ell\setminus\{0\}$, put $H=|v|_\infty$, and fix
$a\in[-M,M]^\ell\cap\mathbb Z^\ell$. We claim that
\begin{equation}\label{eq:strict-grid-fibre}
 \sum_{j:\,|a+jv|_\infty\le M}w_M(a+jv)
 \gtrsim\frac1{HM^{\ell-1}\log M}.
\end{equation}
Write $D=1+d_M(a)$. If $D\le4H$, the term at $j=0$ is at least
$1/(4HZ_M)$. If $D>4H$, there are at least $cD/H$ integers
with $|j|H\le D/2-1$. For these integers,
\[
 |a+jv|_\infty\le M,\qquad
 1+d_M(a+jv)\le3D/2.
\]
Their total mass is at least
$(cD/H)(2/(3DZ_M))\gtrsim1/(HZ_M)$.
Equation~\eqref{eq:strict-grid-weights} proves the claim in both cases.

\emph{Exceptional slopes.}
For $y\in[0,1]^m$, let
\[
 L_y(x_1,x')=x'-x_1y.
\]
Fix $0<h<1$. Let $A_h$ be the set of $y\in[0,1]^m$ for which
infinitely many $k$ admit $1\le H\le N_k^h$ and $b\in\mathbb Z^m$ with
\begin{equation}\label{eq:strict-rational-approx}
 |y-b/H|_\infty\le r_k.
\end{equation}
The Diophantine estimate in
\cite[proof of Theorem~5, pp.~390--391]{KaufmanMattila1975} gives
\begin{equation}\label{eq:strict-rational-dimension}
 \dim_{\mathrm H}A_h\ge hS.
\end{equation}
For $S<m$, use their ambient dimension $\ell$, source dimension
$S$, and approximation parameter $\delta=1-h$. Their conditions
are $k<H<N_k^h$ and
$\operatorname{dist}(Hy_i,\mathbb Z)<H N_k^{-\ell/S}$
for each coordinate. Thus their slope set is contained in the
unrestricted version of $A_h$, and their dimension bound is
$S(1-\delta)=hS$. For $S=m$, choose
a rational $h<\lambda<1$ and apply the same result with
\[
 \widetilde N_k=N_k^\lambda,\qquad
 \widetilde S=\lambda m<m,\qquad \widetilde h=h/\lambda.
\]
The perfect-power condition makes $\widetilde N_k$ integers for
all large $k$, and
\[
 \widetilde N_k^{-\ell/\widetilde S}=r_k,\qquad
 \widetilde N_k^{\widetilde h}=N_k^h,\qquad
 \widetilde h\widetilde S=hm.
\]
This proves the endpoint. Restriction to $[0,1]^m$ preserves the
dimension, since the unrestricted approximation set is invariant
under integer translations.

Fix $y\in A_h$ and a level $k$ satisfying
\eqref{eq:strict-rational-approx}. For $x\in K$, let $Q(a)$ be
its level-$k$ cube. The centres of the siblings $Q(a+j(H,b))$
have the same $L_{b/H}$ image. Since $K$ is bounded and
$|y-b/H|_\infty\le r_k$, all these cylinders project into
$B(L_yx,Cr_k)$. Also $|(H,b)|_\infty\le2H$. Hence
\begin{align*}
 (L_y)_\#\mu(B(L_yx,Cr_k))
 &\gtrsim p_{k-1}
       \sum_{j:\,|a+j(H,b)|_\infty\le M_k}w_{M_k}(a+j(H,b))\\
 &\gtrsim\frac{p_{k-1}}{HM_k^{\ell-1}\log M_k}
 \ge N_k^{-(\ell-1+h)-o(1)}.
\end{align*}
The estimate is uniform in $x$, and holds for infinitely many
$k$. Definition~\ref{def:intermediate} gives
\begin{equation}\label{eq:strict-lower-image}
 \underline{\dim}_t((L_y)_\#\mu)
 \le\underline{\dim}_1((L_y)_\#\mu)
 \le\frac{S(\ell-1+h)}\ell<S,
 \quad y\in A_h,\ 0<t\le1.
\end{equation}
The same strict upper bound holds for the limit as $t\downarrow0$.

\emph{Embedding and upper bounds.}
Identify $\R^\ell$ with $W=\R v\oplus U\subset\R^n$, where
$\dim U=m$, and put $Z=W^\perp$. Let $\mathcal Q$ be a compact
cube with nonempty interior in $\operatorname{Hom}(U,Z)$, and set
\[
 V_{y,A}=\{u-\langle y,u\rangle v+Au:u\in U\},
 \qquad y\in[0,1]^m,\quad A\in\mathcal Q.
\]
This is a bi-Lipschitz Grassmannian chart on the compact domain.
Moreover, $P_{V_{y,A}}|_W=J_{y,A}\circ L_y$ for a linear
isomorphism $J_{y,A}$, so the intermediate dimensions agree.
By \eqref{eq:strict-rational-dimension}, the image of
$A_h\times\mathcal Q$ has dimension at least
\[
 hS+m(n-m-1).
\]
It lies in every exceptional set in the statement. Letting
$h\uparrow1$ through a countable sequence gives the common
lower bound $m(n-m-1)+S=g-m+S$.

For each fixed $t$, \eqref{eq:strict-lower-source} gives
$\kappa_S=0$ and $b_*=S$ in Theorem~\ref{thm:exception}.
Its lower-dimensional version gives zero
$\mathcal H^{g-m+S}$-measure for the fixed-$t$ exceptional set.
The simultaneous exceptional set is contained in the one for
$t=1$. Finally, Corollary~\ref{cor:limit-exception} gives the
same Hausdorff-measure conclusion for the limiting exceptional
set. These upper bounds complete the proof.
\end{proof}

\begin{remark}
Proposition~\ref{prop:exact-strict-lower} concerns the lower dimension
function, the constant source curve $S$, and the strict preservation
threshold $d=S$. It does not prescribe the upper source function or the
quasi-Assouad dimension. The construction therefore establishes
single-measure strict equality in this regime; the simultaneous
upper-and-lower harmonic examples of Theorem~\ref{thm:finite-graded-sharp}
retain their stated $\varepsilon$-sharpness for strict level sets.
\end{remark}

\paragraph{A nonharmonic source function.}

The source functions in Theorem~\ref{thm:finite-graded-sharp} give
$b_*=b$. To obtain a strict improvement, we modify the thin component
so that the source dimension function is capped at a prescribed value.
This also allows constant source functions below the growth dimension.

\begin{lemma}[Capped layers]\label{lem:capped-layers}
Let $1\le m<n$, $0<e<\infty$, $e<q\le\infty$,
$0<p\le\min\{q,m\}$ and $0<a\le\infty$, and suppose that
$\left(\frac1{p}+\frac1{a}\right)^{-1}\le e$.
Let $S,A,\nu_A,v$ and the compact family $E_p\subset G(n,m)$
be as in Lemma~\ref{lem:arithmetic-reference} and the proof of
Theorem~\ref{thm:finite-graded-sharp}.
There is a probability measure $\mu$ with countable compact support
in $\R^n$ such that
\[
 \dim_{\mathrm{qA}}\mu=q,\qquad
 \dim_{\mathrm A}\mu=\infty,
\]
and, simultaneously for $0<t\le1$,
\begin{equation}\label{eq:capped-source}
 \underline{\dim}_t\mu=\overline{\dim}_t\mu
 =\min\left\{\left(\frac1{q}+\frac1{at}\right)^{-1},e\right\}.
\end{equation}
For every $V\in E_p$, simultaneously for $0<t\le1$,
\begin{equation}\label{eq:capped-projection}
 \underline{\dim}_t\mu_V=\overline{\dim}_t\mu_V=\left(\frac1{p}+\frac1{at}\right)^{-1}.
\end{equation}
The set $E_p$ has positive finite
$\mathcal H^{m(n-m-1)+p}$-measure.
For $a=\infty$, both source functions are identically $e$ and
both image functions on $E_p$ are identically $p$. This holds for
every $0<p\le\min\{e,m\}$ and prescribed $e<q\le\infty$.
\end{lemma}

\begin{proof}
\emph{Construction.}
Use the layer maps in \eqref{eq:sharp-layers}, with axis $v$,
and decrease $\ell_j$ by a fixed factor if necessary to fit the
reference supports. Put
\[
 N_j=\begin{cases}
 \lceil R_j^{-a}\rceil,&a<\infty,\\[3pt]
 \lceil e^{j^2}\rceil,&a=\infty,
 \end{cases}
 \qquad M_j=(N_j+1)^{-1}.
\]
Let $\eta^A_{N_j}$ be relative samples of the translate
$\sigma_A$ of $\nu_A$ to $v+A$, with a fixed marked point in
$v+S$. These samples are chosen for the common quadratic range
class of Lemma~\ref{lem:relative-clouds}.

Set $F_j(1)=1$ and, for $0\le\rho<1$, define
\[
 F_j(\rho)=\begin{cases}
 \max\{\rho^q,\ell_j^e\rho^e\},&q<\infty,\\[3pt]
 \ell_j^e\rho^e,&q=\infty.
 \end{cases}
\]
With the same $B\ge16$ as in Lemma~\ref{lem:arithmetic-reference}, put
\[
 k_j=\min\{k\ge1:F_j(B^{-k})\le M_j\}
\]
and
\[
 \lambda_j=\frac1{1+M_j}\left[
 \bigl(M_j+F_j(B^{-k_j})\bigr)\delta_0+
 \sum_{k=1}^{k_j}
 \bigl(F_j(B^{-(k-1)})-F_j(B^{-k})\bigr)\delta_{B^{-k}v}
 \right].
\]
The coefficients telescope, so $\lambda_j$ is a probability.
For $1\le k\le k_j$,
\[
 F_j(B^{-k})\le B^{-e}F_j(B^{-(k-1)}),\qquad
 F_j(B^{-(k-1)})>M_j.
\]
Thus every atom has mass at least $cM_j$, while the atom at $0$
has mass between $M_j/(1+M_j)$ and $2M_j/(1+M_j)$.
Define
\[
 \omega_j=\tfrac12(\eta^A_{N_j}+\lambda_j),\qquad
 \mu=Z^{-1}\sum_{j\ge2}j^{-2}(T_j)_\#\omega_j.
\]
Each layer is finite, every layer atom has mass at least $c/N_j$,
and the separation and origin-tail estimates
\eqref{eq:layer-sep} and \eqref{eq:ball-0} remain valid.
In particular, $\mu$ has countable compact support.

\emph{Ball estimates for the thin component.}
For $q<\infty$, geometric spacing gives
\begin{equation}\label{eq:capped-thin-balls}
 \lambda_j(B(z,u))
 \sim\max\{M_j,F_j(\max\{u,|z|\})\}
 \quad(z\in\spt\lambda_j,\ 0<u\le1).
\end{equation}
For $u<|z|/2$, use the atom at $z$; for larger $u$, sum the
geometric tail. At $z=0$, the added atom is comparable to the
remaining tail, since
$B^{-q}M_j\le F_j(B^{-k_j})\le M_j$.
Writing $x_j^*=T_j0$, scaling gives
\begin{equation}\label{eq:capped-physical-balls}
 (T_j)_\#\lambda_j(B(x_j^*,u))
 \sim\max\{N_j^{-1},(u/\ell_j)^q,u^e\},
 \quad0<u\le\ell_j,\ q<\infty.
\end{equation}
For $q=\infty$, the first thin atom has mass comparable to one.
Below its scale, the same tail calculation gives
\begin{equation}\label{eq:capped-infinite-balls}
 (T_j)_\#\lambda_j(B(x_j^*,u))
 \sim\max\{N_j^{-1},u^e\},
 \quad0<u<\ell_j/B.
\end{equation}
For all $q$, we also have
\begin{equation}\label{eq:capped-uniform-e}
 (T_j)_\#\lambda_j(B(T_jz,u))
 \gtrsim\max\{N_j^{-1},u^e\},
 \quad z\in\spt\lambda_j,\ 0<u\le\ell_j.
\end{equation}
Indeed, if $u<\ell_j|z|$, the centre atom gives the $u^e$ bound.
If $u\ge\ell_j|z|$, its ball contains the positive-axis part
of the marked-point ball of radius $u$. The atomic lower bound
supplies the other term.

\emph{Growth dimensions.}
Suppose first that $q<\infty$. Since $e<q$,
\[
 F_j(R)\le(R/u)^qF_j(u),\quad0<u\le R\le1.
\]
Equation~\eqref{eq:capped-thin-balls} therefore gives a uniform
$q$-growth estimate for $\lambda_j$. For $\eta^A_{N_j}$, use
\eqref{eq:logarithmic-regular-reference} and relative sampling.
The two reference supports are separated and $p\le q$.
Consequently their mixture satisfies a $q$-growth estimate, with
only polynomial losses in $j$ and $1+|\log u|$.

To pass to $\mu$, take $x$ in layer $j$. As in the proof of
Lemma~\ref{lem:general-reference-layers}, the four ranges give
\[
 \frac{\mu(B(x,R))}{\mu(B(x,u))}\lesssim
 \begin{cases}
 j^C(1+|\log u|)^C(R/u)^q,&R<R_j/4,\\[3pt]
 j^C(1+|\log u|)^C(R_j/u)^q,&u<R_j/4\le R,\\[3pt]
 j^2,&R_j/4\le u\le4R_j,\\[3pt]
 1+|\log u|,&u>4R_j.
 \end{cases}
\]
The second estimate uses the same-layer lower mass bound; the
third uses the whole layer, and the fourth uses the origin tail.
Since $j\lesssim1+|\log u|$ in the first three ranges,
\[
 \frac{\mu(B(x,R))}{\mu(B(x,u))}
 \lesssim(1+|\log u|)^C(R/u)^q,
 \quad x\in\spt\mu,\ 0<u\le R<r_0.
\]
At $x=0$, use \eqref{eq:ball-0}. Absorbing the logarithmic
factor at $u\le R^{1+\delta}$ proves
$\dim_{\mathrm{qA}}\mu\le q$.

For the reverse inequality, choose small $\delta,\xi>0$ with
\[
 (1+\xi)(1+\delta)<\frac q{q-e},\qquad
 q\bigl((1+\xi)(1+\delta)-1\bigr)<a,
\]
where the second condition is automatic for $a=\infty$.
Put $\rho_j=R_j^{1+\xi}$ and $u_j=\rho_j^{1+\delta}$.
The displayed inequalities ensure that the $q$-term dominates
both other terms in \eqref{eq:capped-physical-balls}, at each
of the two radii. Neither ball meets another layer or the regular
component. Hence
\[
 \frac{\mu(B(x_j^*,\rho_j))}{\mu(B(x_j^*,u_j))}
 \sim\frac{(\rho_j/\ell_j)^q}{(u_j/\ell_j)^q}
 =(\rho_j/u_j)^q.
\]
This works for every sufficiently small $\delta>0$ and proves
$\dim_{\mathrm{qA}}\mu=q$.

If $q=\infty$, use $\rho_j=2\ell_j/B$ and
$u_j=\rho_j^{1+\delta}$. The outer ball contains the first thin
atom, whereas \eqref{eq:capped-infinite-balls} applies to the
inner ball. Thus
\[
 \frac{\mu(B(x_j^*,\rho_j))}{\mu(B(x_j^*,u_j))}
 \gtrsim\frac1{N_j^{-1}+u_j^e},\qquad
 H(\delta)\ge\frac{\min\{a,e(1+\delta)\}}\delta.
\]
For exponential counts, the minimum is $e(1+\delta)$.
Letting $\delta\downarrow0$ gives infinite quasi-Assouad dimension.
Finally, for every $q$,
\[
 \mu(B(x_j^*,R_j/4))=Z^{-1}j^{-2},\qquad
 \mu(B(x_j^*,2R_j))\gtrsim j^{-1}.
\]
The radius ratio is fixed and the mass ratio tends to infinity,
so $\dim_{\mathrm A}\mu=\infty$.

\emph{Upper bounds for the source dimensions.}
On the thin component, \eqref{eq:capped-uniform-e}, the mass of
a whole layer, and the origin tail give, respectively,
\[
 \mu(B(x,u))\gtrsim
 \begin{cases}
 j^{-2}u^e,&u<\ell_j,\\[3pt]
 j^{-2},&\ell_j\le u\le4R_j,\\[3pt]
 (1+|\log u|)^{-1},&u>4R_j.
 \end{cases}
\]
These bounds imply upper Minkowski dimension at most $e$ for
the thin component. For finite $q$, its lower $q$-power layer
bound and Lemma~\ref{lem:image-layers} also give the upper bound
$(1/q+1/(at))^{-1}$ at parameter $t$.

For $q=\infty$ and $a<\infty$, fix small $r>0$. If
$R_j\ge r^t/C_0$, the atom at any point of layer $j$ gives
$\mu(B(x,r))\gtrsim(1+|\log r|)^{-2}r^{at}$.
If $R_j<r^t/C_0$, take radius $r^t$ and use the origin tail.
This proves the additional upper bound $at$.

The regular sampled component has upper dimension at most
$(1/p+1/(at))^{-1}$ by Lemma~\ref{lem:image-layers}, and
\[
 \left(\frac1p+\frac1{at}\right)^{-1}
 \le\min\left\{\left(\frac1q+\frac1{at}\right)^{-1},e\right\}.
\]
Adding the two components proves the upper bound in
\eqref{eq:capped-source} for every $q$ and $a$.

\emph{Lower bounds for the source dimensions.}
Suppose first that $q,a<\infty$, fix $0<t\le1$, and put
\[
 d=\left(\frac1q+\frac1{at}\right)^{-1}.
\]
At every sufficiently small $r$, choose $R_j\sim r^{d/a}$.
Since $d/a<t$, all balls about $x_j^*$ with radii in $[r,r^t]$
miss the regular component and the other layers. For
$0<s<\min\{d,e\}$, \eqref{eq:capped-physical-balls} gives
\begin{align*}
 \sup_{r\le u\le r^t}\frac{\mu(B(x_j^*,u))}{u^s}
 &\lesssim(1+|\log r|)^C
 \left(r^{d-s}+r^{t(q-s)-qd/a}+r^{t(e-s)}\right)\\
 &=(1+|\log r|)^C
 \left(r^{d-s}+r^{t(d-s)}+r^{t(e-s)}\right)\longrightarrow0.
\end{align*}
Here $q(t-d/a)=td$.

For $q<\infty$, $a=\infty$, fix $0<s<e$, choose
$0<c<t(q-s)/q$, and take $R_j\sim r^c$. The same calculation
gives the bound
\[
 C(1+|\log r|)^C
 \left(e^{-j^2}r^{-s}+r^{t(q-s)-cq}+r^{t(e-s)}\right)
 \longrightarrow0.
\]
For $q=\infty$, $a<\infty$, fix $0<s<\min\{at,e\}$ and
$s/a<c<t$. Taking $R_j\sim r^c$ and applying
\eqref{eq:capped-infinite-balls} gives instead
\[
 C(1+|\log r|)^C
 \left(r^{ca-s}+r^{t(e-s)}\right)\longrightarrow0.
\]
For $q=a=\infty$, take $0<s<e$ and $0<c<t$; the first term
is $e^{-j^2}r^{-s}$. In every case $j\sim|\log r|$, all
displayed powers have positive exponents, and the estimates
hold at every sufficiently small $r$. They therefore give the
lower dimension bound in \eqref{eq:capped-source}.

\emph{Orthogonal projections.}
Fix $V\in E_p$, corresponding to $w\in S$. The identities
\[
 P_V(B^{-k}v)=P_V\bigl(v+(1-B^{-k})w\bigr),\qquad
 P_V0=P_V(v+w)
\]
place every projected thin point in $P_V(v+A)$. For a ball
centred there, the regular reference gives mass at least
$c_Vu^p/(1+|\log u|)^2$. If
$u^p\ge C_V[\log(2N_j)]^3/N_j$, this exceeds the relative
sampling threshold $C\log(2N_j)/N_j$, after increasing $C_V$.
For the remaining radii, the mass $c/N_j$ at every support
point suffices. Thus, including at images of thin points,
\[
 (P_V)_\#\omega_j(B(y,u))
 \gtrsim_V[\log(2N_j)]^{-3}\max\{N_j^{-1},u^p\},
 \quad y\in\spt((P_V)_\#\omega_j),\ 0<u\le1.
\]
After scaling, this verifies the hypothesis of
Lemma~\ref{lem:image-layers} with exponent $p$, and proves the
upper bound in \eqref{eq:capped-projection}.

For the lower bound, choose $y_*$ in $P_V(v+A)$ and outside
the compact countable set $P_V(\{0\}\cup\{B^{-k}v:k\ge1\})$.
Choose $\rho_V>0$ so that $B(y_*,4\rho_V)$ misses this set.
The regular reference gives $B(y_*,\rho_V)$ positive mass,
so the lower sampling estimate supplies a regular sample point
$y_{j,V}$ in this ball for every sufficiently large $j$.
The upper $p$-power bound for the reference holds at every
centre. Therefore the upper sampling estimate gives
\[
 (P_V)_\#\omega_j(B(y_{j,V},u))
 \lesssim_V u^p+\log(2N_j)/N_j,
 \quad0<u<\rho_V.
\]
These balls miss the thin component uniformly in $j$.
Since $P_Vv\ne0$, the projected layers are
separated for all sufficiently large $j$, including from the
finitely many earlier image points, as in
Lemma~\ref{lem:cloud-transfer}. Put
$z_j=R_jP_Vv+\ell_jy_{j,V}$.
For finite $a$, set $d=(1/p+1/(at))^{-1}$ and choose
$R_j\sim r^{d/a}$ at every small $r$. Since $d/a<t$, the
balls $B(z_j,u)$, $r\le u\le r^t$, meet only layer $j$ and
miss its thin component. For $0<s<d$, the sample estimate gives
\begin{align*}
 \sup_{r\le u\le r^t}\frac{\mu_V(B(z_j,u))}{u^s}
 &\lesssim_V(1+|\log r|)^C
       \left(N_j^{-1}r^{-s}+\ell_j^{-p}r^{t(p-s)}\right)\\
 &\lesssim_V(1+|\log r|)^C
       \left(r^{d-s}+r^{t(d-s)}\right)\longrightarrow0,
\end{align*}
where $p(t-d/a)=td$.
For $a=\infty$, fix $0<s<p$ and $0<c<t(p-s)/p$, and take
$R_j\sim r^c$. The corresponding bound is
\[
 C_V(1+|\log r|)^C
 \left(e^{-j^2}r^{-s}+r^{t(p-s)-cp}\right)\longrightarrow0.
\]
Both estimates hold at every sufficiently small $r$. They prove
the matching lower dimensions, and hence
\eqref{eq:capped-projection} for all $0<t\le1$.
\end{proof}

Taking $a=\infty$ in the preceding lemma yields constant source
functions below the prescribed growth dimension. We first compute
their exceptional sets.

\begin{corollary}[Constant source values below the growth dimension]
\label{cor:capped-constant-sharp}
Let $1\le m<n$, $0<e<q\le\infty$, and $0<d\le\min\{e,m\}$.
For every $0<p<d$, there is a probability measure $\mu_p$ with countable
compact support in $\R^n$ such that
\[
 \dim_{\mathrm{qA}}\mu_p=q,\qquad\dim_{\mathrm A}\mu_p=\infty,
 \qquad\underline{\dim}_t\mu_p=\overline{\dim}_t\mu_p=e,
 \quad0<t\le1,
\]
and
\[
 \begin{aligned}
 \dim_{\mathrm H}\{V:\overline{\dim}_t(\mu_p)_V<d\}
 &=m(n-m-1)+p,\\
 \dim_{\mathrm H}\{V:\underline{\dim}_t(\mu_p)_V<d\}
 &=m(n-m-1)+p,
 \end{aligned}
 \quad0<t\le1.
\]
The same dimensions hold when the image dimension is replaced by
its limit as $t\downarrow0$. These exceptional sets contain a common
compact set of positive finite $\mathcal H^{m(n-m-1)+p}$-measure
on which both image functions are identically $p$.

For every $0<t\le1$, both non-strict sublevels at $p$, and the set
where both entire image functions are identically $p$, have Hausdorff
dimension $m(n-m-1)+p$ and positive Hausdorff measure at that exponent.
Thus the entire-source exceptional exponent $d$ is sharp with the
constant source function $e$ and the growth dimension $q>e$ fixed.
\end{corollary}

\begin{proof}
\emph{Construction and lower bounds.}
Apply Lemma~\ref{lem:capped-layers} with $a=\infty$. Its
hypotheses hold because $p<d\le e$. The resulting measure has
the required source data, and
\[
 \underline{\dim}_t(\mu_p)_V
 =\overline{\dim}_t(\mu_p)_V=p,
 \quad V\in E_p,\ 0<t\le1.
\]
Since $p<d$, $E_p$ lies in all the strict sublevels, their
limiting versions, the non-strict sublevels at $p$, and the
common equality set. Its positive finite
$\mathcal H^{m(n-m-1)+p}$-measure gives every asserted lower bound.

\emph{Upper bounds.}
Set
\[
 \mathcal B_p=\{V:P_Vv=0\}\cup\{V:0\in P_V(v+A)\}.
\]
For each nonzero $x\in\R^n$, the family $\{V:P_Vx=0\}$ is
isometric to $G(n-1,m)$ and has dimension $m(n-m-1)$.
Since $A=\bigcup_{j\ge0}c_jS$, the second family is a countable
union of bundles over the $p$-regular sets $c_jS$, each with
fibre $G(n-1,m)$. Since each $c_jS$ is Ahlfors regular,
its product with a compact smooth fibre has dimension
$p+m(n-m-1)$. Countably many smooth local trivializations give
\[
 \dim_{\mathrm H}\mathcal B_p\le m(n-m-1)+p.
\]

Fix $V\notin\mathcal B_p$. Then $P_Vv\ne0$ and
$\operatorname{dist}(0,P_V(v+A))>0$. In every sufficiently
late layer, small balls about $y_j=R_jP_Vv$ therefore meet only
the projected thin component. For $q<\infty$, its marked-point
bound is
\[
 (\mu_p)_V(B(y_j,u))
 \lesssim_V j^{-2}\left(e^{-j^2}+(u/\ell_j)^q+u^e\right).
\]
Fix $0<t\le1$ and $0<s<e$, choose $0<c<t(q-s)/q$, and take
$R_j\sim r^c$ at each small $r$. As in the preceding proof,
\[
 \sup_{r\le u\le r^t}
 \frac{(\mu_p)_V(B(y_j,u))}{u^s}
 \lesssim_V(1+|\log r|)^C
 \left(e^{-j^2}r^{-s}+r^{t(q-s)-cq}+r^{t(e-s)}\right)
 \longrightarrow0.
\]
For $q=\infty$, take any $0<c<t$ and use
\eqref{eq:capped-infinite-balls}; the middle term is omitted.
Thus every lower intermediate dimension of the image is at
least $e$. Lipschitz monotonicity gives
\[
 \underline{\dim}_t(\mu_p)_V
 =\overline{\dim}_t(\mu_p)_V=e,
 \quad V\notin\mathcal B_p,\ 0<t\le1.
\]
The limits as $t\downarrow0$ also equal $e$.
Since $p<d\le e$, all sets in the statement are contained in
$\mathcal B_p$. This proves their upper bounds. Finally,
letting $p\uparrow d$ proves sharpness while keeping $e$ and
$q$ fixed.
\end{proof}

We next combine a capped construction with an uncapped one. Their
source functions cross, making information at an earlier parameter
strictly improve the exceptional exponent.

\begin{corollary}[Attainment of the improved exceptional exponent]
\label{cor:nonharmonic-sharp}
Let $1\le m<n$ and write $k=m(n-m-1)$.
For every $0<p\le2/3$ there is a probability measure with countable compact
support $\mu_p$ such that
$\dim_{\mathrm{qA}}\mu_p=2$,
$\dim_{\mathrm A}\mu_p=\infty$, and both source functions equal
\begin{equation}\label{eq:nonharmonic-source}
 D(t)=
 \max\left\{\min\left\{\frac{2t}{1+t},\frac12\right\},
                    \frac{2t}{2+t}\right\}
 =
 \begin{cases}
  2t/(1+t),&0<t\le1/3,\\
  1/2,&1/3<t\le2/3,\\
  2t/(2+t),&2/3<t\le1.
 \end{cases}
\end{equation}
There is a compact set $E_p\subset G(n,m)$ with
$0<\mathcal H^{k+p}(E_p)<\infty$ on which both image functions are
\begin{equation}\label{eq:nonharmonic-image}
 \underline{\dim}_t(\mu_p)_V
 =\overline{\dim}_t(\mu_p)_V
 =\left(\frac1p+\frac1{2t}\right)^{-1},
 \quad0<t\le1.
\end{equation}
If $p<2/3$, then
\[
 \begin{aligned}
 \dim_{\mathrm H}\{V:\overline{\dim}_1(\mu_p)_V<1/2\}&=k+p,\\
 \dim_{\mathrm H}\{V:\underline{\dim}_1(\mu_p)_V<1/2\}&=k+p.
 \end{aligned}
\]
For $p=2/3$, both non-strict sublevels at $1/2$ have Hausdorff
dimension $k+2/3$ and positive $\mathcal H^{k+2/3}$-measure.
The set where both image functions equal
$t\mapsto(1/p+1/(2t))^{-1}$ for every $0<t\le1$
has dimension $k+p$ and positive $\mathcal H^{k+p}$-measure
for every $0<p\le2/3$.

At $\theta=1$ and $d=1/2$, one has $D(1)=2/3$ and
\[
 b=1,\qquad \kappa_d=1/3,\qquad b_*=2/3.
\]
Consequently the strictly improved bound
\eqref{eq:spectral-graded-exception} is sharp for this fixed
nonharmonic source function.
\end{corollary}

\begin{proof}
\emph{Construction and dimension functions.}
Apply Lemma~\ref{lem:capped-layers} with $(q,a,e)=(2,2,1/2)$.
The condition on $p$ is exactly
\[
 \left(\frac1p+\frac12\right)^{-1}\le\frac12
 \quad\Longleftrightarrow\quad p\le\frac23.
\]
Call the resulting measure $\mu^{(1)}$. With the same
$S,A,v,E_p$, use the uncapped construction of
Theorem~\ref{thm:finite-graded-sharp} with $(q,a)=(2,1)$ to
obtain $\mu^{(2)}$. Put
\[
 \mu_p=\tfrac12\mu^{(1)}+
       \tfrac12(x\mapsto x+Lv)_\#\mu^{(2)}.
\]
Choose $L$ sufficiently large. Since
$\inf_{V\in E_p}|P_Vv|>0$, the two source supports and their
two projected supports are separated for every $V\in E_p$.

On separated supports, the upper and lower dimension functions
of a sum agree with the maximum when both dimensions agree for
each summand. The same maximum rule holds for the growth
dimensions. Consequently,
\begin{align*}
 \underline{\dim}_t\mu_p=\overline{\dim}_t\mu_p
 &=\max\left\{
       \min\left\{\frac{2t}{1+t},\frac12\right\},
       \frac{2t}{2+t}\right\},\\
 \underline{\dim}_t(\mu_p)_V=\overline{\dim}_t(\mu_p)_V
 &=\max\left\{
       \left(\frac1p+\frac1{2t}\right)^{-1},
       \left(\frac1p+\frac1t\right)^{-1}\right\}
  =\left(\frac1p+\frac1{2t}\right)^{-1},
       \quad V\in E_p.
\end{align*}
Also $\dim_{\mathrm{qA}}\mu_p=2$ and
$\dim_{\mathrm A}\mu_p=\infty$. The source formula has the
three stated ranges, with transition points $1/3$ and $2/3$.
This proves \eqref{eq:nonharmonic-source} and
\eqref{eq:nonharmonic-image}.

\emph{Upper bounds for exceptional directions.}
Write $K^{(1)}=\spt\mu^{(1)}$, and let $\mathcal B$ be the union of
\[
 \{V:P_Vv=0\},\qquad
 \{V:0\in P_V(v+A)\},\qquad
 \bigcup_{x\in K^{(1)}}\{V:P_V(x-Lv)=0\}.
\]
The first family has dimension $k$. So does each family in the
last union, which is countable. The middle family is a countable
union of bundles over $c_jS$, with fibre $G(n-1,m)$.
Therefore
\[
 \dim_{\mathrm H}\mathcal B\le k+p.
\]

Fix $V\notin\mathcal B$. Compactness gives
\[
 P_Vv\ne0,\qquad
 \operatorname{dist}(0,P_V(v+A))>0,\qquad
 \operatorname{dist}(P_V(Lv),P_VK^{(1)})>0.
\]
Thus the marked points in late layers of the second component
are separated from the regular image, the first component, and
the other layers. At each small $r$, choose
$R_j\sim r^{2/3}$ and set $y_j=P_V(Lv)+R_jP_Vv$.
The marked-point sampling estimate, with $a=1$ and $q=2$, gives
\begin{align*}
 (\mu_p)_V(B(y_j,r))
 &\lesssim_V(1+|\log r|)^C
       \left(R_j+(r/\ell_j)^2\right)\\
 &\lesssim_V(1+|\log r|)^C r^{2/3}.
\end{align*}
For every $0<s<2/3$, division by $r^s$ gives a quantity tending
to zero at every sufficiently small $r$. Hence
$\underline{\dim}_1(\mu_p)_V\ge2/3$.
The source value at $t=1$ is $2/3$, so Lipschitz monotonicity gives
\[
 \underline{\dim}_1(\mu_p)_V
 =\overline{\dim}_1(\mu_p)_V=2/3,
 \quad V\notin\mathcal B.
\]
All sublevels in the statement, and the common equality set for
the entire image functions, therefore lie in $\mathcal B$.

\emph{Matching lower bounds and the optimized exponent.}
On $E_p$, the image value at $t=1$ is
$(1/p+1/2)^{-1}$. It is less than $1/2$ for $p<2/3$ and equals
$1/2$ for $p=2/3$. The dimension and positive Hausdorff measure
of $E_p$ give all matching lower bounds. Finally,
\[
 D(1)=\frac23,\qquad \kappa_{1/2}=\frac13,\qquad
 b=\left(2-\frac32+\frac12\right)^{-1}=1,
\]
whereas
\[
 b_*=\left[2-\frac13\left(2-\frac12\right)\right]^{-1}
 =\frac23.
\]
Letting $p\uparrow2/3$ proves sharpness of the strict exceptional
exponent; $p=2/3$ gives its non-strict attainment.
\end{proof}

\begin{remark}
The same construction works for every $1<q\le\infty$, replacing
the two growth exponents $2$ by $q$ and the source curve by
\[
 D_q(t)=\max\left\{\min\left\{\left(\frac1{q}+\frac1{2t}\right)^{-1},1/2\right\},\left(\frac1{q}+\frac1{t}\right)^{-1}\right\}.
\]
The source growth is then $q$, whereas all the image formulas and
exceptional-set dimensions in Corollary~\ref{cor:nonharmonic-sharp}
are unchanged. Indeed,
\[
 s=\left(\frac1{q}+1\right)^{-1}>\frac12,\qquad
 \kappa_{1/2}=\frac1{2(2-1/q)},\qquad
 b=1,\qquad b_*=\frac23.
\]
The incidence proof only uses $s>1/2$, which explains the restriction
$q>1$ for the stated non-strict conclusion. When $q=\infty$, the
second, uncapped component uses the external marked-point version
of Lemma~\ref{lem:general-reference-layers}; its source function
is $t$ and the marked image remains a witness outside the same
incidence set. The displayed curve
\eqref{eq:nonharmonic-source} is the particular case $q=2$.
\end{remark}

\subsection{Radial centres}

Radial matching depends on both the layer and the centre, so the
orthogonal construction needs adjustment. Reference sets containing
radial intervals permit exact matching. A second construction allows
an error of order $R_j^2$ and works on a restricted range of scales.
Together they give the following ranges of sharpness.

\begin{theorem}\label{thm:radial-stable-sharp}
Let $m\ge1$, $0<\theta\le1$ and $0<d\le s\le q\le\infty$,
with $d<\infty$. Put
\begin{equation}\label{eq:radial-stable-parameters}
 b=\left(\frac1d-\frac1s+\frac1q\right)^{-1}\le m,
 \qquad
 a=\left[\theta\left(\frac1s-\frac1q\right)\right]^{-1},
\end{equation}
where $a=\infty$ when $s=q$.
In each assertion below, $\mu$ is a probability measure with
countable compact support in $\R^{m+1}$ satisfying
\begin{equation}\label{eq:radial-sharp-source-data}
 \begin{gathered}
 \dim_{\mathrm{qA}}\mu=q,\qquad \dim_{\mathrm A}\mu=\infty,\\
 \underline{\dim}_t\mu=\overline{\dim}_t\mu
 =\left(\frac1q+\frac1{at}\right)^{-1},\quad0<t\le1.
 \end{gathered}
\end{equation}
In particular, both source dimensions at $\theta$ equal $s$.
\begin{enumerate}
\item[(i)] If $b>1$ or $a\theta<b$, then, for every $\varepsilon>0$, there is one such measure $\mu$
for which
\begin{align}
 &\dim_{\mathrm H}\{z\notin\spt\mu:
                    \overline{\dim}_\theta\mu_z<d\}
 >1+b-\varepsilon,\label{eq:radial-stable-optimal}\\
 &\dim_{\mathrm H}\{V\in G(m+1,m):
                    \overline{\dim}_\theta\mu_V<d\}
 >b-\varepsilon.\label{eq:radial-orthogonal-optimal}
\end{align}
Both inequalities also hold for lower intermediate dimensions.
\item[(ii)] Suppose that $d<s$ or $b=m$. If $b\ge1$ or $a\theta\le b$, then there is one such measure $\mu$ for which
\begin{align}
 &\dim_{\mathrm H}\{z\notin\spt\mu:\overline{\dim}_\theta\mu_z\le d\}
 =1+b,\label{eq:radial-nonstrict-optimal}\\
 &\dim_{\mathrm H}\{V\in G(m+1,m):\overline{\dim}_\theta\mu_V\le d\}
 =b.\label{eq:radial-orthogonal-nonstrict-optimal}
\end{align}
Both identities also hold for lower intermediate dimensions.
The common equality sets
\[
\begin{aligned}
&\dim_{\mathrm H}\{z\notin\spt\mu:
\underline{\dim}_\theta\mu_z
=\overline{\dim}_\theta\mu_z=d\}
=1+b,\\
&\dim_{\mathrm H}\{V\in G(m+1,m):
\underline{\dim}_\theta\mu_V
=\overline{\dim}_\theta\mu_V=d\}
=b.
\end{aligned}
\]
have dimensions $1+b$ and $b$, respectively, and positive Hausdorff
measure at these exponents.
\end{enumerate}
\end{theorem}

In part~(i), the measure may depend on $\varepsilon$, but both entire
source functions and the quasi-Assouad dimension remain fixed.
The condition $a\theta<b$ in part~(i) can equivalently be written as
$b(1/s-1/q)>1$.

\begin{proof}
\emph{Reference measures and source dimensions.}
For part~(i), choose a reference exponent $p<b$ as specified
below; for part~(ii), take $p=b$. We use two reference measures.

If $p\ge1$, choose a compact $(p-1)$-Ahlfors regular set $H$
in a small patch of $S^{m-1}\subset\R^m$, and set
\[
 A=\{r\omega:1/2\le r\le3,\ \omega\in H\},\qquad
 S=\{r\omega:1\le r\le2,\ \omega\in H\}.
\]
For $p=1$, take $H$ to be a singleton. The radial
parametrizations are bi-Lipschitz, so $A$ and $S$ are
$p$-Ahlfors regular. They also satisfy
\[
 cS\subset A,\quad3/4\le c\le5/4.
\]
Let $\sigma_A$ be a regular probability on $v+A$ and use
$\sigma_q=(\sigma_A+\tau_q)/2$ for finite $q$, with $\tau_q$
as in \eqref{eq:thin-geometric-reference}; for $q=\infty$, use
$\sigma_\infty=\sigma_A$. The marked point is $0$.

For the second construction, use $S,A,\sigma_q$ from
Lemma~\ref{lem:arithmetic-reference}, for any $0<p\le b$.
Here $p$ is the regular reference exponent and $q$ is the source
growth exponent. For finite $q$, the regular component satisfies
\eqref{eq:reference-ratio} with exponent $q$ because $p\le q$,
with a logarithmic loss in the second construction. The separated
thin component has $q$-growth and mass comparable to $u^q$ at $0$.
Thus Lemma~\ref{lem:general-reference-layers} applies.
For $q=\infty$, use its isolated-marked-point case.
Choose a fixed $0<\kappa\le1/4$ that puts the reference in the
unit ball, use polynomial counts for $a<\infty$ and exponential
counts for $a=\infty$, and write $D_\kappa x=\kappa x$. Thus
\eqref{eq:sharp-mu}, applied to the dilated reference, reads
\[
 \mu=Z^{-1}\sum_{j\ge2}j^{-2}
       (T_j\circ D_\kappa)_\#\eta_{N_j}.
\]
The same lemma gives all the required source data.

In both constructions, take the exterior centre set
\begin{equation}\label{eq:radial-cantor-cone}
 E_{\mathrm{rad}}=\{\lambda(v+w):2\le\lambda\le3,\ w\in S\}.
\end{equation}
Its parametrization is bi-Lipschitz, and therefore
\[
 \dim_{\mathrm H}E_{\mathrm{rad}}=1+p,\qquad
 0<\mathcal H^{1+p}(E_{\mathrm{rad}})<\infty.
\]
The set is uniformly separated from $\spt\mu$.

\emph{Exact radial matching.}
Use the first construction, and fix
$z=\lambda(v+w)\in E_{\mathrm{rad}}$.
For $t\in\{0,B^{-1},B^{-2},\ldots\}$, put
\begin{equation}\label{eq:radial-exact-scalar}
 c_{j,t,\lambda}
 =\frac{\lambda(1-t)}{\lambda-R_j-\ell_j\kappa t}.
\end{equation}
Since $0\le t\le B^{-1}\le1/16$, $R_j\le1/4$ and
$2\le\lambda\le3$, the denominator in
\eqref{eq:radial-exact-scalar} is positive and
$3/4\le c_{j,t,\lambda}\le5/4$. Consequently
$v+c_{j,t,\lambda}w\in v+A$. Set
\[
 X=R_jv+\ell_j\kappa tv,\qquad
 Y=R_jv+\ell_j\kappa(v+c_{j,t,\lambda}w).
\]
A direct calculation gives
\[
 Y-z=\frac{\lambda-R_j-\ell_j\kappa}
              {\lambda-R_j-\ell_j\kappa t}(X-z).
\]
Both numerator and denominator of the multiplier are positive,
so $\pi_z(X)=\pi_z(Y)$. Thus every thin image, including the
marked image, has a regular preimage in the same layer; the
matching is uniform in $j$, $t$ and $z\in E_{\mathrm{rad}}$.

The normalized layer map
$x\mapsto\pi_z(R_jv+\ell_j\kappa x)$ has Lipschitz constant
$C\ell_j$, uniformly for $z\in E_{\mathrm{rad}}$.
The regular reference mass in a ball of radius $u/(C\ell_j)$
therefore gives image mass at least
$c\min\{1,(u/\ell_j)^p\}$ at every image support point.
Relative sampling and the mass at each sample point yield
\[
 \nu_j(B(e,u))\gtrsim[\log(2N_j)]^{-C}
 \max\{N_j^{-1},\min\{1,(u/R_j)^p\}\},
\]
where $\nu_j$ is the radial image of layer $j$.
Moreover $\spt\nu_j\subset B(\pi_z(0),CR_j)$.
Lemma~\ref{lem:image-layers} gives
\[
 \overline{\dim}_t\mu_z
 \le\left(\frac1p+\frac1{at}\right)^{-1},
 \quad z\in E_{\mathrm{rad}},\ 0<t\le1.
\]
This includes exponential counts, for which the right-hand side is $p$.

\emph{Approximate radial matching.}
Use the second construction, with finite $a$, and write
\[
 f_{j,z}(x)=\pi_z(R_jv+\ell_j\kappa x).
\]
For each thin point $x$, including $0$, the identities in
\eqref{eq:exact-projected-inclusion} give $y\in v+A$ with
$x-y\parallel z$. For regular points, take $y=x$.
Since $D\pi_z(0)(x-y)=0$, the fundamental theorem of calculus gives
\begin{align*}
 f_{j,z}(x)-f_{j,z}(y)
 =\ell_j\kappa\int_0^1
 \bigl[D\pi_z(R_jv+\ell_j\kappa(y+s(x-y)))-D\pi_z(0)\bigr]
 (x-y)\,ds.
\end{align*}
The reference supports and centre set are fixed compact sets
with positive separation. Thus
\begin{equation}\label{eq:radial-mask-error}
 |f_{j,z}(x)-f_{j,z}(y)|\lesssim\ell_jR_j\lesssim R_j^2,
\end{equation}
uniformly in $j$ and $z\in E_{\mathrm{rad}}$.
Choose $C_1$ so that the error in
\eqref{eq:radial-mask-error} is at most $u/2$ whenever
$u\ge C_1R_j^2$. The uniform $C\ell_j$-Lipschitz bound then
maps the reference ball about $y$ of radius $u/(2C\ell_j)$
into $B(f_{j,z}(x),u)$. The lower mass bound in
\eqref{eq:logarithmic-regular-reference} therefore gives
\[
 (f_{j,z})_\#\sigma_q(B(f_{j,z}(x),u))
 \gtrsim j^{-C}\min\{1,(u/R_j)^p\}.
\]
Here the logarithmic loss is bounded by a power of $j$, since
$u\ge C_1R_j^2$. Applying relative sampling, and using the
atomic mass for smaller radii, gives
\begin{equation}\label{eq:radial-window-sample}
 \begin{aligned}
 (f_{j,z})_\#\eta_{N_j}(B(e,u))&\gtrsim N_j^{-1},
       &&u>0,\\
 (f_{j,z})_\#\eta_{N_j}(B(e,u))
 &\gtrsim[\log(2N_j)]^{-C}
       \max\{N_j^{-1},\min\{1,(u/R_j)^p\}\},
       &&u\ge C_1R_j^2.
 \end{aligned}
\end{equation}
These estimates hold at every image sample point $e$.

Fix $0<t\le1$ and choose
\[
 \max\left\{\frac{pat}{p+at},\frac{at}{2}\right\}<\beta<at.
\]
We verify the defining upper dimension condition at every small
$r$. If $R_j\ge r^{\beta/a}$, take $u=r$. The atom gives
\[
 \mu_z(B(e,r))\gtrsim j^{-2}R_j^a
 \gtrsim(1+|\log r|)^{-2}r^\beta.
\]
If $r^t/(4C_0)\le R_j<r^{\beta/a}$, take $u=r^t$.
Since $2\beta/a>t$, one has $C_1R_j^2\le r^t$ for all small
$r$, so the second line of \eqref{eq:radial-window-sample} applies.
Using $p(t-\beta/a)\le t\beta$, we obtain
\[
 \mu_z(B(e,r^t))
 \gtrsim(1+|\log r|)^{-C}
       \min\{1,(r^t/R_j)^p\}
 \gtrsim(1+|\log r|)^{-C}r^{t\beta}.
\]
If $R_j<r^t/(4C_0)$, the ball of radius $r^t$ contains the
tail of the image layers, so its mass is at least
$c(1+|\log r|)^{-1}$. This also handles $\pi_z(0)$.
The finitely many initial layer points have a positive minimum
mass, so they also satisfy the required bound for small $r$.
Absorbing logarithms by an arbitrarily small increase in the
exponent, and then decreasing $\beta$, proves
\[
 \overline{\dim}_t\mu_z
 \le\max\left\{\frac{pat}{p+at},\frac{at}{2}\right\}.
\]
The inequality $pat/(p+at)\ge at/2$ is equivalent to $at\le p$.
Hence the harmonic bound holds throughout this closed range.
At $at=p$, the auxiliary exponent $\beta$ is still chosen strictly
above $at/2$, and the desired value follows by letting $\beta$
decrease to $at/2$.

\emph{Orthogonal projections and strict sublevels.}
For either construction, put $V_w=(v+w)^\perp$, $w\in S$.
Every thin image lies in the regular image: in the first
construction use $(1-B^{-k})S\subset A$, and in the second use
\eqref{eq:digit-concealment}. The lower reference mass bound
and relative sampling verify Lemma~\ref{lem:image-layers}, giving
\[
 \overline{\dim}_t\mu_{V_w}
 \le\left(\frac1p+\frac1{at}\right)^{-1},
 \quad w\in S,\ 0<t\le1.
\]
The direction family is bi-Lipschitz to $S$ and has dimension $p$.

If $b>1$, choose $\max\{1,b-\varepsilon\}<p<b$ and use the
first construction. If $a\theta<b$, choose
$\max\{a\theta,b-\varepsilon\}<p<b$ and use the second.
At $t=\theta$, the radial upper bound in each case is
$(1/p+1/(a\theta))^{-1}$, and
\[
 \left(\frac1p+\frac1{a\theta}\right)^{-1}
 <\left(\frac1b+\frac1{a\theta}\right)^{-1}=d.
\]
The centre and direction families have dimensions $1+p$ and
$p$, respectively. They prove both strict inequalities in
part~(i), for both intermediate dimensions.

\emph{Non-strict sublevels and equality sets.}
For part~(ii), take $p=b$. If $b\ge1$, use the first
construction; if $a\theta\le b$, use the second. Both give the
image upper bound $d$ at $\theta$. The centre family carries
a $(1+b)$-Frostman probability, and the direction family carries
a $b$-Frostman probability. By Lemmas~\ref{lem:radial-trans} and
\ref{lem:grassmann-trans}, their transversality exponent is $b$.
Theorem~\ref{thm:transversal} gives, for both image dimensions,
the matching almost-everywhere lower bound
\[
 \left[\frac1s+\left(\frac1b-\frac1q\right)\right]^{-1}=d.
\]
Both common equality sets therefore have full measure for the
corresponding Frostman probabilities. The mass distribution
principle gives dimensions at least $1+b$ and $b$, and positive
Hausdorff measure at those exponents.

If $b<m$, the assumption in part~(ii) gives $d<s$. Choose
$e>d$ with $e<s$ and
\[
 b_e=\left(\frac1e-\frac1s+\frac1q\right)^{-1}<m.
\]
The non-strict sublevels at $d$ lie in the strict sublevels at
$e$. Theorems~\ref{thm:radial} and~\ref{thm:exception} bound
their dimensions by $1+b_e$ and $b_e$. Letting $e\downarrow d$
gives the required upper bounds. If $b=m$, use the ambient
dimensions $m+1$ and $m$. The common equality sets are contained
in these non-strict sublevels, so the same upper bounds finish
all the assertions.
\end{proof}

The two matching arguments leave the strict range
$b\le1$, $a\theta\ge b$ outside the layered construction.
For non-strict sublevels they include $b=1$ and $a\theta=b$,
but leave $b<1$, $a\theta>b$ untreated. The obstruction to exact
matching in the present centre cone is geometric: for fixed $w\ne0$,
matching the marked point for every $\lambda\in[2,3]$ requires a
nondegenerate interval of multiples of $w$ in the regular reference.
Its dimension must therefore be at least one. The condition $d<s$ or $b=m$ in part~(ii) is also needed:
if $d=s$ and $q=b<m$, Lipschitz monotonicity makes the non-strict
sublevel the whole parameter space.

For constant source functions, we obtain further ranges by using
the reference measure directly. The next proposition gives radial
dimension bounds, with exact formulas when $p<m$, and keeps the
Assouad dimension finite.

\begin{proposition}[Radial images of a constant source function]
\label{prop:radial-constant-reference}
Let $m\ge1$, $0<q<\infty$, and $0<p\le\min\{q,m\}$.
There are a compactly supported probability measure $\mu$ in
$\R^{m+1}$ and a compact $(1+p)$-Ahlfors regular set
$E\subset\R^{m+1}\setminus\spt\mu$ such that
\[
 \dim_{\mathrm A}\mu=\dim_{\mathrm{qA}}\mu=q,
 \qquad
 \underline{\dim}_t\mu=\overline{\dim}_t\mu=q,
 \quad 0<t\le1,
\]
and, for every $z\in E$ and $0<t\le1$,
\begin{equation}\label{eq:radial-constant-reference-window}
 p\le\underline{\dim}_t\mu_z
 \le\overline{\dim}_t\mu_z\le\max\{p,qt\}.
\end{equation}
In particular, both projected dimensions equal $p$ on the whole
set $E$, simultaneously for $0<t\le p/q$.
If $p<m$, one may choose the same pair $(\mu,E)$ to satisfy
all the preceding assertions and
\begin{equation}\label{eq:radial-constant-reference-exact}
 \underline{\dim}_t\mu_z=\overline{\dim}_t\mu_z
 =\max\{p,qt\}
 \qquad(z\in E,\ 0<t\le1).
\end{equation}

\end{proposition}

\begin{proof}
\emph{Construction and source estimates.}
Write $\R^{m+1}=\R v\oplus U$, with $v$ a unit vector and
$\dim U=m$. Choose a compact $p$-Ahlfors regular set $S\subset U$
bounded away from $0$, and a regular probability $\nu_S$ on $S$.
Let $\sigma$ be its translate to $v+S$. For $B\ge16$, set
\[
 \tau_q=(1-B^{-q})\sum_{j\ge1}B^{-q(j-1)}\delta_{B^{-j}v},
 \qquad \mu=\tfrac12(\sigma+\tau_q).
\]
The supports are separated. For $x\in\spt\tau_q$ and $0<u\le1$, geometric summation gives
\[
 \tau_q(B(x,u))\sim\max\{|x|,u\}^q.
\]
Here the constants may depend on $B$ and $q$, but not on $x$ or $u$.
At $x=0$ this gives $\tau_q(B(0,u))\sim u^q$, and, for
$0<u\le R\le1$,
\[
 \frac{\tau_q(B(x,R))}{\tau_q(B(x,u))}
 \lesssim
 \left(\frac{\max\{|x|,R\}}{\max\{|x|,u\}}\right)^q
 \le(R/u)^q.
\]
Since $p\le q$, the regular component and the separated-support
argument give $\dim_{\mathrm A}\mu\le q$.
For small $u$, $\mu(B(0,u))\sim u^q$, so for every $s<q$,
\[
 \sup_{r\le u\le r^t}\frac{\mu(B(0,u))}{u^s}
 \lesssim r^{t(q-s)}\longrightarrow0.
\]
Thus $\underline{\dim}_t\mu\ge q$ for all $t>0$.
The dimension comparisons in Lemma~\ref{lem:window} now give
all the stated source equalities, including both growth dimensions.

\emph{Radial image bounds.}
Define
\[
 E=\{\lambda(v+w):2\le\lambda\le3,\ w\in S\}.
\]
The parameter map is bi-Lipschitz, so $E$ is compact and
$(1+p)$-Ahlfors regular. It is separated from $\spt\mu$.
For $z=\lambda(v+w)\in E$,
\[
 \pi_z(0)=\pi_z(v+w).
\]
The restriction of $\pi_z$ to any fixed bounded patch of the
hyperplane $v+U$ is uniformly bi-Lipschitz for $z\in E$.
Indeed, if $\omega=\pi_z(x)$ for $x\in v+U$, its inverse is
\[
 x=z+\frac{1-\lambda}{\langle\omega,v\rangle}\omega,
\]
and
\[
 |\langle\omega,v\rangle|
 =\frac{\lambda-1}{|x-z|}
\]
is uniformly bounded below when $x$ ranges over the bounded patch.
Consequently $\xi_z=(\pi_z)_\#\sigma$ is
$p$-Ahlfors regular, with uniform constants, and
$\pi_z(0)\in\spt\xi_z$.

Write $t_j=B^{-j}$ and $e_{j,z}=\pi_z(t_jv)$. Uniformly on $E$,
\[
 |e_{j,z}-\pi_z(0)|\lesssim t_j,
 \qquad \mu_z(\{e_{j,z}\})\gtrsim t_j^q.
\]
Fix $0<t\le1$ and $D>\max\{p,qt\}$. If
$t_j\ge r^{D/q}$, the atom gives
\[
 \mu_z(B(e_{j,z},r))\gtrsim t_j^q\ge r^D.
\]
If $t_j<r^{D/q}$, then $D/q>t$ implies
$|e_{j,z}-\pi_z(0)|\le r^t/2$ for all small $r$. Regularity gives
\begin{align*}
 \mu_z(B(e_{j,z},r^t))
 &\ge\tfrac12\xi_z(B(\pi_z(0),r^t/2))\\
 &\gtrsim r^{tp}\ge r^{tD}.
\end{align*}
At points of $\spt\xi_z$, use its regular lower bound directly.
This verifies the defining upper condition at every support
point, and proves
$\overline{\dim}_t\mu_z\le\max\{p,qt\}$.

For the lower bound, choose $y_z\in\spt\xi_z$ outside the
compact countable set $\{\pi_z(0)\}\cup\{e_{j,z}:j\ge1\}$.
Such a point exists since $p>0$. Small balls about $y_z$ meet
no thin image, so
\[
 \mu_z(B(y_z,u))\lesssim u^p.
\]
For every $s<p$, the same supremum calculation as for the
source gives $\underline{\dim}_t\mu_z\ge p$. This proves
\eqref{eq:radial-constant-reference-window}.

\emph{Exact image functions for $p<m$.}
Choose $0<\delta<1$ with $0\le p-\delta\le m-1$, and a compact
$(p-\delta)$-Ahlfors regular set $H\subset S^{m-1}$. A singleton
is allowed when $p=\delta$. Choose a large integer $N$ so that
\[
 B=N^{1/\delta}\ge16,\qquad T=\frac{N-1}{B-1}<\frac1{100},
\]
and let
\[
 C=\left\{\sum_{k\ge1}d_kB^{-k}:d_k\in\{0,\ldots,N-1\}\right\},
 \qquad S=\{r\omega:r\in1+C,\ \omega\in H\}.
\]
The similarities defining $C$ have strong separation, so $C$
is $\delta$-Ahlfors regular. The radial parametrization is
bi-Lipschitz, and therefore this choice of $S$ is $p$-Ahlfors
regular. Use the same $B$ in $\tau_q$.

Fix $z=\lambda(v+w)\in E$, write $w=r\omega\in S$, and put
\[
 c_{j,\lambda}=\frac{\lambda(1-t_j)}{\lambda-t_j}.
\]
Direct calculation gives
\[
 \pi_z(t_jv)=\pi_z(v+c_{j,\lambda}w),\qquad
 \frac12t_j\le r-c_{j,\lambda}r\le\frac7{10}t_j.
\]
Let $r_j$ be the level-$j$ prefix of $r\in1+C$. Its cylinder
lies in $[r_j,r_j+Tt_j]$. Distinct level-$j$ prefix values are separated by at least
$(1-T)t_j$. Indeed, at their first differing digit $k\le j$,
the leading difference is at least $B^{-k}$, while the total
remaining difference is at most $TB^{-k}$.
Consequently,
\[
 c_{j,\lambda}r-r_j
 \in[-7t_j/10,-(1/2-T)t_j],
\]
The cylinder at $r_j$ and all following cylinders lie to the
right of this point by at least $(1/2-T)t_j$. Every preceding
cylinder ends at or to the left of $r_j-(1-2T)t_j$, and hence
lies to its left by at least $(3/10-2T)t_j$. Therefore
\begin{align*}
 \operatorname{dist}(c_{j,\lambda}r,1+C)
 &\ge\min\{1/2-T,\,3/10-2T\}t_j\ge t_j/4,\\
 \operatorname{dist}(c_{j,\lambda}w,S)&\ge t_j/4.
\end{align*}
This calculation only uses cylinder separation, so $B$ need not
be an integer. Apply the uniform inverse Lipschitz estimate on one bounded
hyperplane patch containing both $v+S$ and all
$v+c_{j,\lambda}w$. It yields
\[
 \operatorname{dist}(\pi_z(t_jv),\pi_z(v+S))\gtrsim t_j.
\]
Also $\pi_z$ is uniformly bi-Lipschitz on
$\{tv:0\le t\le B^{-1}\}$, since $w$ stays bounded away from
$0$. Together with geometric spacing, this gives a fixed $c_0>0$
such that
\[
 \mu_z(B(\pi_z(t_jv),u))=c_qt_j^q,
 \quad0<u<c_0t_j,\ z\in E.
\]

Fix $0<t\le1$ and $0<s<qt$. At every sufficiently small $r$,
choose $j$ with
\[
 2c_0^{-1}r^t\le t_j<2Bc_0^{-1}r^t.
\]
All balls about $\pi_z(t_jv)$ with radii in $[r,r^t]$ contain
only that thin atom. Hence
\[
 \sup_{r\le u\le r^t}
 \frac{\mu_z(B(\pi_z(t_jv),u))}{u^s}
 \le\frac{c_qt_j^q}{r^s}\lesssim r^{qt-s}\longrightarrow0.
\]
This proves $\underline{\dim}_t\mu_z\ge qt$ at every positive
parameter. Combining it with the lower bound $p$ and the upper
bound already proved gives
\eqref{eq:radial-constant-reference-exact}.
\end{proof}

\paragraph{Sharpness for constant source functions.}
We apply the preceding proposition to the radial exceptional-set bound.
Let $0<q<\infty$, $0<\theta\le1$, and
$0<d\le\min\{q,m\}$. Since the source function is constant,
$s=q$ and $b=d$.
If $\theta q<d$, choose
\[
 \max\{\theta q,d-\varepsilon,0\}<p<d
\]
in Proposition~\ref{prop:radial-constant-reference}. The resulting
measure has both source functions identically $q$, has finite
Assouad dimension $q$, and satisfies
\[
 \dim_{\mathrm H}\{z\notin\spt\mu:\overline{\dim}_\theta\mu_z<d\}
 >1+d-\varepsilon.
\]
The same bound holds for lower intermediate dimensions.
If $\theta q\le d$ and either $d<q$ or $d=m$, take $p=d$.
The same source data then give
\[
 \dim_{\mathrm H}\{z\notin\spt\mu:\overline{\dim}_\theta\mu_z\le d\}
 =1+d.
\]
The same identity holds for lower intermediate dimensions, and the
common equality set has positive $\mathcal H^{1+d}$-measure. The lower bound follows
from the whole Ahlfors regular centre set $E$; the upper bound
follows by decreasing strict thresholds to $d$, or from the
ambient dimension when $d=m$.
These are additional parameter ranges for compactly supported
measures.

\section{Nonconstant projection dimensions}
\label{sec:examples}

We turn from the size of exceptional sets to the dependence of projected
dimension functions on direction for one fixed measure. The proposition
first realizes both bounds of Theorem~\ref{thm:orthogonal-main} on
open direction sets, then gives a planar construction with prescribed
finite collections of image functions and prescribed Haar proportions.

\begin{proposition}\label{prop:nonconstant-directions}
Let $1\le q\le\infty$ and $0<a\le\infty$.
In each assertion below, $\rho$ is a probability measure with
countable compact support satisfying
\begin{equation}\label{eq:subtarget-source}
\begin{gathered}
 \dim_{\mathrm{qA}}\rho=q,\qquad
 \dim_{\mathrm A}\rho=\infty,\\
 \underline{\dim}_\theta\rho
 =\overline{\dim}_\theta\rho
 =\left(\frac1q+\frac1{a\theta}\right)^{-1},
 \quad0<\theta\le1.
\end{gathered}
\end{equation}
\begin{enumerate}
\setlength{\itemsep}{0.6\baselineskip}

\item[(i)]
For integers $1\le m<n$ with $m<q$, there exist a measure
$\rho$ on $\R^n$ and disjoint nonempty open sets
$\mathcal U_-,\mathcal U_+\subset G(n,m)$ such that, for every
$0<\theta\le1$,
\[
 \underline{\dim}_\theta\rho_V
 =\overline{\dim}_\theta\rho_V
 =
 \begin{cases}
 \displaystyle\left(\frac1m+\frac1{a\theta}\right)^{-1},
       &V\in\mathcal U_-,\\[5pt]
 \displaystyle\left(\frac1q+\frac1{a\theta}\right)^{-1},
       &V\in\mathcal U_+.
 \end{cases}
\]
Thus both bounds of Theorem~\ref{thm:orthogonal-main} are
attained on open direction sets by the same measure.
$\underline{\dim}_\theta\rho
 =\overline{\dim}_\theta\rho<m$, $0<\theta\leq1$ if $a<mq/(q-m)$ for $q<\infty$, or $a<m$ for $q=\infty$.

\item[(ii)]
Let $M\ge2$ be an integer and let $p_1,\ldots,p_M\in[1,q]$.
Given $\lambda_1,\ldots,\lambda_M>0$ with
$\sum_{i=1}^M\lambda_i=1$, there exist a measure $\rho$ as above
on $\R^2$ and pairwise disjoint nonempty open sets
$U_1,\ldots,U_M\subset G(2,1)$ such that
$\gamma_{2,1}(U_i)=\lambda_i$ for every $i$ and
$G(2,1)\setminus\bigcup_{i=1}^M U_i$ is finite.
For every $i$, every $V\in U_i$, and every $0<\theta\le1$,
\begin{equation}\label{eq:subtarget-directional-spectrum}
 \underline{\dim}_\theta\rho_V
 =\overline{\dim}_\theta\rho_V
 =\left(\frac1{p_i}+\frac1{a\theta}\right)^{-1}.
\end{equation}
The same image formula holds for $\rho_z$ whenever
$z\notin\spt\rho$ and $z^\perp\in U_i$.
$\underline{\dim}_\theta\rho
 =\overline{\dim}_\theta\rho<1$, $0<\theta\leq1$ if $a<q/(q-1)$
for finite $q>1$, if $a<\infty$ for $q=1$, or if $a<1$
for $q=\infty$.

\end{enumerate}
\end{proposition}

We first give the transfer lemma used in
Section~\ref{sec:third-round}. It converts reference mass bounds
into exact projected dimension functions for the layered measure.
The point realizing an upper mass bound need not belong to every
sample: a nearby sampled point suffices.

\begin{lemma}\label{lem:cloud-transfer}
Let $\sigma$ be a probability measure with
$\spt\sigma\cup\{x_*\}$ contained in the unit ball, and let
$\rho_a$ be its layered measure
\eqref{eq:sharp-mu}, with axis $v$. Use polynomial sample counts
for $0<a<\infty$, and $N_j=\lceil e^{j^2}\rceil$ for $a=\infty$.
Fix $V$ with $P_Vv\ne0$, write $\xi=(P_V)_\#\sigma$, and
assume $P_Vx_*\in\spt\xi$.
Suppose that, for some $p\in(0,\infty)$, $h\ge0$, $u_0>0$, and
$z_*\in\spt\xi$,
\begin{equation}\label{eq:transfer-base-balls}
\begin{aligned}
 &\xi(B(z,u))\ge\frac{c_Vu^p}{(1+|\log u|)^h},\ z\in\spt\xi,\\
 &\xi(B(z_*,u))\le C_Vu^p,
\end{aligned}
\quad0<u<u_0.
\end{equation}
Then, for every $0<t\le1$,
\[
 \underline{\dim}_t((\rho_a)_V)
 =\overline{\dim}_t((\rho_a)_V)
 =\left(\frac1p+\frac1{at}\right)^{-1},
\]
where the right-hand side is $p$ for $a=\infty$.
The samples can be fixed independently of $V$.
\end{lemma}

\begin{proof}
\emph{Estimates for the projected samples.}
Fix $V$ as in the statement and set
$\xi=(P_V)_\#\sigma$. For all sufficiently large $N$, put
\[
 \delta_N=\left(\frac{C_V'[\log(2N)]^{h+1}}N\right)^{1/p}.
\]
By the lower bound in \eqref{eq:transfer-base-balls}, choosing
$C_V'$ sufficiently large gives
\[
 \xi(B(z_*,\delta_N))\ge C\log(2N)/N.
\]
Lemma~\ref{lem:relative-clouds} therefore supplies a projected
sample point
$z_{N,V}\in B(z_*,\delta_N)$.
For $0<u\le1$, the inclusion
$B(z_{N,V},u)\subset B(z_*,u+\delta_N)$ and the upper relative
sampling bound give
\begin{align*}
 (P_V)_\#\eta_N(B(z_{N,V},u))
 &\lesssim\xi(B(z_*,u+\delta_N))+\log(2N)/N\\
 &\lesssim_V u^p+\delta_N^p+\log(2N)/N.
\end{align*}
For radii outside the original small reference range, enlarge
the constant using the total mass. We conclude that
\begin{equation}\label{eq:transfer-cloud-witness}
 (P_V)_\#\eta_N(B(z_{N,V},u))
 \lesssim_V u^p+[\log(2N)]^{h+1}/N,
 \quad0<u\le1.
\end{equation}

For a projected sample point $z$, apply
\eqref{eq:relative-cloud-balls} to
$A=P_V^{-1}(B(z,u))$. This set meets $\spt\eta_N$, so
\[
 (P_V)_\#\eta_N(B(z,u))
 \gtrsim [\log(2N)]^{-1}
       \max\{N^{-1},\xi(B(z,u))\}.
\]
The reference lower bound applies at $z$, including the projected
marked point, because $P_Vx_*\in\spt\xi$. If $u^p\ge N^{-1}$,
then $1+|\log u|\lesssim_p\log(2N)$; if $u^p<N^{-1}$,
the mass of the sample point itself is at least $(N+1)^{-1}$.
Extending the reference bound to $u\le1$ by adjusting its constant,
we obtain
\begin{equation}\label{eq:transfer-cloud-lower}
 (P_V)_\#\eta_N(B(z,u))
 \gtrsim_V[\log(2N)]^{-h-1}\max\{N^{-1},u^p\},
 \quad z\in\spt((P_V)_\#\eta_N),\ 0<u\le1.
\end{equation}

\emph{Layer geometry and upper dimensions.}
Put $w=P_Vv\ne0$ and
$\nu_j=(P_V\circ T_j)_\#\eta_{N_j}$. Scaling
\eqref{eq:transfer-cloud-lower} gives
\[
 \nu_j(B(y,u))\gtrsim_V[\log(2N_j)]^{-C}
 \max\{N_j^{-1},\min\{1,(u/R_j)^p\}\}.
\]
Also $\spt\nu_j\subset B(R_jw,C\ell_j)\subset B(0,C_VR_j)$.
Lemma~\ref{lem:image-layers} gives the claimed upper bounds.

For distinct sufficiently large $j,k$,
\[
 |R_jw-R_kw|\ge\tfrac12|w|\max\{R_j,R_k\},\qquad
 \ell_j+\ell_k=o(\max\{R_j,R_k\}).
\]
Thus late image layers are separated by a constant times the
larger layer scale. The earlier image points form a finite set:
its nonzero points stay away from late layers, and its points at
$0$ have distance comparable to $R_j$ from layer $j$.
Consequently every sufficiently late layer is separated from
all the other image points by at least $c_VR_j$.

\emph{Lower dimensions.}
Let $a<\infty$, fix $0<t\le1$, and set
\[
 d=\left(\frac1p+\frac1{at}\right)^{-1}.
\]
For every small $r$, choose $R_j\sim r^{d/a}$ and put
$y_j=R_jw+\ell_jz_{N_j,V}$. Since $d/a<t$, one has
$r^t=o(\ell_j)$ and $r^t=o(R_j)$. Hence all balls
$B(y_j,u)$ with $r\le u\le r^t$ meet only layer $j$.
For $0<s<d$, scaling \eqref{eq:transfer-cloud-witness} gives
\begin{align*}
 \sup_{r\le u\le r^t}
 \frac{(\rho_a)_V(B(y_j,u))}{u^s}
 &\lesssim_V(1+|\log r|)^C
     \left(N_j^{-1}r^{-s}+\ell_j^{-p}r^{t(p-s)}\right)\\
 &\lesssim_V(1+|\log r|)^C
     \left(r^{d-s}+r^{t(p-s)-pd/a}\right)\\
 &\lesssim_V(1+|\log r|)^C
     \left(r^{d-s}+r^{t(d-s)}\right)\longrightarrow0.
\end{align*}
We used $p(t-d/a)=td$. Since the estimate holds at every
sufficiently small $r$, it proves
$\underline{\dim}_t((\rho_a)_V)\ge d$.

For $a=\infty$, fix $0<s<p$, choose $0<c<t(p-s)/p$, and take
$R_j\sim r^c$. The same separation and sampling estimates give
\[
 \sup_{r\le u\le r^t}
 \frac{(\rho_\infty)_V(B(y_j,u))}{u^s}
 \lesssim_V(1+|\log r|)^C
 \left(e^{-j^2}r^{-s}+r^{t(p-s)-cp}\right)\longrightarrow0.
\]
Here $j\sim|\log r|$ and $t(p-s)-cp>0$.
Thus the lower dimension is at least $p$. Together with the
upper bounds, this proves all the equalities. The samples were
chosen for the common range class; only the witness point
$z_{N,V}$ depends on $V$.
\end{proof}

\begin{proof}[Proof of Proposition~\ref{prop:nonconstant-directions}]
\emph{Part (i).}
Write $W=\R v\oplus U\subset\R^n$, where $v$ is a unit
vector, $v\perp U$ and $\dim U=m$.
Take a closed $m$-dimensional cube $S\subset U$ bounded away
from $0$, and use the construction in
Lemma~\ref{lem:arithmetic-reference} with $p=m$.
Denote by $A\supset S$ and $\sigma_q$ the resulting compact set
and reference probability. After a fixed dilation, use $\sigma_q$
and marked point $0$ in
\eqref{eq:sharp-mu}, with axis $v$ and the two choices of
sample counts corresponding to $a$. By
Lemma~\ref{lem:general-reference-layers}, the measure $\rho$
satisfies \eqref{eq:subtarget-source}.

Put $Z_0=W^\perp$, and let $Q$ be a compact cube with nonempty
interior in $\operatorname{Hom}(U,Z_0)$.
Use the Grassmannian chart $V_{w,A_0}$ from the proof of
Theorem~\ref{thm:finite-graded-sharp}. Let $\mathcal U_-$ be
the chart image of $\operatorname{int}S\times\operatorname{int}Q$.
For $n=m+1$, the second factor is the zero-dimensional space.
This is a nonempty open set, and
Lemma~\ref{lem:cloud-transfer}, with exponent $m$, gives
\[
 \underline{\dim}_\theta\rho_V=\overline{\dim}_\theta\rho_V
 =\left(\frac1m+\frac1{a\theta}\right)^{-1},
 \quad V\in\mathcal U_-,\ 0<\theta\le1.
\]

Since $A$ is compact, choose a nonempty open chart box with
$w\notin A\cup\{0\}$, and call its image $\mathcal U_+$.
It is disjoint from $\mathcal U_-$, because $S\subset A$.
For every $V\in\mathcal U_+$, $P_Vv\ne0$ and
$0\notin P_V(v+A)$. For finite $q$, the reference image has
a $q$-power upper bound at $0$ and a uniform lower $q$-power
bound with logarithmic loss. Lemma~\ref{lem:cloud-transfer}
therefore gives both image functions equal to the source
function. For $q=\infty$, use the isolated marked-point
calculation in the proof of Theorem~\ref{thm:finite-graded-sharp}.
These arguments are pointwise on the two open sets and hold
for all positive parameters. Finally,
\[
 \left(\frac1q+\frac1a\right)^{-1}<m
 \quad\Longleftrightarrow\quad
 a<\frac{mq}{q-m}
\]
for $q<\infty$; for $q=\infty$, the condition is $a<m$.
Monotonicity in $\theta$ proves the last assertion of part~(i).

\emph{Part (ii): the reference and direction sets.}
Choose a centrally symmetric convex $2M$-gon $P$, with opposite
vertex pairs $\pm z_1,\ldots,\pm z_M$, such that the vertex
normal cones have the prescribed Haar proportions $\lambda_i$.
For example, take supporting lines tangent to a circle whose
successive normal-angle gaps are $\pi\lambda_i$, and repeat
them after rotation by $\pi$. Each opposite vertex pair then
has normal-cone proportion $\lambda_i$ in $G(2,1)$.
Dilate $P$ into the unit ball, put $\Gamma=\partial P$, and
choose a unit vector $v$ parallel to an edge. For $q<\infty$,
choose $x_*$ in an edge interior; for $q=\infty$, take $x_*=0$.

Define
\[
 w_i(r)=\begin{cases}
 r^{p_i-1},&p_i<\infty,\\
 e^{-1/r},&p_i=\infty,
 \end{cases}
 \qquad
 w_*(x)=\begin{cases}
 |x-x_*|^{q-1},&q<\infty,\\
 1,&q=\infty,
 \end{cases}
\]
using their continuous extensions at zero, with $r^0=1$.
Let $\sigma$ be the probability on $\Gamma$ with density
\begin{equation}\label{eq:polygon-reference-density}
 \frac{d\sigma}{d(\mathcal H^1|_\Gamma)}(x)
 =c_Pw_*(x)\prod_{i=1}^M w_i(|x-z_i|)w_i(|x+z_i|).
\end{equation}
The density is positive except at finitely many points;
therefore $\spt\sigma=\Gamma$ and $\sigma$ is nonatomic.
Let $\Sigma$ contain the vertices and, for finite $q$, $x_*$.
From each open vertex normal cone remove the finitely many
directions perpendicular to an edge or to $x-y$ for distinct
$x,y\in\Sigma$. Denote the resulting sets by $U_i$.
They are pairwise disjoint, nonempty and open, have Haar
measures $\lambda_i$, and their union has finite complement.
In particular, $P_Vv\ne0$ on their union.

\emph{Source estimates and sampling.}
For finite $q$, near each $z\in\Sigma$ the density is comparable
to $|x-z|^\beta$, for some $0\le\beta\le q-1$.
On each incident edge, elementary integration gives
\[
 \sigma(B(x,u))\sim u(|x-z|+u)^\beta.
\]
At a vertex, add the contributions from its two incident
edges. Away from $\Sigma$, the density is bounded above and
below. Consequently,
\begin{align*}
 \frac{\sigma(B(x,R))}{\sigma(B(x,u))}
 \lesssim\frac{R(|x-z|+R)^\beta}
                  {u(|x-z|+u)^\beta}
 \lesssim(R/u)^{1+\beta}\le(R/u)^q,\quad
 \sigma(B(x_*,u))&\sim u^q.
\end{align*}
The local estimates extend to all $0<u\le R\le1$ by compactness
and separation of the finitely many special points.
For $q=\infty$, the marked point has positive distance from
$\Gamma$. Thus the respective hypotheses of
Lemma~\ref{lem:general-reference-layers} hold.

Choose the common relative samples with all vertices prescribed
in every sufficiently late layer, and construct $\rho$ by
\eqref{eq:sharp-mu}, with axis $v$ and the corresponding sample
counts. That lemma proves \eqref{eq:subtarget-source} and the
growth assertions. The same samples will be used for every
orthogonal and radial image.

\emph{Orthogonal images with finite $p_i$.}
Fix $V\in U_i$, and orient its unit vector $e$ so that $z_i$
is the exposed vertex. Set $h=e\cdot z_i>0$, so the image of
$\Gamma$ is $[-h,h]$. On every edge, projection has nonzero
constant derivative. Near either image endpoint, only the two
incident edges contribute, and
\[
 |x-z_i|\sim_V h-e\cdot x.
\]
The change-of-variables formula and the vertex weight give
\[
 \frac{d(P_V)_\#\sigma}{ds}(s)
 \sim_V(h-|s|)^{p_i-1}
\]
near $\pm h$. Every interior image level has two boundary
preimages. Since the projections of the points in $\Sigma$ are
distinct, at least one preimage lies away from $\Sigma$.
Compactness gives a positive lower density bound on every
compact subinterval of $(-h,h)$. Integrating at the endpoints
and in the interior, and using $p_i\ge1$, yields
\begin{align*}
 (P_V)_\#\sigma(B(s,u))&\gtrsim_Vu^{p_i},
       \quad s\in[-h,h],\\
 (P_V)_\#\sigma(B(h,u))&\sim_Vu^{p_i}.
\end{align*}
The projected marked point also belongs to $[-h,h]$.
Lemma~\ref{lem:cloud-transfer}, with exponent $p_i$, now proves
\eqref{eq:subtarget-directional-spectrum}.

\emph{Orthogonal images with $p_i=\infty$.}
Fix $V\in U_i$ and $0<\theta\le1$, and choose $e,h$ as in the
preceding case. Here $q=\infty$. The strict normal-cone condition gives
$|x-z_i|\lesssim_V h-e\cdot x$ near the exposed vertex.
Since $e^{-1/r}$ vanishes faster than every power,
\[
 (P_V)_\#\sigma(B(P_Vz_i,u))\lesssim_{V,L}u^L,
 \quad L>0.
\]
The vertex is a prescribed sample point, so relative sampling
gives
\[
 (P_V)_\#\eta_N(B(P_Vz_i,u))
 \lesssim_{V,L}u^L+\log(2N)/N.
\]
The image layers are separated by the geometry in
Lemma~\ref{lem:cloud-transfer}. For $a<\infty$, fix $0<s<a\theta$,
choose $s/a<c<\theta$ and $L>\theta s/(\theta-c)$, and take
$R_j\sim r^c$ at each small $r$. At
$y_j=R_jP_Vv+\ell_jP_Vz_i$, the scaled estimate gives
\begin{align*}
 \sup_{r\le u\le r^\theta}\frac{\rho_V(B(y_j,u))}{u^s}
 &\lesssim_{V,L}(1+|\log r|)^C
       \left(R_j^ar^{-s}+\ell_j^{-L}r^{\theta(L-s)}\right)\\
 &\lesssim_{V,L}(1+|\log r|)^C
       \left(r^{ca-s}+r^{\theta(L-s)-cL}\right)\longrightarrow0.
\end{align*}
Both exponents are positive by the choices of $c,L$.
Thus the lower image dimension is at least $a\theta$; Lipschitz
monotonicity gives the matching upper bound.
For $a=\infty$, take any $s>0$, $0<c<\theta$, and $L>\theta s/(\theta-c)$.
The first term becomes $e^{-j^2}r^{-s}$, which tends to zero
because $j\sim|\log r|$. Hence both image dimensions are
infinite, as required.

\emph{Radial projections: normalized reference estimates.}
Fix $z\notin\spt\rho$ with $z^\perp\in U_i$.
Since $0\in\spt\rho$, $z\ne0$. For late layers define
\[
 b_j=\frac{z-R_jv}{\ell_j},\qquad d_j=|b_j|,\qquad
 n_j=b_j/d_j,
\]
and choose unit $e_j\perp n_j$ tending to the unit vector
$e\in z^\perp$ exposing $z_i$. Put
\[
 F_j(x)=\frac{e_j\cdot x}{1-n_j\cdot x/d_j}.
\]
Then $d_j\sim_z\ell_j^{-1}\to\infty$ and
$F_j\to(x\mapsto e\cdot x)$ in $C^2(P)$. Its level sets are
lines. The strict normal-cone inequality
\[
 e\cdot(z_i-x)\ge c_z|z_i-x|,\quad x\in P
\]
and $C^1$ convergence show that $F_j$ has its maximum at $z_i$
and minimum at $-z_i$, for every sufficiently large $j$.

The edge derivatives of $F_j$ are uniformly bounded above and
away from zero, and the finitely many values $F_j(\Sigma)$
remain uniformly separated. Write
\[
 h_j^-=F_j(-z_i),\qquad h_j^+=F_j(z_i),\qquad
 \xi_j=(F_j)_\#\sigma.
\]
For $p_i<\infty$, the edgewise change-of-variables formula gives
\[
 \frac{d\xi_j}{ds}(s)\sim_z
 \begin{cases}
 (s-h_j^-)^{p_i-1},&s\text{ near }h_j^-,\\
 (h_j^+-s)^{p_i-1},&s\text{ near }h_j^+,
 \end{cases}
\]
with neighborhoods and constants independent of large $j$.
Away from these neighborhoods, every level line meets $\Gamma$
in two distinct points. At least one stays a fixed distance
from $\Sigma$: otherwise a limiting level would either contain
two distinct points of $\Sigma$, contrary to the excluded
directions, or be tangent at an exposed vertex, contrary to
its distance from the two extreme levels. Thus $d\xi_j/ds$ is
bounded below uniformly on the remaining interval. Integration
gives a lower $p_i$-power ball bound throughout $[h_j^-,h_j^+]$
and an upper $p_i$-power bound at $h_j^+$, uniformly in $j$.
For $p_i=\infty$, the upper bound at $h_j^+$ holds with every
finite power $L$ by the same vertex estimate.
The marked image is covered by the lower bound when $p_i<\infty$:
for finite $q$ it comes from $\Gamma$; for $q=\infty$, $F_j(0)=0$
lies in $F_j(\Gamma)$ because its level line crosses the interior of $P$.

To convert to the chord metric, observe that
\[
 \pi_{b_j}(x)
 =\frac{-n_j+(F_j(x)/d_j)e_j}
        {\sqrt{1+(F_j(x)/d_j)^2}}.
\]
Since $F_j$ is uniformly bounded on $P$,
\[
 |\pi_{b_j}(x)-\pi_{b_j}(y)|
 \sim d_j^{-1}|F_j(x)-F_j(y)|
 \sim_z\ell_j|F_j(x)-F_j(y)|.
\]
Consequently, for finite $p_i$, the radial reference image of
layer $j$ has lower ball mass
\[
 c_z\min\{1,(u/\ell_j)^{p_i}\}
\]
at every support point, and upper ball mass
$C_z(u/\ell_j)^{p_i}$ at the image of $z_i$.
For infinite $p_i$, its upper bound at that image is
$C_{z,L}(u/\ell_j)^L$ for every $L>0$.

\emph{Radial projections: samples and layer dimensions.}
The preimages of spherical caps belong to the common range
class in Lemma~\ref{lem:relative-clouds}. Hence the same samples
satisfy the radial counterparts of the two projected-sample
estimates in Lemma~\ref{lem:cloud-transfer}. The upper witnesses
are taken at the prescribed vertex $z_i$ itself.

The choice of $v$ and the excluded directions give
$D\pi_z(0)v\ne0$. Thus
\[
 |\pi_z(R_jv)-\pi_z(R_kv)|\gtrsim_z\max\{R_j,R_k\},
 \quad j\ne k\text{ large},\qquad
 \operatorname{diam}\pi_z(T_j\Gamma)\lesssim_z\ell_j.
\]
The image layers are therefore separated; the finitely many
earlier image points are handled as in
Lemma~\ref{lem:cloud-transfer}. Their supports also lie in
$B(\pi_z(0),C_zR_j)$.

Fix $0<\theta\le1$. For finite $p_i$, the lower sample bounds verify
Lemma~\ref{lem:image-layers}, giving the desired upper dimensions.
For finite $a$, put $d=(1/p_i+1/(a\theta))^{-1}$, take
$R_j\sim r^{d/a}$, and use the image of the prescribed vertex
as the centre $y_j$. Scaling the upper sample estimate gives
\[
 \sup_{r\le u\le r^\theta}\frac{\rho_z(B(y_j,u))}{u^s}
 \lesssim_z(1+|\log r|)^C
       \left(r^{d-s}+r^{\theta(d-s)}\right)\longrightarrow0,
 \quad0<s<d.
\]
For exponential counts, choose $0<s<p_i$ and
$0<c<\theta(p_i-s)/p_i$; the bound is
\[
 C_z(1+|\log r|)^C
 \left(e^{-j^2}r^{-s}+r^{\theta(p_i-s)-cp_i}\right)\longrightarrow0.
\]
These estimates give the matching lower dimensions.
For $p_i=\infty$, use the arbitrary-power upper bound at the
vertex and the choices of $c,L$ from the orthogonal argument.
It gives both radial dimensions equal to $a\theta$ for finite $a$,
and infinite for $a=\infty$. This proves every asserted radial
formula for the same measure $\rho$.

Finally,
\[
 \left(\frac1q+\frac1a\right)^{-1}<1
 \quad\Longleftrightarrow\quad
 a<\left(1-\frac1q\right)^{-1}.
\]
Together with \eqref{eq:subtarget-source} and monotonicity in
$\theta$, this proves the last assertion of part~(ii), including
$q=1$ and $q=\infty$ under the reciprocal conventions.
\end{proof}

For finite $a$ and a fixed positive parameter,
$p\mapsto(1/p+1/(a\theta))^{-1}$
maps $[1,q]$ onto
\[
 \left[\frac{a\theta}{1+a\theta},
       \left(\frac1q+\frac1{a\theta}\right)^{-1}\right],
\]
with $1/q=0$ when $q=\infty$; the upper endpoint is then
$a\theta$. For $a=\infty$ the corresponding interval is $[1,q]$,
including the infinite endpoint when $q=\infty$.
These intervals are bounded by the sharp almost-everywhere lower
bound and the source upper bound. Part~(ii) of
Proposition~\ref{prop:nonconstant-directions} realizes any finite
collection of the corresponding functions with prescribed positive
Haar proportions, keeping both entire source functions and the
quasi-Assouad dimension fixed. The radial formulas realize the same
functions on the corresponding open sectors of exterior centres.
For example, with finite $a$, $q=2$ and $p_i=3/2,2$ give
$3a\theta/(3+2a\theta)$ and $2a\theta/(2+a\theta)$; both source
dimensions remain below one when $0<a<2$.
Thus a fixed measure can have different projected dimension functions
on sets of directions of positive Haar measure. In particular, a
direction-independent capacity cannot yield a single almost-everywhere
projected value for every measure.

\paragraph{Further questions.}
Several questions remain. Proposition~\ref{prop:exact-strict-lower}
attains the strict exceptional dimension with one measure for a constant
lower source function at its preservation threshold. It remains to
obtain such attainment while prescribing both entire source functions
and the growth dimension, in particular for the source data and
arbitrary strict thresholds $d<s$ in
Theorem~\ref{thm:finite-graded-sharp}.
Corollary~\ref{cor:nonharmonic-sharp} establishes sharpness when
$b_*<b$ for one prescribed nonharmonic source function. Sharpness of
the optimized exponent for every admissible prescribed source function
remains open.

For finite $a$, the radial strict range $b\le1$, $a\theta\ge b$
and the non-strict range $b<1$, $a\theta>b$ are not covered by
the present constructions. For constant source functions
$s=q<\infty$, Proposition~\ref{prop:radial-constant-reference}
adds $\theta q<d$ and $\theta q\le d$, respectively; the
remaining parameters and the infinite constant-source case
are not settled here. Further questions concern direction
distributions beyond the common harmonic family in
Proposition~\ref{prop:nonconstant-directions}, and local mass
conditions ensuring almost-everywhere constancy.

\appendix
\section{Point-capacity characterization}\label{sec:capacity}

We return to the point capacity \eqref{eq:originalcapacity}.
For every positive decay exponent, its upper and lower asymptotic
rates each determine a unique critical exponent. The next proposition
gives sharp comparisons with the corresponding source dimensions and
a growth condition for equality. It extends
\cite[Theorem~5.7]{AMT}; see also \cite{BFF,AMT}.

\begin{proposition}\label{prop:capacity}
Fix $0<\theta\le1$ and $\alpha>0$, and set
$q=\dim_{\mathrm{qA}}\mu$. The equation
\[
 \limsup_{r\downarrow0}
 \frac{\log C_{r,\theta}^{s,\alpha}(\mu)}{-\log r}=s
\]
has a unique solution $c_+\in[0,\alpha]$. The equation with
$\liminf$ has a unique solution $c_-\in[0,\alpha]$.
Writing $D_+=\overline{\dim}_\theta\mu$ and
$D_-=\underline{\dim}_\theta\mu$, one has
\begin{equation}\label{eq:capacity-sharp-comparison}
 \left[\frac1{D_\pm}
       +\left(\frac1\alpha-\frac1q\right)_+\right]^{-1}
 \le c_\pm\le\min\{D_\pm,\alpha\}.
\end{equation}
In particular, $q\le\alpha$ implies $c_\pm=D_\pm$.
The lower bound in \eqref{eq:capacity-sharp-comparison} is attained
simultaneously at all positive parameters by the layered measures of
Lemma~\ref{lem:general-reference-layers}, with $0<a\le\infty$:
\begin{equation}\label{eq:capacity-layer-values}
 c_+=c_-
 =\left(\frac1{\min\{q,\alpha\}}+\frac1{a\theta}\right)^{-1}.
\end{equation}
\end{proposition}

\begin{proof}
Define
\[
 \begin{aligned}
 h_+(s)&=\limsup_{r\downarrow0}
 \frac{\log C_{r,\theta}^{s,\alpha}(\mu)}{-\log r}-s,\\
 h_-(s)&=\liminf_{r\downarrow0}
 \frac{\log C_{r,\theta}^{s,\alpha}(\mu)}{-\log r}-s.
 \end{aligned}
\]
The kernel comparison \cite[Lemma~5.3]{AMT}, valid for real
$\alpha>0$, gives
\begin{equation}\label{eq:capacity-monotone}
 -(s-t)\le h_\pm(s)-h_\pm(t)\le-\theta(s-t),
 \quad0\le t\le s\le\alpha.
\end{equation}
Compactness and $0<\mu(K)<\infty$ give
$\mu(K)^{-1}\le C_{r,\theta}^{s,\alpha}(\mu)\lesssim r^{-\alpha}$.
Thus both functions are finite, continuous and strictly decreasing,
with $h_\pm(0)\ge0$ and $h_\pm(\alpha)\le0$. Their unique zeros
$c_\pm$ therefore exist without any growth assumption.

For $D_\pm<s\le\alpha$, choose $D_\pm<t<s$.
At every sufficiently small $r$ for $+$, and along a sequence
$r\downarrow0$ for $-$, every $x\in K$ admits $u\in[r,r^\theta]$
with $\mu(B(x,u))\ge cu^t$. Hence
\[
 \begin{aligned}
 \int\phi_{r,\theta}^{s,\alpha}(x-y)\,d\mu(y)
 &\ge cr^su^{t-s}\ge cr^{s+\theta(t-s)},\\
 C_{r,\theta}^{s,\alpha}(\mu)&\lesssim r^{-s-\theta(t-s)}.
 \end{aligned}
\]
It follows that $h_\pm(s)\le\theta(t-s)<0$, proving the upper
bound in \eqref{eq:capacity-sharp-comparison}.

For its lower bound, take $0<s$ strictly below the asserted value.
Write $\dim_\varphi^+\mu=\overline{\dim}_\varphi\mu$ and
$\dim_\varphi^-\mu=\underline{\dim}_\varphi\mu$.
When $q<\infty$, choose
\[
 Q>\max\{q,\alpha\},\qquad 0<t<D_\pm,\qquad
 0<\varphi<\theta
\]
so that $t<\dim_\varphi^\pm\mu$ and
$s<t/(1+t/\alpha-t/Q)$. This is possible by
Lemma~\ref{lem:window}, taking $t\uparrow D_\pm$ and
$Q\downarrow\max\{q,\alpha\}$.
When $q=\infty$, take $Q=\infty$, $\varphi=\theta$, and
$t<D_\pm$ sufficiently large that $s<t/(1+t/\alpha)$.
In either case, small-mass witnesses occur at arbitrarily small
source scales for $+$ and at every sufficiently small scale for $-$.
Lemma~\ref{lem:potential}, with $\kappa=\theta$, gives
\[
 C_{r,\theta}^{s,\alpha}(\mu)\gtrsim r^{-s-\beta},
 \qquad r=\delta^{1+t/\alpha-t/Q},\quad\beta>0.
\]
The change of scale preserves these quantifiers, so
$h_\pm(s)>0$. This proves \eqref{eq:capacity-sharp-comparison}.
More generally, the same argument with any $\kappa\in(0,1]$
shows that $c_\pm$ is bounded below by the right-hand side of
\eqref{eq:two-window-bound}, using the corresponding source dimension.

It remains to prove the equality assertion. For exponential sample
counts, $D_\pm=q$, so \eqref{eq:capacity-sharp-comparison} already gives
$c_\pm=\min\{q,\alpha\}$. Now let $\mu$ be a polynomial layered
measure in Lemma~\ref{lem:general-reference-layers}.
The lower bound just proved gives the right-hand side of
\eqref{eq:capacity-layer-values}. If $\alpha\ge q$, this already
equals the source dimension. Suppose $\alpha<q$, allowing $q=\infty$,
and put $d=\alpha a\theta/(\alpha+a\theta)$.
We show that $h_\pm(s)<0$ for $d<s\le\alpha$.

Write $K_j=T_j(\spt\eta_{N_j})$, $L=1+|\log r|$, and
$A=s+\theta(\alpha-s)$.
If $x\in K_j$ and $R_j\ge r^\theta/8$, then $j\lesssim L$.
The point mass at $x$ is at least $c j^{-2}R_j^a$;
the layers with indices $k\ge j+2$ have mass comparable to $j^{-1}$
and lie at distance at most $2R_j$ from $x$. Consequently,
\[
 \begin{aligned}
 \int\phi_{r,\theta}^{s,\alpha}(x-y)\,d\mu(y)
 &\ge cL^{-2}\max\{R_j^a,r^A R_j^{-\alpha}\}\\
 &\ge cL^{-2}r^{aA/(a+\alpha)}.
 \end{aligned}
\]
If $R_j<r^\theta/8$, all layers with $R_k<r^\theta/8$ lie in
$B(x,r^\theta)$. Their total mass is comparable to $L^{-1}$, giving
\[
 \int\phi_{r,\theta}^{s,\alpha}(x-y)\,d\mu(y)
 \ge cL^{-1}r^{(1-\theta)s}.
\]
This estimate also holds at $x=0$. Since $s>d$, both
$aA/(a+\alpha)$ and $(1-\theta)s$ are strictly smaller than $s$.
The infimum of the potential is therefore at least
$cL^{-2}r^{s-\varepsilon}$ for some $\varepsilon>0$, uniformly
at every small $r$. Hence $h_\pm(s)<0$, proving
\eqref{eq:capacity-layer-values}.
\end{proof}

For every finite $q>\alpha$, choosing
$0<a<\alpha q/(q-\alpha)$ makes all the source dimensions in
\eqref{eq:general-source-values} strictly smaller than $\alpha$,
while both capacity exponents are strictly smaller than the
corresponding source dimensions. For $q=\infty$, the same conclusion
holds with $0<a<\alpha$. Thus the growth threshold in the
characterization cannot be replaced by an upper bound on the source
intermediate dimensions alone. On the other hand, the two-window
bound shows that a source dimension function constant on
$(0,\theta]$ has capacity exponent $\min\{D_\pm,\alpha\}$,
irrespective of its quasi-Assouad dimension.

\paragraph{Capacity exponents for nonharmonic source functions.}
The same estimates also give explicit capacity exponents for the
nonharmonic examples of Section~\ref{sec:third-round}.
For the capped layers of Lemma~\ref{lem:capped-layers},
the two critical exponents are, for every $\alpha>0$ and $0<t\le1$,
\begin{equation}\label{eq:capped-capacity}
 c_+^\alpha(t)=c_-^\alpha(t)
 =\min\left\{e,\left(\frac1{\min\{q,\alpha\}}+\frac1{at}\right)^{-1}\right\}.
\end{equation}
For $a=\infty$, the source function is constant and equal to $e$.
Letting the source window tend to zero gives the lower bound
$\min\{e,\alpha\}$, and the source comparison gives the reverse
inequality. This is \eqref{eq:capped-capacity} since $e<q$.
Now let $a<\infty$. If $q\le\alpha$, the source comparison is an
equality and proves the assertion.
For $\alpha<q$, allowing $q=\infty$, put
$c=\min\left\{e,\left(\frac1{\alpha}+\frac1{at}\right)^{-1}\right\}$ and choose a source parameter
$\kappa\le t$ with $\left(\frac1{q}+\frac1{a\kappa}\right)^{-1}=c$.
Then the source value at $\kappa$ is $c$ and
$\kappa(1/c-1/q)=1/a$.
The two-window lower bound is therefore exactly $c$.
Conversely, the source comparison bounds the zeros by $e$.
Every layer atom has mass at least $c'/N_j$, and the layer
weights and geometry are unchanged, so the atomic-floor and
deeper-tail argument proving \eqref{eq:capacity-layer-values}
also bounds the zeros by $\left(\frac1{\alpha}+\frac1{at}\right)^{-1}$.
This proves the formula at both finite and infinite growth.

Point-capacity zeros take their maximum under finite sums of
measures with separated compact supports, provided the upper and
lower zero agree for each summand.
Indeed, the potential of the sum is at least the summand's
potential at every point of its support.
Conversely, at a minimizing point for one summand, the other
summands contribute only
$O(r^{s+t(\alpha-s)})$, which is of strictly higher order than
$r^s$ when $s<\alpha$.
The strict monotonicity in \eqref{eq:capacity-monotone} proves
the assertion.
Applying this rule to the measures in
Corollary~\ref{cor:nonharmonic-sharp} gives the exact capacity exponents
\begin{equation}\label{eq:nonharmonic-capacity}
 c_+^{\,\alpha}(t)=c_-^{\,\alpha}(t)
 =
 \max\left\{
   \min\left\{\frac12,\left(\frac1{\min\{2,\alpha\}}+\frac1{2t}\right)^{-1}\right\},
   \left(\frac1{\min\{2,\alpha\}}+\frac1{t}\right)^{-1}
 \right\}.
\end{equation}
These capacity exponents attain the optimized two-window lower bound
for every pair $(t,\alpha)$. Indeed, with the common growth exponent
$q=2$, the bound is increasing in the source value and hence commutes
with the finite maximum; optimization over the source parameter also
commutes with this maximum.
For the source functions $D_q$, $1<q\le\infty$, in the remark after
Corollary~\ref{cor:nonharmonic-sharp}, the same calculation proves
\eqref{eq:nonharmonic-capacity} and its optimality with
$\min\{2,\alpha\}$ replaced by $\min\{q,\alpha\}$.

\Needspace{6\baselineskip}
\medskip
\noindent\textbf{Disclosure on artificial intelligence (AI) assistance:}
The authors acknowledge the use of OpenAI tools in the preparation of
this manuscript, including language editing, LaTeX formatting,
organization, proof checks, and the search for sharpness constructions and related examples in Sections~\ref{sec:sharp}--\ref{sec:examples}.
The authors retain full responsibility for the final verification of all mathematical statements and proofs.

\end{document}